\documentclass[a4paper,18pt]{article}
\usepackage{amsmath}
\usepackage{indentfirst}
\usepackage[english]{babel}
\usepackage{graphics}
\usepackage{tabularx}
\usepackage{mathtools}
\usepackage{float}
\usepackage{amssymb}
\usepackage{multirow}
\usepackage{caption}
\usepackage{subcaption}
\usepackage{fancyhdr}
\usepackage{dsfont}
\usepackage{float}
\usepackage{amsfonts}
\usepackage{amsthm}
\usepackage[dvips]{graphicx}
\usepackage[latin1]{inputenc}
\usepackage{enumerate}
\usepackage{amsxtra} 
\usepackage{amstext}
\usepackage{amssymb} 
\usepackage{latexsym}
\usepackage[pdftex,colorlinks]{hyperref}
\usepackage[font=small,labelfont=bf]{caption}   
\newcommand{\N}{\mathbb N} 

\newcommand{\LL}{\mathcal L}
\newcommand{\GL}{\textrm{GL}}
\newcommand{\SL}{\textrm{SL}}
\newcommand{\PV}{\textrm{P$(V)$}}
\newcommand{\PS}{\textrm{P$(V^*)$}} 
\newcommand{\C}{\mathbb C}

\newcommand{\R}{\mathbb R}
\newcommand{\E}{\mathbb E}
\newcommand{\kk}{\mathrm{k}}
\newcommand{\Q}{\mathbb Q}
\newcommand{\p}{\mathbb P}
 \newtheorem{prop}{Proposition}[section]
\newtheorem{lemme}[prop]{Lemma}
\newtheorem{propdef}[prop]{Proposition/Definition}
\newtheorem{corollaire}[prop]{Corollary}
\newtheorem{theo}[prop]{Theorem}

\newtheorem{remarque}[prop]{Remark}

\newcommand{\Z}{\mathbb{Z}}

\providecommand{\keywords}[1]{\textbf{\textit{Keywords: \,}} #1}
\title{{Random matrix products when the top Lyapunov exponent is simple}}
\author{\Large{Aoun Richard}\footnote{
  Richard~Aoun, \textsc{American University of Beirut, Department of Mathematics,  
Faculty of Arts and Sciences, 
 P.O. Box 11-0236 
Riad El Solh,
Beirut 1107 2020, 
LEBANON} 
  \textit{E-mail address}:  \texttt{ra279@aub.edu.lb}
} \,and\, \Large{Guivarc'h Yves}  \footnote{Yves~Guivarc'h, \textsc{UFR math\'{e}matiques 
Universit\'{e} de Rennes 1 
Beaulieu - B\^{a}timent 22 et 23 
263 avenue du G\'{e}n\'{e}ral Leclerc 
35042 Rennes, FRANCE  }   \textit{E-mail address}:  \texttt{yves.guivarch@univ-rennes1.fr}
}}
\date{}
    
\begin{document}
    \setlength{\footnotesep}{\baselineskip}

 \maketitle

 \abstract  

    In the present paper, we treat random matrix products on the general linear group $\GL(V)$, where $V$ is a vector space defined on any local field,   when the top Lyapunov exponent is simple, without irreducibility assumption. In particular, we show the existence and uniqueness of the stationary measure $\nu$ on $\PV$ that is relative to the top Lyapunov exponent and we describe the projective subspace generated by its support. We observe that   the dynamics takes place in a open set of $\PV$ which has the structure of a    skew product space. Then,  we relate this support to the limit set of the semigroup $T_{\mu}$ of  $\GL(V)$ generated by the random walk. Moreover, we show that $\nu$ has H\"older regularity and give some limit theorems concerning the behavior of the random walk and the probability of hitting a hyperplane. These results generalize known ones  when  $T_{\mu}$ acts strongly irreducibly and proximally (i-p to abbreviate) on $V$.  In particular, when applied to the affine group  in the so-called contracting case or more generally when the Zariski closure of $T_{\mu}$ is not necessarily reductive, the H\"older regularity of the stationary measure together with the description of the limit set are new. We mention that we don't use results from the i-p setting; rather we see it as a particular case. 
 
\vspace{0.5cm}\noindent \keywords{Random matrix products, Stationary measures, Lyapunov exponents, Limit sets, Large deviations}

\vspace{0.5cm}\noindent \textbf{\emph{MSC2010}}:   37H15, 60B15, 20P05

\tableofcontents
  \newcommand{\Addresses}{{ 
  \bigskip

  Richard~Aoun, \textsc{American University of Beirut, Department of Mathematics,  
Faculty of Arts and Sciences, 
 P.O. Box 11-0236 
Riad El Solh,
Beirut 1107 2020, 
LEBANON}\par\nopagebreak
  \textit{E-mail address}:  \texttt{ra279@aub.edu.lb}

  \medskip

  Yves~Guivarc'h, \textsc{UFR math\'{e}matiques 
Universit\'{e} de Rennes 1 
Beaulieu - B\^{a}timent 22 et 23 
263 avenue du G\'{e}n\'{e}ral Leclerc 
35042 Rennes, FRANCE  }\par\nopagebreak
  \textit{E-mail address}:  \texttt{yves.guivarch@univ-rennes1.fr}
 
}}

\maketitle

\section{Introduction}
 Let $V$ be a finite dimensional vector space over a local field $\mathrm{k}$ and $\mu$ a probability measure on the general linear group $\GL(V)$.   Random Matrix Products Theory studies the behavior of a random walk on $\GL(V)$ whose increments are taken independently with respect to $\mu$. This theory is well-developed when the sub-semigroup $T_{\mu}$ generated by the support of $\mu$   is strongly irreducible (algebraic assumption) and contains a proximal element (dynamical assumption) \cite{furstenberg}, \cite{bougerol}, \cite{guivarch-raugi}, \cite{BQbook}.  The latter framework, which will be abbreviated by i-p,  had shown to be  a  powerful tool for   understanding     the actions of reductive algebraic groups \cite{Guivarch3}, \cite{BQstationaire}, \cite{aoun-free}, \cite{Breuillard-super-strong}...  One reason is that a great information  on the structure of a reductive algebraic group is encoded in its  irreducible and proximal representations. This setting had also proved its efficiency in the solution to some fundamental problems involving stochastic recursions \cite{kesten1}, \cite{guivarch-page-affine}. \\

 In this article, we extend this theory from the i-p setting to a more general and natural framework. More precisely, we consider a     probability measure $\mu$ on $\GL(V)$ and assume only that its first Lyapunov exponent is simple; in some sense we keep the dynamical condition and assume no     algebraic condition  on the support of $\mu$.  Recall that by a fundamental theorem of Guivarc'h-Raugi \cite{guivarch-raugi}, our setting includes the i-p setting.  But it also includes new settings as  random walks  on the affine group in the called contracting case or more generally  any probability measure on a subgroup $G$ of $\GL(V)$ that may fix    some proper subspace $L$ of $V$ provided the action on $L$ is    less expanding than that on the quotient $V/L$.\\

 Our goal is then to obtain  limit theorems concerning the random walk and the existence, uniqueness and regularity of stationary probability measure on the projective space of $V$.  Our results give also new information  about the limit sets of some non irreducible linear groups. In our proofs we don't use results from the i-p setting but rather see it as a particular case where our assumption concerning the Lyapunov exponent is satisfied. When applied to a probability measure on the affine group in the contracting case, the regularity of the stationary probability measure as well as the description of its support using the limit set of $T_{\mu}$ are new.  
More generally, we show that the dynamics takes place on an open subset of $\PV$ which has essentially the structure of a skew product     space with   basis   a projective space and  fiber  an  affine space.
  We believe that   this generalization  can be useful to treat random walks on non necessarily reductive algebraic groups just as the i-p setting has proved its efficiency.

Here is the structure of the article.
\begin{itemize}
\item  In Section 2 we state formally our results. We note that Section \ref{guiding} shows the geometry behind our results and gives main examples that can be guiding ones through our paper. 
\item Section 3 consists of some preliminary results concerning orthogonality in non-Archimedean local fields and some results on Lyapunov exponents. 

\item  In Section 4, we show the existence and uniqueness of the stationary measure on the projective space whose cocycle average is the top Lyapunov exponent (Theorem \ref{existence_uncite} stated in Section 2). In addition, we describe the projective subspace generated by its support and show that it is not degenerate on it.   \\ 
  The existence appeals to   Oseledets theorem. The uniqueness is explicit: 
  we show in Proposition \ref{propunicite}  that when $\lambda_1>\lambda_2$, every limit point of the  right random walk 
   $(R_n)_{n\in \N^*}$ suitable normalized  is almost surely of rank    one, and the projection of its image   in $\PV$ is a random variable of  law $\nu$.

\item In Section 5, we   make more precise the results of Section 4 by
 relating the support of our unique stationary measure to the limit set of $T_{\mu}$ (Theorem \ref{limitset} stated in   Section 2). 
\item In Section 6, we show the H\"older regularity of the stationary measure (stated in Theorem \ref{th1}). Moreover, we describe an important related large deviation estimate for the hitting probability of a hyperplane (Proposition \ref{hitting}). 
\end{itemize}

  \section*{Acknowledgements}
\noindent Both authors have the pleasure to thank Emmanuel Breuillard for fruitful discussions. 
It is also a pleasure to thank \c Ca$\breve{\textrm{g}}$ri Sert for enlightening discussions
 on the joint  spectral  radius/spectrum.  Part of this project was financed by the  European Research Council, grant no 617129.  RA thanks  also UFR Math\'{e}matiques, Universit\'{e} de Rennes 1 for the facilities given in January 2017. 

 \section{Statement of the results}
 
 \subsection{Uniqueness of the Stationary Measure}\label{uniqueness1}
 
From now on,  $\mathrm{k}$ is a local field of any characteristic, $V$ a finite dimensional vector space
defined over $\mathrm{k}$. Denote by $\PV$ the projective space of $V$. We consider a
probability measure $\mu$ on the general linear group $\GL(V)$ and
denote by $T_{\mu}$ (resp.~$G_{\mu}$) the semigroup (resp.~subgroup) of $\GL(V)$ generated by the support of
$\mu$. We define on the same probabilistic space $(\Omega,
\mathcal{A},\p)$ a sequence $(X_i)_{i\in \N^*}$ of independent
identically distributed random variables of law $\mu$. The right
(resp.~left) random walk  a time $n$ is by definition the random
variable $R_n=X_1\cdots X_n$ (resp.~$L_n=X_n\cdots X_1$).  Endow $V$ with any norm $||\cdot||$ and keep for simplicity the same symbol for the operator norm on $\textrm{End}(V)$. We will
always assume that $\mu$ has a moment of order one, i.e.~$\E(\log^{+}||X_1^{\pm 1}||)<+\infty$ 
and denote by $\lambda_1(\mu)\geq
\cdots \geq \lambda_d(\mu)$ the Lyapunov exponents of $\mu$ defined
recursively by:

$$\lambda_1(\mu)+\cdots + \lambda_i(\mu) = \lim_{n\rightarrow +\infty} \frac{1}{n} \E (\log||\bigwedge^i L_n||)= \lim_{n\rightarrow +\infty} \frac{1}{n} \log||\bigwedge^i L_n||,$$
the last equality is an almost sure equality and is guaranteed by
the subadditive ergodic theorem of Kingman \cite{kingman}. 
In most of the paper, there will no be confusion about the 
probability measure and therefore we will omit specifying $\mu$ when writing the Lyapunov exponents. \\

For every finite dimensional representation   $(\rho,W)$ of
$G_{\mu}$, we denote by  $\lambda(\rho,W)$ the top Lyapunov
exponent relative to the pushforward probability measure
  $\rho(\mu)$ of $\mu$ by the map $\rho$, when the latter has a moment of order one.
  When there is no confusion on the action of $G_{\mu}$ on $W$, we will
  simply denote this exponent by  $\lambda(W)$. To simplify, we will refer to it as \emph{the Lyapunov exponent of $W$}.
   By convention, if $(\rho,W)$ is the null representation, then   $\lambda(\rho,W)=-\infty$.\\
   
   Finally recall that if $T$ is a topological semigroup acting continuously on a topological space $X$ and $\mu$ is a Borel probability measure on $T$, then a Borel probability measure $\nu$ on $X$ is said to be $\mu$-stationary, or $\mu$-invariant,  if for every continuous real  function $f$ defined on $X$, the following equality holds: 
   
   $$\iint_{G\times X}{f(g\cdot x) d\mu(g)d\nu(x)} = \int_X{f( x) d\nu(x)}.$$
 
 \vspace{1cm}
\begin{propdef}Let $\mathcal{W}$ be the set of all $G_{\mu}$-stable
vector subspaces of $V$ ordered by inclusion.
  Let
$$\mathcal{L}_{\mu}:=\sum_{\underset{\lambda(W)<\lambda_1}{W\in \mathcal{W}}}{W}.$$

\noindent Then
$\mathcal{L}_{\mu}$ is a proper $G_{\mu}$-stable subspace of $V$ whose Lyapunov exponent is less that $\lambda_1$, and is the  greatest element of   $\mathcal{W}$ with these properties. 
\label{def_Lmu}\end{propdef}

We will check this  Proposition/Definition in Section \ref{onL} (Lemma \ref{proofpropdef}) and give additional information of the subspace $\mathcal{L}_{\mu}$.

\vspace{0.5cm}

The motivation of this definition comes from the following  result of Furstenberg-Kifer. 
\begin{theo} \cite[Theorem 3.9]{furstenberg-kifer}
  Let $\mu$ be a probability on $\GL(V)$ that  has a moment of order one. 
  Then there exists $r\in \{1, \cdots, d\}$,  a sequence  of   $T_{\mu}$-invariant subspaces
   $(\mathcal{L}_i=\mathcal{L}_i(\mu))_{i=0}^r$ 
  $$ \{0\}=\mathcal{L}_{r} \subset     \mathcal{L}_{r-1}\subset  \cdots \subset 
   \mathcal{L}_1 \subset  \mathcal{L}_{0}=V$$
  and a sequence of real values $\lambda_1(\mu)=\beta^1(\mu) > \beta^2(\mu) > \cdots > \beta^r(\mu)$
  such that if $x\in  \mathcal{L}_{i-1}\setminus  \mathcal{L}_{i}$, then almost surely, 
     $$\lim_{n\rightarrow +\infty} \frac{1}{n} \log||L_n x|| =\beta^{i}(\mu).$$ \label{furstenberg-kifer}\end{theo}

\begin{remarque}
\begin{enumerate}
\item 
It is immediate that the subspace
$\mathcal{L}_{\mu}$ defined in Proposition/Definition \ref{def_Lmu} coincides with the subspace 
$ \mathcal{L}_1(\mu)$ defined in the theorem above.
 Hence we will be using in the rest of article, the following useful equivalence: 
   $$x\not\in \mathcal{L}_{\mu}\,\, \Longleftrightarrow \,\, \textrm{a.s.}\,\lim_{n\rightarrow +\infty} \frac{1}{n} \log||L_n x|| = \lambda_1.$$
 \item 
 Furstenberg and Kifer gave actually an expression of $\lambda_1$ in terms of the ``cocycle average'' of stationary measures.
 More precisely, let $N$ be the set of all $\mu$-stationary measures on $\PV$. For every $\nu\in N$, let
 $$\alpha(\nu)=\iint_{\GL(V)\times \PV}{\log {\frac{||g v||}{||v||}}\, d\mu(g)\,d\nu([v]).}$$
 Then, they showed that 
\begin{enumerate}
\item $\lambda_1=\sup\{\alpha(\nu); \nu \in N\}$. 
\item   $\mathcal{L}_{\mu}=\{0\}$, if and only if,  $\alpha(\nu)$ is the same for all $\nu\in N$ (and hence equal to $\lambda_1$).
\end{enumerate}
\item Note that   the filtration given by Furstenberg and Kifer is deterministic, unlike the one given by Oseledets theorem.  The set $\{\beta_1(\mu), \cdots, \beta_r(\mu) \}$ is included in the Lyapunov spectrum $\{\lambda_1(\mu), \cdots, \lambda_d(\mu)\}$ but the inclusion may be strict. For $x\in V\setminus \{0\}$ fixed, the growth of $||L_n x||$ is almost surely as $\exp\left(n \beta_i(\mu)\right)$ for some $i=1, \cdots, r$.. Hence, the Lyapunov exponents that are distinct from the $\beta_i(\mu)$'s do not characterize the growth of the norm of $||L_n x||$ if we fix first $x$ and then perform a random walk. However they do characterize norm growth if we perform a random walk and choose $x$ in a random subspace of the filtration given by Oseledets theorem.  
 \end{enumerate}
  \label{works_FK}\end{remarque}
  \vspace{0.25cm}

\noindent For every non zero vector $x$ (resp.~ non zero subspace $W$) of $V$, we denote by $[x]$ (resp. $[W]$) its projection on $\PV$.  
 Our  first result describes the stationary measures on $\PV$.  \begin{theo}
  Let $\mu$ be a probability measure on $\GL(V)$ such that
  $\lambda_1>\lambda_2$. Then,
  \begin{enumerate}
 \item[a)]  There exists a  unique   $\mu$-stationary probability measure $\nu$ on $\PV$ which satisfies $\nu([\mathcal{L}_{\mu}])=0$.
  \item[b)]  The projective subspace of $\PV$ generated by the support of $\nu$ is $[\mathcal{U}_{\mu}]$, where  
  $$\mathcal{U}_{\mu}:=\bigcap_{\underset{\lambda(W)=\lambda_1}{W\in \mathcal{W}}}{W}.$$ 
 Moreover,   $\nu$ is non degenerate on $[\mathcal{U}_{\mu}]$
  (i.e.~$\nu$ gives zero mass to every proper projective subspace of
   $[\mathcal{U}_{\mu}]$).
\item[c)] $\left(\PV, \nu\right)$ is a  $\mu$-boundary in the sense of Furstenberg (\cite{Furst73}) , i.e.~there exists a random variable $\omega \mapsto [Z(\omega)]\in \PV$ such that, 
  for $\p:=\mu ^{\otimes \N}$-almost every $\omega:=(g_n)_{n\in \N}\in \GL(V)^{\N}$,   $g_1\cdots g_n \nu$ converges weakly to the Dirac probability measure $\delta_{[Z(\omega)]}$. \end{enumerate}\label{existence_uncite}\end{theo}

   An immediate corollary is the following   
   \begin{corollaire} Keep the notation of Theorem \ref{furstenberg-kifer}. 
   Suppose that for every $i=1, \cdots, r$ 
   the exponent $\beta_i(\mu)$ is simple when seen as a top Lyapunov exponent 
   for the restriction of the random walk to $ \mathcal{L}_{i-1}(\mu)$. 
     Then there are exactly $r$ distinct ergodic $\mu$-stationary measures on $\PV$. 
  \label{spectrum-FK} \end{corollaire}         
    
  \begin{remarque}
  The  assumption of Corollary \ref{spectrum-FK} is equivalent to saying that, for every $i=1, \cdots, r$, 
   $\beta_i(\mu)$ is simple as a top Lyapunov exponent of 
   $\LL_{i-1}(\mu)/\LL_{i}(\mu)$. Hence, by Guivarc'h-Raugi's
    theorem \cite{guivarch-raugi}, 
   a sufficient condition for the finiteness of ergodic $\mu$-stationary measures on $\PV$ 
    is that each quotient $\LL_{i-1}(\mu)/\LL_{i}(\mu)$ is
   strongly irreducible and proximal. Definitely, another sufficient condition is the simplicity of the  Lyapunov spectrum
   , i.e.~ $\lambda_1(\mu)>\cdots > \lambda_d(\mu)$. 
   \end{remarque}

         \begin{remarque}        
 After finishing this paper, it came to our knowledge that Benoist and Bru\`{e}re have   studied   recently and independently the existence and uniqueness of stationary measures on projective spaces over $\R$ in a non irreducible context, in order to study recurrence on affine grassmannians. We will state one of the  main results of the authors, namely \cite[Theorem 1.6]{benoist-bruere}, then discuss the similarities and differences with Theorem \ref{existence_uncite} stated above.   \\
 
\noindent In  \cite[Theorem 1.6 b) ]{benoist-bruere}, the authors consider a real vector space $V$, $G$ 
 Zariski connected algebraic group subgroup $G$ of $\GL(V)$, $W$ a $G$-invariant subspace of $V$ such that $W$ has no complementary $G$-stable subspace,  the action  of $G$ on $W$ and the quotient $V/W$ is i-p and such that the representations
  of $G$ in $W$ and $V/W$ are not equivalent. 
  Then for every probability measure $\mu$ such that $\lambda(V/W)>\lambda(W)$ and whose support is 
  compact and generates a Zariski dense subgroup of   $G$,  the authors show that there exists a unique 
  $\mu$-stationary probability measure $\nu$ on the open set $\PV\setminus [W]$ and that   the   Cesaro mean $\frac{1}{n}\sum_{j=1}^n{\mu^{*j} \star \delta_x}$ converges weakly to $\nu$.   \\

\noindent Theorem \ref{existence_uncite} recovers the aforementioned result.    Indeed, $\mu$ has  a moment of order one since its support is assumed to be compact. The conditions on the Lyapunov exponents imply that   $\mathcal{L}_{\mu}=W$ and    $\lambda_1>\lambda_2$. Moreover, since $W$ has no complementary $G$-stable subspace, then  $\mathcal{U}_{\mu}=V$.    \\

\noindent Theorem \ref{existence_uncite} permits actually  to   relax the i-p assumption on the action on $W$ in the previous statement; only the condition i-p on the quotient and $\lambda(V/W)>\lambda(W)$ is enough.
 Moreover, there is no need for the compactness of the support of $\mu$; a moment of order one is enough. 
 Furthermore, $\mu^{*j} \star \delta_x$ converges weakly to $\nu$ (see Remark  \ref{weakly}), not only in average.  In addition, the vector space $V$ can be defined on any local field $\mathrm{k}$.  \\

 \noindent We note that,  in the rest of the present paper, we will be interested in understanding further properties of this stationary measure. Namely in  Theorem \ref{limitset}  (Section \ref{section_limit_set}) below,  we  describe more precisely the support of $\nu$  in terms of the  the limit   set of $T_{\mu}$ and we  prove its H\"older regularity in Theorem \ref{th1}    (Section \ref{section_regularity}).  \\

\noindent It is worth-mentioning that in \cite[Theorem 1.6 a) ]{benoist-bruere}, the authors  show that when  $\lambda(W)\geq \lambda(V/W)$,  there is no $\mu$-stationary probability measure on $\PV\setminus [W]$ and that the above Cesaro mean converges weakly to zero. This says somehow that $G$-stable subspaces with top Lyapunov exponent guide the dynamics. This information is not disjoint   from the one given by Part b) of Theorem \ref{existence_uncite} saying that the projective subspace generated by the support of $\nu$ is $[\mathcal{U}_{\mu}]$.   \\

\noindent 
The techniques used in the two papers are highly different.  In the present paper we obtain the existence of such a  stationary measure via Oseledets theorem while Benoist and Bru\`{e}re use Banach-Alaoglu theorem and a method developed in \cite{eskin-margulis} for the situation of locally symmetric spaces. Concerning the uniqueness of the stationary measure,   Benoist and Bru\`{e}re's proof is by contradiction via a   beautiful argument of joining measure and previous results on stationary measures on the projective space by Benoist-Quint \cite{bq-pr}.    Here we use methods of \cite{Furst73} and \cite{guivarch-raugi} based on the $\mu$-boundary property. Our method is more explicit as it was described in the introduction (see Propositions \ref{propunicite} and Proposition \ref{corollaire-unicite}). 
  
  \label{remarque_bruere}\end{remarque}

  \subsection{The geometry behind Theorem \ref{existence_uncite} and guiding Examples} 
  \label{guiding}
  \subsubsection{ The geometry behind Theorem \ref{existence_uncite} }\label{geobehind}
 
By Theorem \ref{existence_uncite}, our dynamics takes place in the open dense subset $\PV\setminus [\LL_{\mu}]$ of $\PV$. Here we understand further this dynamics by considering  $\PV$ as a compactification of $\PV\setminus [\LL_{\mu}]$ 
and identifying  topologically $\PV\setminus [\LL_{\mu}]$   with a compact  quotient of the product space $\LL_{\mu}\times S(V/\LL_{\mu})$. We will  check that this product  space is dynamically    a skew product   space with base the unit sphere $S(V/\LL_{\mu})$ and fibers $\LL_{\mu}$ 
seen
 as an affine space. Hence, for each random walk, we are choosing   
 a ``model''  or a ``realization'' of   $\PV$ that depends on the   $G_{\mu}$-stable space $\LL_{\mu}$. 
 This point of view will be used in Section \ref{limit_proof}.  \\
 
Formally,  let $\mathcal{S}_{\kk}$ be the unit sphere of the local field $(\kk, |\cdot|)$,  
  $L$   a proper subspace of $V$ and $G$ a subgroup of $\GL(V)$ that stabilizes $L$.  
 Let $||\cdot||$ be a norm on the quotient $V/L$ such that $(V/L, ||\cdot||)$ is an inner product space (resp.~orthogonalizable) when $\kk$ is Archimedean (resp.~ when $\kk$ is non-Archimedean, see Section \ref{nonarchisection}).  
 Denote 
    by $S(V/L)$ the unit sphere of $V/L$. 
 Fix a supplementary   $\tilde{L}$ of $L$ in $V$. Identifying $V/L$  and $\tilde{L}$ in the usual way, the map $(t,\xi)\in L \times S(V/L)  \longmapsto  [t+\xi]\in \PV\setminus [L]$ yields a homeomorphism between $\PV\setminus [L]$ and the orbit space $X/\mathcal{S}_{\kk}$,  where 
 $X$ is  the product space $$X:= L \times S(V/L)$$
and  $\mathcal{S}_{\kk}$  acts on $X$ in the natural way. 
Using this bijection,  the space $X/\mathcal{S}_{\kk}$ is endowed with a natural structure of $G$-space such that the natural map
 $X/\mathcal{S}_{\kk}\simeq \PV \setminus [L] \overset{\psi}{\longrightarrow} \textrm{P}(V/L)$ is $G$-equivariant. The action of $G$ on $X/\mathcal{S}_{\kk}$ can be lifted to an action of $G$ on $X$ which commutes with the natural 
action of $\mathcal{S}_{\kk}$ on $X$, as we explain hereafter. \\

     Every element $g\in G$ can be written in a basis compatible with the decomposition $V= L \oplus \tilde{L}$ in the 
 form     
    $\begin{pmatrix} 
A & B \\
0 & C 
\end{pmatrix}$ with $A$ (resp. $C$) is a square matrix representing the action of $g$ on $L$ (resp.~ on the quotient vector space $V/L$) and 
$B$ is a rectangular matrix. 
To write down equations properly, one has to make a choice in   
 normalizing non zero vectors in 
$V/L$. We write a polar decomposition of $(V/L)\setminus \{0\}$:  $(V/L)\setminus \{0\}=\R^{*}_{+}\times S(V/L)$ when $\kk=\R$ or $\kk=\C$ and 
$(V/L)\setminus \{0\}= {\varpi}  ^{\Z} \times S(V/L)$ when $\kk$ is non-Archimedean  
 (for a fixed uniformizer   $ \varpi $ and a fixed discrete valuation on $\kk$).
 Let $N: (V/L)\setminus \{0\} \longrightarrow \kk\setminus \{0\}$ such that $N(x)$ is the unique $\R^{*}_{+}$ or $\varpi^{\Z}$-part  of 
    the non zero vector $x$ of $V/L$ in its polar decomposition. In the Archimedean case, one has simply that $N(x)=||x||$.      One can then check that 
 the following formula defines an  action of $G$ on $X=L\times S(V/L)$ that lifts the action of $G$ on $X/\mathcal{S}_{\kk}$ and  commutes with the action of $\mathcal{S}_{\kk}$ on $X$:
  \\
                                    
\begin{equation} \begin{pmatrix} 
A &  B \\
0 & C 
\end{pmatrix}       \cdot (t,\xi) =  \left(\frac{At+ B  \xi}{N(C\xi)}, \frac{C\xi}{N(C\xi)}\right)\label{easybutcrucial}.\end{equation}

 \noindent We observe that the $G$-space $X$ has a skew product structure given by the above formula with base the unit sphere of the vector space $V/L$. Considering $L$ as an affine space,     the fiberwise action is given by affine maps, as for  $g=\begin{pmatrix} 
A &  B \\
0 & C 
\end{pmatrix}  \in G$ and $\xi\in S(V/L)$ fixed,  the map 
 $$\sigma(g,\xi) : t \longrightarrow \frac{At+B\xi}{N(C\xi)}$$
 is an affine transformation of the affine space $L$. Moreover,  the 
  map $\sigma: G \times S(V/L) \longrightarrow \textrm{Aff}(L)$ is a cocycle, i.e.~
  $\sigma(g_1g_2, \xi)=\sigma(g_1,g_2 \cdot \xi) \circ \sigma(g_2, \xi)$  where 
  $g \cdot \xi = \frac{  C \xi }{N(C \xi)}$
  .\\

Let now $\mu$ be a probability measure on $\GL(V)$ whose top Lyapunov exponent is simple and such that  $G=G_{\mu}$ and $L=\mathcal{L}_{\mu}$. 
The $\mu$-random walk on $X$   is then given by the following recursive stochastic equation: 

\begin{equation}t_n= \frac{A_n t_{n-1} + B_n \xi_{n-1}}{N(C_n \xi_{n-1})} \,\,\,,\,\,\, \xi_n=\frac{C_n \xi_{n-1}}{N(C_n \xi_{n-1})} 
\,\label{stochastic}\end{equation}
where 
$\left\{\left( \begin{array}{ccc}
A_n& B_n\\
0&  C_n\  \end{array} \right); n\in \N\right\}$ is a sequence of independent random variables on $\GL(V)$ of same law $\mu$. The result of Theorem \ref{existence_uncite} translates in saying that there exists a unique $\mu$-stationary probability measure $\nu$ on $X/\mathcal{S}_{\kk}$. This measure can be lifted to a probability measure $\tilde{\nu}$   on $X$ which is $\mu$-stationary, $\mathcal{S}_{\kk}$-invariant and unique for these properties. 
Note that the pushforward measure $\psi \star \nu$ of $\nu$ (resp.~$\tilde{\psi} \star \tilde{\nu}$) 
by   the natural map   $\PV\setminus [L] \overset{\psi}{\longrightarrow} \textrm{P}(V/L)$
 (resp.~$X  \overset{\tilde{\psi}}{\longrightarrow}  S(V/L)$) is  also a $\mu$-stationary   probability measure on $\textrm{P}(V/L)$ (resp.~$S(V/L)$).  Since the top Lyapunov exponent of  $\pi(\mu)$, projection of $\mu$ on $\GL(V/L)$,   is also simple and satisfies $\mathcal{L}_{\pi(\mu)}=\{0\}$ (see item 3.~of Remark \ref{V/L}), Theorem \ref{existence_uncite} applies again on $V/L$ and implies that    $\psi \star \nu$  (resp.~$\tilde{\psi} \star \tilde{\nu}$) is the unique $\mu$-stationary probability measure on $\textrm{P}(V/L)$ (resp.~on $S(V/L)$ which is $\mathcal{S}_{\kk}$-invariant). 
We note that when $\mathcal{U}_{\mu}=V$,
 the condition $\lambda_1>\lambda_2$  forces the action on $V/L$ to be
   strongly irreducible and to contain  a proximal element (see Lemma \ref{degenere}). Hence,
    the uniqueness of the probability measure $\psi\star \nu$ on   $\textrm{P}(V/L)$  can be   seen in this case 
    as a corollary of   Guivarc'h-Raugi's work  \cite{guivarch-raugi}  based on techniques developed by Furstenberg \cite{Furst73}. 
  
  \noindent Finally, note that    stochastic recursions similar to \eqref{stochastic}  appeared recently in   \cite[Section 5]{guivarch-page-affine}, with $\dim(\mathcal{L}_{\mu})=1$,    as a crucial tool to   prove the homogeneity at infinity of the measure $\nu$, in the affine situation.

 \subsubsection{Guiding Examples\\}\label{example_guiding}
 
  The guiding examples through this article are the following. The first two (i-p setting and the affine one) are standard and we just check 
that our general framework include them. The third example is an interesting new one that mixes somehow the first two. Together with the simulations of Section \ref{example-simulation}, they   illustrate  our new  geometric setting and the dynamic on it.

\begin{enumerate}
\item \textbf{The irreducible linear groups.} \\
If $T$ is a sub-semigroup of $\GL(V)$ that acts irreducibly on $V$, then   for every probability measure $\mu$ such that $T_{\mu}=T$, we have by irreducibility $\mathcal{L}_{\mu}=\{0\}$ and $\mathcal{U}_{\mu}=V$. By    a theorem of Guivarc'h-Raugi  \cite{guivarch-raugi}, 
the condition $\lambda_1>\lambda_2$ is equivalent to saying that $T$ is i-p 
(strongly irreducible and contains a proximal element).  The results given by    Theorem \ref{existence_uncite}  are known in this case and are due also to Guivarc'h and Raugi in the same paper. With the notation of Section \ref{geobehind}, $X$ is just the unit sphere of $V$ (for a fixed norm).

\vspace{1cm}

\item \textbf{The affine group.} \\
Let $L$ be a hyperplane of $V$ and $T$ a sub-semigroup of $\GL(V)$ 
that stabilizes $L$. Assume for the simplicity that the action on $V/L$ is trivial. Hence, in a suitable basis of $V$, all the elements of  $T$ have a matrix of the form 
                            $\left(\begin{array}{cc}
                              A & \underline{b}  \\
                              0 & 1 \\
                            \end{array}\right)$ with $A$ representing the action on the vector space $L$.          
                                                   The projective space $\PV$ is seen as a compactification of the affine space $L$ with $L$ an affine chart (the action on the base of the product space $X$ is trivial). 
 It will be clear in the following discussion whether $L$ is seen as a subspace of $V$ or as an affine space. \\
    Let $\mu$ be a probability measure on $\GL(V)$ such that $T_{\mu}=T$.                 
                         We denote by  $a_1$ (resp.~$a_2$)   the  top (resp.~second) Lyapunov exponent of the probability measure $A(\mu)$, relative to the linear part of $\mu$.
  Then by Lemma \ref{triangulaire} and Corollary \ref{corollaire-triangulaire} below, the following equalities hold
    $$\lambda_1(\mu)=\max\{a_1,0\}\,\,, \,\,\lambda_2(\mu)= \min\{a_1,\max\{a_2,0\}\}.$$
  The subspaces    $\mathcal{L}_{\mu}$  and $\mathcal{U}_{\mu}$ of $V$ depend  on the measure $\mu$, 
                 unlike the previous example. More precisely,

\begin{enumerate}
\item \underline{Contracting case ($a_1<0$).} In this case, $0=\lambda_1>\lambda_2=a_1$ and 
$\mathcal{L}_{\mu}=L$.  If we assume moreover that $T$ does not fix any proper affine subspaces of $L$, then this translates to the linear action by saying that every $T$-stable vector space of $V$ is included in $L$.    In particular we have $\mathcal{U}_{\mu}=V$.  We can then apply Theorem \ref{existence_uncite}. Its content    
   translates   back to the affine action by saying that there is a unique $\mu$-stationary  probability measure on $L$ and that this measure gives zero mass to  any affine subspace.  This result is well known (see for instance \cite{kesten1}, \cite{bougerol-picard}).

\vspace{0.25cm} \item \underline{Expansive case ($a_1>0$)}: In this case,  $ a_1=\lambda_1$ and $\lambda_2=\max\{a_2, 0\}$. Assume for simplicity that the sub-semigroup $A_T$ generated by $A(\mu)$  acts irreducibly on the vector space $L$. 
Hence the condition $\lambda_1>\lambda_2$ is equivalent to $A_T$ being i-p, which we assume to hold in the sequel. 
Assume moreover that $T$ does not fix a point in $L$. With these assumptions, 
$\mathcal{U}_{\mu}=L$ and $\mathcal{L}_{\mu}=\{0\}$.  In this case,  Theorem \ref{existence_uncite} 
says that  there exists a unique $\mu$-stationary probability measure on the (compactified) affine space $L$ and that it is   concentrated on the hyperplane at infinity. This probability measure corresponds to the unique $A(\mu)$-stationary probability measure on the projective space $\textrm{P}(L)$ of $L$ (we are back to Example 1).

   \end{enumerate}
   
             We note that our results do not apply to the interesting case $a_1=0$, called the critical case. 
   
   \vspace{1cm}
   
   \item \textbf{The Automorphism group of the Heisenberg group.}
   Let $L$ be a one-dimensional subspace of $\R^3$ and   $G$ the group of automorphisms of $V$ that stabilizes $L$. In a suitable basis of $\R^3$, we can identify $G$ with the following matrix group:      
   $$G=\left\{g=\left(\begin{matrix} 
a_g& \underline{b_g}\\
0& C_g
\end{matrix} \right); a_g\in \R\setminus \{0\}; \underline{b}_g\in \R^2; C_g\in \textrm{GL}_2(\R) \right\}\subset \textrm{GL}_3(\R).$$

The group $G$ can be thought of a dual of the affine group on  $\R^2$. In this context, random walks on $G$ appeared naturally in  \cite[Section 5]{guivarch-page-affine} as we have mentioned in the previous section. Also, if one imposes the condition $|a|=\det(g)$  in the definition of $G$, then by letting   the continuous Heisenberg group $\mathcal{H}_3$ act on its Lie algebra,  it can be proved (see \cite{Folland}) that $G$ is isomorphic to the automorphism group of  $\mathcal{H}_3$; the one dimensional fixed subspace of $\R^3$ being the center of  $\mathcal{H}_3$. \\
 
With the notation of Section \ref{geobehind}, $X= \R\times S^1$ and the projective plane $\textrm{P}^2(\R)$ is  seen as a one-point compactification of $X/\{\pm 1\}$.   Recall that by formula \eqref{easybutcrucial}, $X$  has a structure of skew-product space   whose base is a circle 
 and fibers the affine line $L$.   Now let $\mu$ be a probability measure on $G$. Assume that: 
\begin{enumerate}
\item the action of $T_{\mu}$ on $\R^3/L$ is irreducible
\item $\int_{G}{\log{|a_g|}\,d\mu(g)} < \lambda_1(\R^3/L)$.
\end{enumerate}
In this case, $\lambda_1>\lambda_2$, if and only if, the action of $T_{\mu}$ on $\R^3/L$ is strongly irreducible and proximal (i-p) (see Lemma \ref{corollaire-triangulaire}). By the irreducibility of the action on  the quotient,  $\mathcal{L}_{\mu}=L$. Moreover,  
   $\mathcal{U}_{\mu}=\R^3$, if and only if, there does not exist a $G_{\mu}$-invariant decomposition $\R^3=L\oplus W$. With these conditions, the content of Theorem \ref{existence_uncite} is new.  
   The stationary measure given by the aforementioned theorem   projects onto the projective line to
    the  $\mu$-stationary probability measure relative to the i-p semigroup of $\GL_2(\R)$, projection of 
    $T_{\mu}$ on $\R^3/L$.
   We refer  to the simulations of Section \ref{example-simulation}. \\
   
   Note that when   $\int_{G}{\log{|a_g|}\,d\mu(g)} > \lambda_1(\R^3/L)$ and $L$ has no $G_{\mu}$-invariant supplementary in $\R^3$, we have also $\lambda_1>\lambda_2$ but $\LL_{\mu}=\{0\}$.    Theorem \ref{existence_uncite} applies and implies that the unique $\mu$-stationary probability measure on $\textrm{P}^2(\R)$ is $[L]$, i.e.~the point at infinity in $X/\{\pm 1\}$. This   case is similar to the expansive one in the affine situation.

     \end{enumerate}
   \subsection{The support of the stationary measure and Limit Sets }\label{section_limit_set}

 Our next goal will be to  relate the support of the stationary measure $\nu$ obtained above with the limit set of $T_{\mu}$.  We refer to  \cite{gold-guiv}  and \cite{Guivarch3} when such a study is conducted in the strong irreducible and proximal case.  
 \noindent We begin by some notations for a general semigroup $T\subset \GL(V)$ and two $T$-invariant subspaces $L$ and $U$ of $V$ such that $U\not\subset L$.  Denote by  $[g]\in \textrm{PGL}(V)$ the projective map 
   associated to a linear automorphism $g\in \GL(V)$ and by $\textrm{PT}:=\{[g]; g\in T\}\subset \textrm{PGL}(V)$ the projection of $T$ onto $\textrm{PGL}(V)$.   \\
    We will need the notion and some properties of   quasi-projective transformation  introduced by \cite{Furst73} and  developed in \cite{goldsheid-margulis}. Recall that a quasi-projective transformation is a  map from $\PV$ to itself obtained by a pointwise limit of a sequence
    of  projective transformations.  Denote by $\mathcal{Q}$ the set of quasi-projective maps.  

  \begin{itemize}
  \item 
 We denote by $\widehat{T}\subset  \mathcal{Q}$ the set of quasi-projective transformations  $\mathfrak{q}: \PV \longrightarrow \PV$, pointwise limits  of projective maps $[g_n]\in \textrm{PT}$   with the following property:   there exists a proper projective subspace $[W]$ of $\PV$ such that $[U]\not\subset [W]$ and for every $ y \not\in [W]$, $\mathfrak{q}(y)$ is point $p(\mathfrak{q})\in [U]$. Let~
  
 $$\Lambda(T)=\{p(\mathfrak{q}); \mathfrak{q}\in \widehat{T}\}\subset [U].$$
 
We will check in Lemma \ref{quasi} that this is a closed $T$-invariant subset of $\PV$. We will call it the limit set of $T$ (note that it depends on  the subspace $U$). 
     \item
  We consider the $T$-space $O=\PV\setminus [L]$ and we endow it with the   topology induced from that of $\PV$. If $X\subset O$, we denote by $\overline{X}$ its closure in $\PV$ and by $\overline{X}^O$ its closure in $O$. 
 
 Let 
  $\Lambda^a(T)=\Lambda(T)\cap O$ so that $\Lambda^a(T)$ is a closed $T$-invariant subset of $O$.

 \item  Let $T_0$ (resp.~${T^a_0}$) the subset of $T$ which consists of elements $g$ with a simple and unique dominant eigenvalue corresponding to a direction $p^{+}(g)\in [U]$ (resp.~$p^{+}(g)\in [U\setminus L]$).  
     \end{itemize} 
\vspace{0.1cm}

\begin{remarque}
  The choice of the superscript ``a'' in the definition above refers to ``affine'' in line with the description given 
  in Section \ref{geobehind}. 
  Indeed,   suppose that $U=V$, fix a norm on the quotient $V/L$  and let    $g\in T_0^a$. In a suitable basis of $V$,   $g$ can be represented as a matrix    
     $g =\left(\begin{matrix} 
A& B\\
0& C
\end{matrix} \right)\in T$  with $A$ the restriction of $g$ to $L$,  $C$ a  proximal element 
and $\lambda_{\textrm{top}}(g)=\lambda_{\textrm{top}}(C)$, where $\lambda_{\textrm{top}}(\cdot)$ denotes the top eigenvalue. 
Pick a normalized eigenvector $\xi_0$ of $C$. 
Then the point $p^+(g)\in \PV$ can be identified with   $\mathcal{S}_{\kk}(t_0, \xi_0)$ with $t_0\in L$ being the unique fixed point of the affine 
map $t\mapsto \frac{At +B\xi_0}{\lambda_{\textrm{top}}(C)}$ of $L$ (seen as an affine space). 
Note that this affine map is equal to 
the map $\sigma(g,\xi_0)$ introduced in Section \ref{geobehind}, up to an element in $\mathcal{S}_{\kk}$. 
  When $L$ is a hyperplane of $V$ and $C$ is trivial, then 
 $p^+(g)$ represents exactly the  fixed point of the affine map $t\mapsto At +B$ of $L$. \label{explique-notation}\end{remarque}
\vspace{0.1cm}

 \begin{theo} 
Let $T$ be a semigroup of $\GL(V)$, $L$ and $U$ be $T$-invariant subspaces such that $U\not\subset L$. Let  $\mu$ be a probability measure on $\GL(V)$ such that $\lambda_1>\lambda_2$, $T_{\mu}=T$, $\mathcal{L}_{\mu}=L$ and $\mathcal{U}_{\mu}=U$.  
  Let $\nu$ be the unique stationary measure on $\PV\setminus [L]$. Then,

  \begin{enumerate}

 \item  $T_0^{a} \neq \emptyset$ and $\textrm{Supp}(\nu)=\overline{p^{+} (T_0)} = \overline{p^+(T_0^a)}$.

 \item  

$\textrm{Supp}(\nu)=\Lambda(T)$.

 \item
 For    any $[x] \in \PV\setminus [L]$, we have $\Lambda(T) \subset \overline{T\cdot [x]}$. 
 In particular, $\Lambda^a(T)$ is the unique $T$-minimal subset  of $\PV\setminus [L]$.  
 \end{enumerate}
\label{limitset}\end{theo}

We easily deduce the following characterization of the compactness of $\textrm{Supp}(\nu)$ when seen in   the open subset $O=\PV\setminus [L]$ of $\PV$.

  \begin{corollaire}  The following are equivalent: 
  
  \begin{enumerate}
 \item $\textrm{Supp}(\nu)\cap O$ is compact 
 \item $\textrm{Supp}(\nu)$ is a $T$-minimal subset of $\PV$. 
  \item There exists $[x]\in O$ such that   $\overline{T \cdot  [x]}^O$ is compact  
 \item (assume in this part   that $U=V$ and use the notation of Remark \ref{explique-notation})\\
 There exists $c>0$ such that for every 
    $g =\left(\begin{matrix} 
A& B\\
0& C
\end{matrix} \right)\in T_0^a$, one has 
$$||\left(A-\lambda_{\textrm{top}}(C) I\right)^{-1} (B \xi_C)||<c,$$ where  $||\cdot||$ is a fixed norm on $L$, $I$ is the identity matrix, 
$\lambda_{\textrm{top}}(C)\in \kk$ is the top eigenvalue of $C$ and
 $\xi_C$ is any   eigenvector of $C$ corresponding to $\lambda_{\textrm{top}}(C)$ of norm one.

  \end{enumerate} \label{cor-limit}\end{corollaire}

    \begin{remarque} 
\begin{enumerate}

\item  When $\LL_{\mu}=\{0\}$ (as in the i-p case or in the expansive cases of Examples 2 and 3 
in Section \ref{example_guiding}),  it follows from Corollary \ref{cor-limit} that  
$\textrm{Supp}(\nu)$ is the unique $T$-minimal subset of $\PV$. 
When $\LL_{\mu}\neq \{0\}$ (as the contracting cases of Examples 2 and 3) and 
$\textrm{Supp}(\nu)\cap O$ is compact, then  
the latter is a $T$-minimal subset of $\PV$ but never the unique such one as $[L]$ is a compact $T$-invariant subset of $\PV$
 that does not intersect 
$\textrm{Supp}(\nu)$. In particular, $\textrm{Supp}(\nu)$ is the unique $T$-minimal subset of $\PV$, if and only if, 
  $\LL_{\mu}=\{0\}$. 
\item It follows from Theorem \ref{limitset} that the support of $\nu$  depends only on $T$ and not on $\mu$ (provided  $T_{\mu}=T$, $\lambda_1(\mu)>\lambda_2(\mu)$ and $\mathcal{U}_{\mu}=U$, in which case $L=\mathcal{L}_{\mu}$ is uniquely determined as $U\not\subset L$).

\item We assume $U=V$ and adopt the notation of Remark \ref{explique-notation}. 
It follows from item 3 of Theorem \ref{limitset} that a sufficient condition for the non compactness of 
 $\textrm{Supp}(\nu)$ in $O$ is the existence of at least one proximal element $g\in T$ with an attracting direction $p^+(g)\in [L]$. We will check in Lemma \ref{additional} that proximality is not needed, i.e.~ if there exists  $g=   \left(\begin{matrix} 
 A& B\\
 0& C
 \end{matrix} \right) \in T$ with $\rho_{\textrm{spec}}(A)>\rho_{\textrm{spec}}(C)$, then $\textrm{Supp}(\nu)\cap O$ is not compact. 
 For the situation where $T$ is  a non degenerate  semigroup of the affine transformations of the real line
  in the contracting case, this boils down to the well-known fact 
    that
  the support of the unique stationary measure $\nu$ on the affine line is non compact when there exists at least one transformation $x\mapsto ax +b$ with $|a|>1$.

\item We continue the previous remark. 
 It is easy to see 
   that if $\textrm{Supp}(\mu)$ is a bounded subset of affinities of the real line such that $|a|<1$ for every $x\mapsto ax+b$ in 
   $\textrm{Supp}(\mu)$, then the support of the unique stationary measure on the real line is compact
    (the well-known example of Bernoulli convolutions fits in this category, see Remark \ref{convolution} ). 
In our situation,  having   $\rho_{\textrm{spec}}(A)<\rho_{\textrm{spec}}(C)$ for every  $g=   \left(\begin{matrix} 
 A& B\\
 0& C
 \end{matrix} \right)$ in $T$ is not sufficient to insure the compactness of the support of $\nu$ in $O$. We refer to Example 3.~ of Section \ref{example-simulation}. 
 
 \item We give in Section \ref{compactness-criterion} a sufficient condition 
 for the compactness of the support of $\nu$ in 
 $O$ in the case   $\textrm{dim}(\LL_{\mu})=1$ (see Example 3 of Section \ref{example_guiding}) 
 using the notion of joint spectral radius and the geometric setting of Section \ref{geobehind}.  Note that the 
 joint spectral radius is known to play a role in the existence of an 
  attractor to   affine iterated functions systems (IFS) and that projective IFS are gaining  a lot of importance recently 
 (see for instance \cite[Section 5]{barnsley1}). 
  
 \item  
It is definitely interesting to conduct a study concerning the tail of $\nu$ when the latter 
is not compact in $O$. 
We refer to   \cite{guivarch-page-affine1} for the case of the affine line.    
 
\end{enumerate}
 \label{remark_support}\end{remarque} 
  
  \subsection{Regularity of the stationary measure}\label{section_regularity}
The following result shows that the unique stationary measure  $\nu$ given
  by Theorem \ref{existence_uncite} has H\"older regularity when $\mu$ has an exponential moment, i.e.~when  $\int_{\GL(V)}{||g^{\pm 1}||^{\tau}\,d\mu(g)}<+\infty$ for some
   $\tau>0$.  We denote by $\delta$ the Fubini-Study metric on the projective space $\PV$ (see Definition \ref{fubini-study}).  We recall that  
 the projective subspace of $\PV$ generated by $\nu$ is $[\mathcal{U}_{\mu}]$ and that $\nu$ is non degenerate on it. 
 Hence, the following result gives a precision of that fact. 

 \begin{theo}
 Let $\mu$ be a probability measure on $\GL(V)$ such that
  $\lambda_1>\lambda_2$.  
  If $\mu$ has an exponential moment, then
  there exists $\alpha>0$ such that
$$\sup_{\textrm{$H$ hyperplane of $\mathcal{U}_{\mu}$}} \int{\delta^{-\alpha}([x],[H])\,d\nu([x])}<+\infty.$$

          \label{th1}\end{theo} 
    
 \begin{remarque}
 \begin{enumerate}
 \item 
    We note that we will give a slightly more general statement in Theorem \ref{regularite}
     involving the distance to any projective hyperplane of $\PV$.  
\item  Assume that  $\mathcal{U}_{\mu}$ is not a one dimensional subspace of $V$ (otherwise $\nu$ is a Dirac probability measure). 
Theorem \ref{th1} implies then, through Markov's inequality,  that $\nu$ is  $\alpha$-H\"older, i.e.~ there exists $D>0$ such that for every $\epsilon>0$, and for $\nu$-almost every $[x]\in [\mathcal{U}_{\mu}]$, 
     $\nu\left(B([x],\epsilon)\right)<  D \epsilon^{\alpha}$, where $B(\cdot, \cdot)$ denotes the open ball in the metric space $\left([\mathcal{U}_{\mu}], \delta\right)$.    In particular, the Hausdorff dimension of $\nu$ is greater or equal to $\alpha$. 
     \end{enumerate}
     \end{remarque}    
\noindent In the i-p case, Theorem \ref{th1} is known and is due to  Guivarc'h \cite{Guivarch3}. When applied to the affine group it is new. More precisely, 

\begin{corollaire}
Let $\mu$ be a probability measure on the group of affinities of
an affine space $L$ whose support does not fix any proper affine
subspace. Assume that the Lyapunov exponent of the linear part of
$\mu$ is negative (contracting case). Then the unique
$\mu$-stationary probability measure $\nu$ on $L$   has   a positive Hausdorff dimension. 
 \label{dimaffine}\end{corollaire}

\begin{remarque} We note that the problem of 
estimation of the Hausdorff dimension of $\nu$ was initially considered by Erd\"{o}s (see for instance \cite{sixty}) if $T\subset \textrm{Aff}(\R)$ preserves an interval of the line. It  led recently to deep results in similar situations (see \cite{hochman}, \cite{breuillard-varju} for example).  In   the more general situation of this paper, we get only qualitative results on the dimension of $\nu$. \label{convolution}\end{remarque}
 
One of the important estimates in random matrix products theory is the probability of return of the random walk to hyperplanes. It is well studied in the i-p case and leads to fundamental spectral gap results \cite{bourgain_gamburd}, \cite{breuillard_gamburd}, \cite{saxce}, \cite{breuillard_note}.... The general setting studied in this paper leads to new estimates in this direction.  \begin{prop}
Let $V$ a finite dimensional vector space and  $\mu$ be a
probability measure on $\GL(V)$ with an exponential moment such
that $\lambda_1>\lambda_2$.
  Then, for every  $\epsilon>0$, there exist $\beta=\beta(\epsilon)>0$, $n_0=n_0(\epsilon)\in \N$ such that for every $n\geq n_0$,  $x\in V\setminus \mathcal{L}_{\mu}$ and every $f\in V^*\setminus \LL_{\check{\mu}}$, 
      
\begin{equation}   \p \left[ \delta\left(L_n[x], [\textrm{Ker}(f)]\right)\leq \exp(-\epsilon n)
 \right]\leq   \frac{\exp\left(-n \beta \right)}{\delta([x], [\LL_{\mu}])\, \delta([f], [\LL_{\check{\mu}}])}.\nonumber\end{equation}
 In this statement $V^*$ denotes the dual space of $V$ and  $\check{\mu}$ is  the pushforward probability measure of $\mu$ by the map $g \in \GL(V)\longmapsto g^t\in \GL(V^*)$. 
 \label{hitting}\end{prop}

\section{Preliminaries}\label{preliminaries}
\subsection{Linear algebra preliminaries}
Our proofs rely on  suitable choice of norms on our vector spaces and on the 
expression of the distance between a point   and a projective subspace of $\PV$ (Lemma \ref{dist_droite_sev} below). 
  For the convenience of the reader, we recall in Section \ref{nonarchisection} basic facts about orthogonality in non-Archimedean vector spaces
  (c.f.~ \cite{nonarchi} for instance). 
The reader interested only in vector spaces over Archimedean fields can check directly Section \ref{fubinisection}.

\subsubsection{Non-Archimedean orthogonality}\label{nonarchisection}

Let $(\mathrm{k},|\cdot|)$ be a non-Archimedean local field. 
We denote by $\mathcal{O}_\mathrm{k}=\{x\in \mathrm{k}; |x|\leq 1\}$ its ring of integers and $\mathcal{O}_{\kk}^{\times}=\{x\in \kk; |x|=1\}$ the group of 
 units of $\mathcal{O}_{\kk}$. Let 
$V$ be  a   vector space over $\kk$ of dimension $d\in \N^*$ and $B_0=(e_1, \cdots, e_d)$   a fixed 
 basis of $V$.   We consider the following norm on $V$:
  $$||x||:=\max\{|x_i|; i=1, \cdots, d\}$$
 where  the $x_i$'s are the coordinates of the vector $x$ in the basis $B_0$.  Every such  finite 
 dimensional normed vector space over a non-Archimedean local field will said to be orthogonalizable.  \\
 
 We   say that two subspaces 
$E$ and $F$ of $V$ are  {orthogonal} when $||v+w||=\max\{||v||, ||w||\}$ for every 
$v\in E$ and $w\in F$.  A family of vectors $(v_1,\cdots, v_r)$ in $V$ is said to be orthogonal if for every $\alpha_1,\cdots, \alpha_r\in \kk$,
 $||\alpha_1 v_1+\cdots +\alpha_r v_r||=\max \{|\alpha_1 |||v_1||, \cdots, |\alpha_r| ||v_r||\}$. \\

  We recall that   $\GL_d(\mathcal{O}_\mathrm{k})$ is 
 the subgroup of the general linear group $\GL_d(\mathrm{k})$ formed by the  matrices  $g$ such that       $g$ and $g^{-1}$  have coefficients in $\mathcal{O}_\mathrm{k}$; which is equivalent to impose that $g$ has coefficients in $\mathcal{O}_\mathrm{k}$ and that 
  $\det(g)\in \mathcal{O}_\mathrm{k}^{\times}$. One can show that 
  $\GL_d(\mathcal{O}_\mathrm{k})$ is a maximal compact subgroup of $\GL_d(\mathrm{k})$.
  The following lemma gives crucial results of orthogonality in non-Archimedean vector spaces   
  similar to the classical  ones  in the Archimedean setting.
 
 \begin{lemme}
 
\begin{enumerate}
\item  For every basis $B=(v_1,\cdots, v_d)$ of $V$, the following statements are equivalent:

 \begin{itemize}
 \item[i.] $B$ is orthonormal
   \item[ii.]  The  transition matrix from $B_0$ to $B$ belongs to $\GL_d(\mathcal{O}_\mathrm{k})$
   \item[iii.] $B$ is a basis of the   $\mathcal{O}_\mathrm{k}$-module  $\mathcal{O}_\mathrm{k}e_1\oplus \cdots \oplus \mathcal{O}_\mathrm{k} e_d\simeq \mathcal{O}_\mathrm{k}^d$.\end{itemize}
\item  Every subspace $E$ of $V$ has an orthonormal basis and admits an  orthogonal complement $E^{\perp}$.
 \end{enumerate}\label{Gram-Schmidt}\end{lemme}

\begin{proof}
Without loss of generality,  $V=\kk^d$ and $B_0$ the canonical basis.
  One can easily show that $\GL_d(\mathcal{O}_\mathrm{k})$  is  the isometry group of  $(V,||\cdot||)$.

\begin{enumerate}

\item  The equivalence between items i., ii.~and iii.~is an easy consequence of the fact that $\GL_d(\mathcal{O}_\mathrm{k})$ acts by isometries on $V$.
\item[2.,3.]  Let $r$ be the dimension of $E$ as a $\mathrm{k}$-vector space,  $M=\mathcal{O}_\mathrm{k}^d$ and $E'=E\cap M$. Then $M$ is a free  $\mathcal{O}_\mathrm{k}$-module of rank $d$ and $E'$ is a submodule.  Since $\mathrm{k}$ is a local field, then $\mathcal{O}_\mathrm{k}$ is a Principal Ideal Domain (PID). Then the structure theorem of modules over PID's gives a basis $B=(v_1,\cdots, v_n)$ of $M$, $r\in \N^*$ and scalars $d_1,\cdots, d_k\in \mathcal{O}_\mathrm{k}$ such that $(d_1v_1,\cdots, d_r v_r)$ is a basis of $E'$ as a $\mathcal{O}_\mathrm{k}$-module.  The set $B$ is clearly also a basis of the $\mathrm{k}$-vector space $V$, $r$ the dimension of $E$ as $\mathrm{k}$-vector space  and $(v_1,\cdots, v_r)$ a basis of the subspace $E$ of $V$.  By the equivalence between 1.i.~ and 1.iii.,  $B$ is orthonormal. Hence items 2 and 3 follow immediately.  \end{enumerate}
\end{proof}

\begin{remarque}
Unlike the Archimedean case, a subspace   may 
have more than one orthogonal complement in $V$. Indeed, consider $\kk=\Q_2$, $V=\kk^2$ and   the one dimensional subspaces 
 $E$, $E_1$ and $E_2$ of $V$ generated respectively by $(1,0)$, $(0,1)$ 
and $(1,1)$. Then  $E_1$ and $E_2$ are two distinct orthogonal complements of $E$ because the identity matrix and the matrix  $\left(
     \begin{array}{cc}
                                                                       1& 1 \\
                                                                       0 & 1 \\
                                                                     \end{array}
                                                                   \right)$ belong to $\textrm{SL}_2(\Z_2)$.                                                         
                                                                               \label{nonunicity}\end{remarque}

                                                                               \begin{remarque} Let $E$ be a subspace of $V$ and $||\cdot||$ the quotient norm in $V/E$, i.e.~ for every
                                                                               $\overline{x}\in V/E$, 
                                                                               $$||\overline{x}||:=\inf\{||x+y||; y\in E\}.$$
                                                                             One can   easily show  that for any orthogonal supplementary $E^{\perp}$ of $E$ in $V$,
                                                                              and for every $\overline{x}\in V/E$, the following holds:
                                                                              $$||\overline{x}||=||\pi_{E^{\perp}}(x)||.$$
                                                                              Here $\pi_{E^{\perp}}$ denotes the projection onto $E^{\perp}$ with kernel $E$. 
                                                                              \label{quotient_norm}\end{remarque}

\subsubsection{The Fubini-Study metric} \label{fubinisection}
Now $(\mathrm{k},|\cdot|)$ is a local field and $V$ a   vector space over $\mathrm{k}$ of 
dimension $d\geq 2$ and $B_0$ a fixed basis of $V$. 
When $\kk$ is Archimedean, we endow $V$ with the canonical norm $||\cdot||$ for which $(V,||\cdot||)$ is an inner product space and 
$B_0$ is an  orthonormal basis. When $\kk$ is non-Archimedean,  we endow $V$ with the norm described in the previous section. \\

\noindent We consider the norm on $\bigwedge^2 V$, which will be denoted  also by 
$||\cdot||$,  such that $(e_i\wedge e_j)_{1\leq i<j \leq d}$ is an orthonormal basis of $\bigwedge^2 V$.

\begin{propdef}(Fubini-Study metric) \\
Let $(V,||\cdot||)$ as above and $\PV$ the projective space of $V$. 
\begin{enumerate}
\item For every $[x],[y]\in \PV$, we set:
$$\delta([x],[y]):=\frac{||x \wedge y||}{||x|| ||y||}.$$
Then $\delta$ defines a metric on $\PV$, called the Fubini-Study metric (see for instance \cite[Prop. 2.8.18]{bombieri}).\\
\item For every subset $Y$ of $\PV$ and $[x]\in \PV$, let
$$\delta([x],Y)=\inf_{[y]\in Y}{\delta([x],[y])}.$$\end{enumerate}\label{fubini-study}\end{propdef}

\begin{remarque} The following are easy facts. 
\begin{enumerate}
\item 
When $\kk$ is non-Archimedean, $\delta$ is actually ultrametric. 
\item The metric $\delta$ is bounded by one. 
\item If $x$ and $y$ are orthogonal, then $\delta([x],[y])=1$. \end{enumerate}
\end{remarque}

The following lemma will be fundamental for us. 
 
\begin{lemme}
Let $\mathrm{k}$ be a local  field and $(V,||\cdot||)$ as above. 
Let $E$ be a subspace of $V$ 
and $E^{\perp}$ an orthogonal complement (see Lemma \ref{Gram-Schmidt} when $\mathrm{k}$ is non-Archimedean). 
We denote by $\pi_{E^{\perp}}$ the  
 projection onto $E^{\perp}$ with kernel $E$.  
By abuse of notation, we denote also by $||\cdot||$
  the quotient norm on $V/E$ (see Remark \ref{quotient_norm}). 
  Then, for every non zero vector $x$ of $V$,

\begin{equation}\delta([x],[E])=\frac{||\pi_{E^{\perp}}(x)||}{||x||}=\frac{||\overline{x}||}{||x||}.\label{normes}
\end{equation}
 
 \label{dist_droite_sev}\end{lemme}

\begin{proof}By Remark \ref{quotient_norm}, it is remaining to prove   the left equality only.    Let $[x]\in \PV$. WLOG $x\not\in E$. 
We write $x=x_1+x_2$, with $0\neq x_1\in E$ and $\pi_{E^{\perp}}(x)= x_2\in  E^{\perp}$.
On the one hand, 
$$\delta([x],[E])\leq \delta([x],[x_1])= \frac{||x \wedge x_1||}{||x||\,||x_1||} =  \frac{||x_2 \wedge x_1||}{||x||\,||x_1||}\leq \frac{||x_2||}{||x||}.$$
This proves that $\delta([x],[E]) \leq  \frac{||\pi_{E^{\perp}}(x)||}{||x||}$. 
 
\noindent On the other hand, let $B$ be an orthonormal basis of $V$ obtained by concatenating a orthonormal basis, 
say $(v_1,\cdots, v_r)$ of $E$  and  an orthonormal basis, say $(v_{r+1},\cdots, v_d)$, of $E^{\perp}$
 (see Lemma \ref{Gram-Schmidt} when $\mathrm{k}$ is non-Archimedean). Let   $y\in E\setminus\{0\}$. 
  By writing $x_1,x_2$ and $y$ in the basis $B$, we see  that
$x_1\wedge y$ belongs to subspace of $\bigwedge^2 V$ generated by
 $(v_i\wedge v_j)_{1\leq i< j \leq r}$ and $x_2\wedge y$ to the one generated by $(v_i\wedge v_j)_{(i,j)\in  \{1,\cdots, r\}\times \{r+1,\cdots, d\}}$. 
The basis $(v_i\wedge v_j)_{1\leq i<j \leq d}$ is also orthogonal in $\bigwedge^2 V$. Hence  $||x \wedge y||= ||x_1\wedge y + x_2\wedge y||\geq 
||x_2\wedge y||$. Since $x_2$ and $y$ are orthogonal in $V$, we have 
   that $||x_2\wedge y||=||x_2||\,||y||$. Hence, for every $y\in E\setminus\{0\}$, $||x \wedge y||\geq ||x_2||\,||y||$. Hence,
$$\delta([x],[E])=\inf_{[y]\in [E]}{\frac{||x \wedge y||}{||x||\,||y||}}\geq \frac{||x_2||}{||x||}=\frac{||\pi_{E^{\perp}}(x)||}{||x||}.$$
The left equality of \eqref{normes} is proved.

\end{proof}

 \subsection{Preliminaries on Lyapunov exponents}

In Lemma \ref{corollaire-triangulaire}, we recall a   crucial result due to Furstenberg-Kifer  that  reduces the computation of the top Lyapunov exponent
 of a random walk on a group of upper triangular block matrices to the top Lyapunov exponents of
  the random walks induced on the diagonal parts. For the reader's convenience, we include a proof.
The, we  deduce     Corollary \ref{corollaire-triangulaire} which shows  that all the other Lyapunov exponents of
$\mu$ can be also read on the diagonal part with the right multiplicity.
\begin{lemme}\cite[Lemma 3.6]{furstenberg-kifer}, \cite{bougerol} \\
Let $\mathrm{k}$ be a local field, $V$ a finite dimensional vector space
defined over $\mathrm{k}$,    $\mu$ be a probability on $\GL(V)$ having a
moment of order one. Consider a $G_{\mu}$-invariant subspace $W$
of $V$. Then the first Lyapunov exponent $\lambda_1$ of $\mu$ is
given by:
 $$\lambda_1= \max \{\lambda_1(W), \lambda_1(V/W)\}.$$\label{triangulaire}\end{lemme}

 \begin{proof} Since we deal here only  with only top Lyapunov exponents, we will omit the subscript $1$ in the notation. 
  
 Without loss of generality, all the elements of $G_{\mu}$ are represented by  $d\times d$ invertible  matrices of the form $\left(\begin{array}{cc}
                A & B \\
                0 & C
              \end{array}\right)$ where $A$ represents the action of $G_{\mu}$ on $W$ and  $C$  the action on the quotient $V/W$.   We use the canonical norm on $V$ and the associated operator norm on $\textrm{End}(V)$. Only     the inequality  $\lambda  \leq  \tilde{\lambda}:= \max \{\lambda(W), \lambda (V/W)\}$ requires a proof.  
   For every $n\in \N^*$, write       $L_n=\left(\begin{array}{cc}
                                                                                                                           A_n & B_n \\
                                                                                                                           0 & C_n \\
                                                                                                                         \end{array} \right)$ for the left random walk at time $n$. Denote by $(L'_n)_{n\in \N^*}$   the sequence of random variables defined by $L'_n:=X_{2n}\cdots X_{n+1}$ so that $L_{2n}=L'_n L_n$.  Writing $L'_n=\left(\begin{array}{cc}
                                                                                                                           A'_n & B'_n \\
                                                                                                                           0 & C'_n \\
                                                                                                                         \end{array}\right)$, we have that:                                                                                                                         
                                                                                                                                                          \begin{equation}B_{2n}=A'_nB_n+B'_nC_n.\label{eq11}\end{equation}                                                               
 Fix for now 
     $\epsilon>0$ and $\eta \in (0,1)$. 
We know that the sequences of random variables  $\left(\frac{1}{n} \log ||L_n||\right)_{n\in \N^*}$ $\left(\frac{1}{n} \log ||A_n||\right)_{n\in \N^*}$ and $\left(\frac{1}{n} \log ||C_n||\right)_{n\in \N^*}$ converge in probability respectively to $\lambda$, $\lambda(W)$ and $\lambda(V/W)$.
Moreover, $A'_n$ (resp. $B'_n$) has the same law as $A_n$ (resp. $B_n$) for every $n\in \N^*$. We deduce that there exists   $n_0=n_0(\epsilon, \eta)$, such that for every $n\geq n_0$, all the following four real number are greater than $1-\eta$: 
 $\p\left( ||B_n||\leq e^{n \lambda  +n\epsilon}\right)$,  $\p\left( ||C_n||\leq e^{n \lambda(V/W)+n\epsilon}\right)$,  $\p\left( ||A'_n||\leq e^{n \lambda(W)+n\epsilon}\right)$ and  $ \p\left( ||B'_n||\leq e^{n \lambda  +n\epsilon}\right)$.       Using now identity \eqref{eq11} and the inequality $\tilde{\lambda}\leq \lambda$, we deduce that for $n\geq n_0$, 
\begin{equation}\p \left( ||B_{2n}||\leq 2 e^{n (\lambda + \tilde{\lambda})+ 2 n \epsilon}\right)\geq 1-4\eta.\label{opo1}\end{equation}
  But since $2\lambda(W)\leq \lambda+\widetilde{\lambda}$ and $2\lambda(V/W)\leq \lambda +\widetilde{\lambda}$, we obtain two other   estimates similar to  \eqref{opo1} by replacing $B_{2n}$ with $A_{2n}$ and $C_{2n}$ respectively (and  taking again $n_0$ bigger if necessary). Hence,  for every $n\geq n_0$,  
 $$\p\left( ||L_{2n}||\leq  4 e^{n (\lambda +\tilde{\lambda} )+ 2  n \epsilon}\right)\geq 1-6\eta.$$
 But  by the convergence of  $\left(\frac{1}{n} \log ||L_n||\right)_{n\in \N^*}$ in probability to $\lambda$, we can impose that for $n\geq n_0$,  
  $\p\left( ||L_{2n}||\geq  e^{2 n\lambda - n  \epsilon}\right)\geq 1-\eta$, so that for $n\geq n_0$, 
   $$\p \left(e^{2n\lambda - n\epsilon} \leq ||L_{2n}|| \leq 4 e^{n (\lambda + \widetilde{\lambda}) + 2n \epsilon} \right)\geq 1- 7 \eta.$$
  Choosing any $\eta\in (0,\frac{1}{7})$, and letting $n\rightarrow +\infty$ and then $\epsilon \rightarrow 0$, we get that   $\lambda \leq \tilde{\lambda}$. 
      \end{proof}

\noindent
 
\begin{corollaire}
Consider the same situation as in the previous lemma.  Denote by
$S_1$ (resp.~$S_2$) the set of Lyapunov exponents  associated to
the probability measure induced on $W$ (resp.~$V/W$). Then the set
of Lyapunov exponents associated to $\mu$  is $S_1\cup S_2$. Also the multiplicity of an exponent for the random walk in $\GL(V)$ is the sum
 of its multiplicity as an exponent  for the restricted random walk in 
  $\GL(W)$ (if any) and as an exponent for the random walk in  $\GL(V/W)$  (if any).  
\label{corollaire-triangulaire}\end{corollaire}

 \begin{proof}
  First note that if $E$ and $F$ are two $G_{\mu}$-invariant finite dimensional vector spaces, then $\lambda_1(\bigwedge^2 E)=\lambda_1(E)+\lambda_2(E)$ and $\lambda_1(E\otimes F)=\lambda_1(E)+\lambda_1(F)$. Let now $\widetilde{W}$ be a supplementary of $W$ in $V$.  
  Let $k\in \{2,\cdots, d\}$.  The following decomposition holds   
    $$\bigwedge^k V = \underset{\underset{i+j=k}{0\leq i,j\leq k}}{\bigoplus}{\left(\bigwedge^i W \otimes \bigwedge^{j}  \widetilde{W}\right)}.$$
    For every $p\in \{0,\cdots k\}$, let 
    $$F_p:=  \underset{\underset{i+j=k}{0\leq j\leq p}}{\bigoplus}{\left(\bigwedge^i W \otimes \bigwedge^{j}  \widetilde{W}\right)}.$$  This is  a $G_{\mu}$-invariant subspace of $\bigwedge^k V$ 
    and   the quotient 
    $F_{p}/F_{p-1}$ is isomorphic as $G_{\mu}$-representation to  $\bigwedge^{k-p}W\,\otimes\, \bigwedge^{p}(V/W)$ (with the convention $F_{-1}=\{0\}$). Since 
    $\{0\}=F_{-1}\subseteq F_0\subseteq \cdots \subseteq F_{k-1} \subseteq F_k=\bigwedge^k V$ is a filtration of $\bigwedge^k  V$, 
we apply    Lemma \ref{triangulaire}  at most $k+1$ times and use the observations at the beginning of the proof in order to get the following identity:

 \begin{equation}\lambda_1+\cdots + \lambda_k = \max \{\lambda_1(W)+\cdots + \lambda_{k-p}(W) + \lambda_1(V/W)+\cdots + \lambda_p(V/W);\, p=0, \cdots, k\}.\label{wallayalla}
    \end{equation}
 In the previous equation, we used      the convention $\lambda_i(W)=-\infty$ (resp.~$\lambda_i(V/W)=-\infty$) if $i$ exceeds the dimension of $W$ (resp.~$V/W$). Note that for $k=1$,  \eqref{wallayalla} boils  down to Theorem \ref{triangulaire}. Let $m_1$ be    the multiplicity of the top Lyapunov exponent $\lambda_1$ (as an exponent in $\GL(V)$).  Applying  \eqref{wallayalla} for $k=1, \cdots, 1+m_1$   gives two informations: 
 first that $\lambda_2$ is the second largest number in the set $S_1\cup S_2$ and second that the multiplicity of $\lambda_1$ in $\GL(V)$ is the sum of its multiplicity as an exponent in $\GL(W)$ and in $\GL(V/W)$. Recursively, one shows the desired property for the all the other Lyapunov exponents. 
       \end{proof}

\subsection{On the subspaces $\mathcal{L}_{\mu}$ and $\mathcal{U}_{\mu}$}\label{onL}
Let  $\mu$ be a probability measure on $\GL(V)$. In Definition \ref{def_Lmu}, we introduced the following subspace of $V$:  
  
$$\mathcal{L}_{\mu}:=\sum_{\underset{\lambda(W)<\lambda_1}{W\in \mathcal{W}}}{W}. $$
In the statement of Theorem  \ref{existence_uncite}, we introduced the following subspace of $V$: 
$$\mathcal{U}_{\mu}:=\bigcap_{\underset{\lambda(W)=\lambda_1}{W\in \mathcal{W}}}{W}.$$
In this section, we state some useful properties of these subspaces that follow immediately from their definition. 
 
  \begin{lemme}
$\mathcal{L}_{\mu}$ is a proper $G_{\mu}$-stable subspace of $V$ whose Lyapunov exponent is less that $\lambda_1$, and is the  greatest element of   $\mathcal{W}$ with these properties. \\
 When $\lambda_1>\lambda_2$, $\mathcal{U}_{\mu}\not\subset \mathcal{L}_{\mu}$. In
 particular, $\mathcal{U}_{\mu}$ is non zero in this case and is    the
 smallest $G_{\mu}$-subspace whose Lyapunov exponent is
 $\lambda_1$.\label{proofpropdef}\end{lemme}

\begin{proof}  
  The subspace $\mathcal{L}_{\mu}$ has   the claimed property because on the one hand the sum that defines it can be made a finite one and on the other hand if $W_1$ and $W_2$ are two $G_{\mu}$-stable subspaces of $V$, then one can easily prove   that $\lambda(W_1+W_2)=\max\{\lambda(W_1),\lambda(W_2)\}$.  Assume now that $\lambda_1>\lambda_2$ and consider two $G_{\mu}$-stable subspaces $W_1$ and $W_2$  of $V$ such that $\lambda(W_1)=\lambda(W_2)=\lambda_1$. We will prove that $\lambda(W_1\cap W_2)=\lambda_1$; and the claim concerning $\mathcal{U}_{\mu}$ will immediately follow. Indeed, assume that $\lambda(W_1\cap W_2)< \lambda_1$. Then by    Lemma \ref{triangulaire} and Corollary   \ref{corollaire-triangulaire}, we deduce  that     the top Lyapunov exponent of $E:=V/W_1\cap W_2$ is simple and is equal to $\lambda_1$.  The same holds for the subspaces $W_1/W_1\cap W_2$ and $W_2/W_1\cap W_2$ of $E$. By simplicity of $\lambda_1$ in $E$, we deduce that  $(W_1/W_1\cap W_2) \cap  (W_2/W_1\cap W_2)\neq \{0\}$, contradiction.  \end{proof}
 
   The following  easy lemma will be crucial for us. 
   For every $g\in \GL(V)$, we denote by $g^t\in \GL(V^*)$ the transpose linear map on the dual $V^*$ of $V$, i.e.~ $(g^t f) (x)=f(g x)$ for every $g\in \GL(V)$, $f\in V^*$ and $x\in V$. For every subspace $W$ of $V$, we denote by $W^0\subseteq V^*$ its annihilator, i.e.~ $W^0=\{f\in V^*; f_{|_{W}}=0\}$.  
      
      \begin{lemme}(Duality between $\mathcal{L}_{\mu}$ and $\mathcal{U}_{\mu}$)\\
   Let $\mu$ be a probability
 measure on $\GL(V)$ such that  $\lambda_1>\lambda_2$. Denote by
 $\check{\mu}$ the probability measure on $\GL(V^*)$ defined as the law of $X_1^t $, where $X_1$ has law $\mu$.
  Then,
 $$\mathcal{L}_{\mu}^{0} = \mathcal{U}_{\check{\mu}}\,\,,\,\,\mathcal{U}_{\mu}^{0} =
 \mathcal{L}_{\check{\mu}}.$$
\label{duality}\end{lemme}

\begin{proof}If $W\subseteq V$ is a $G_{\mu}$-stable subspace of $V$, then $W^{0}$ is a $G_{\check{\mu}}$-stable subspace of $V^*$ which is  isomorphic as $G_{\check{\mu}}$-space to $(V/W)^*$.  Hence 
$$\lambda_1(W^{0}, \check{\mu}) = \lambda_1\left((V/W)^*, \check{\mu}\right)=\lambda_1(V/W, \mu).$$
Hence, 
 by Lemma \ref{triangulaire}, 
 $\lambda_1=\max \{\lambda_1(W, \mu), \lambda_1(W^{0}, \check{\mu})\}$. 
 Using Corollary \ref{corollaire-triangulaire}, we deduce that when $\lambda_1>\lambda_2$,   one and only one of  the numbers $\lambda_1(W, \mu)$ and $\lambda_1(W^{0}, \check{\mu})$ is equal to $\lambda_1$. Also, we deduce that $W^{0}$ contains a $G_{\check{\mu}}$-stable subspace of $V^*$ of $\check{\mu}$-Lyapunov exponent is equal to $\lambda_1$, if and only, $W$ is included in a $G_{\mu}$-stable subspace whose $\mu$-Lyapunov exponent is less than
  $\lambda_1$. Applying the previous remarks for $W:=\mathcal{L}_{\mu}$, we get that  $\mathcal{L}_{\mu}^{0}$ is a $G_{\check{\mu}}$-stable subspace of $V^*$ whose Lyapunov exponent for $\check{\mu}$ is equal to $\lambda_1$ and is the smallest such subspace.   Since $\mu$ and $\check{\mu}$ have the same Lyapunov exponents,  Lemma \ref{proofpropdef} yields  the identity  $\mathcal{L}_{\mu}^{0} = \mathcal{U}_{\check{\mu}}$. The equality $\mathcal{U}_{\mu}^{0} =
 \mathcal{L}_{\check{\mu}}$ follows also.   \end{proof}
 
 During the proofs, we will frequently go back to the case where $\mathcal{U}_{\mu}$ is the whole space $V$. We refer  to three   guiding examples of Section \ref{guiding} where this condition was always satisfied, thanks to a ``natural'' geometric condition  imposed at each time. 
The following lemma reformulates this condition in different ways. \begin{lemme}
Assume that $\lambda_1>\lambda_2$. The following properties
 are equivalent:

 \begin{enumerate}
 \item $\mathcal{U}_{\mu}=V$. 
   \item For every $G_{\mu}$-stable
 proper subspace $W$ of $V$,  $\lambda(W)<\lambda_1$
 \item $\mathcal{L}_{\mu}$ is the  greatest element of $\mathcal{W}\setminus \{V\}$, i.e.~every   $G_{\mu}$-stable subspace of $V$ is
 either $V$ or is included in $\mathcal{L}_{\mu}$. 
 \item
 $\mathcal{L}_{\check{\mu}}=\{0\}$.
  \end{enumerate}
 
\noindent  Moreover, when one of these conditions is fulfilled, the action of $T_{\mu}$ on the quotient $V/{\mathcal{L}_{\mu}}$ is strongly irreducible and proximal. \label{degenere}\end{lemme}

 \begin{proof} The equivalence between (1), (2), (3) and (4) is easy to prove by definition of $\mathcal{L}_{\mu}$ and $\mathcal{U}_{\mu}$, and by Lemmas \ref{proofpropdef} and \ref{duality}. We prove now the last statement. Assume  that (3) holds. It follows that the action of $T_{\mu}$ on the quotient $V/\mathcal{L}_{\mu}$ is irreducible. But by  Lemma \ref{triangulaire} and Corollary   \ref{corollaire-triangulaire},  the top Lyapunov exponent of $V/\mathcal{L}_{\mu}$ is simple. It is enough now to recall the following known result   from \cite{guivarch-raugi} (see also \cite[Theorem 6.1]{bougerol}):  if  $E$ is a vector space defined over a local field and  $\eta$ is a probability measure on $\GL(E)$  such that $T_{\eta}$ is irreducible, then $T_{\eta}$ is i-p if and only if the top Lyapunov exponent relative to  $\eta$ is simple. This ends the proof.  \end{proof}

   \begin{remarque} If $\rho: G_{\mu}\longrightarrow \GL(\mathcal{U}_{\mu})$ is the restriction map to $\mathcal{U}_{\mu}$, then it is easy to see that $\mathcal{U}_{\rho(\mu)}=\mathcal{U}_{\mu}$ and that $\mathcal{L}_{\rho(\mu)}=\mathcal{L}_{\mu} \cap \mathcal{U}_{\mu}$. Observe also that it follows from Lemma \ref{degenere} that the action of $T_{\mu}$ on $\mathcal{U}_{\mu}/\mathcal{L}_{\mu}\cap \mathcal{U}_{\mu}$ is strongly irreducible and proximal. We will frequently use  the representation  $\rho$ to go back to the case $\mathcal{U}_{\mu}=V$. \label{restU}\end{remarque}

  \vspace{0.2cm}
 \begin{remarque}
 \begin{enumerate}
 \item 
Another case for which estimates are easier to   handle is the case  $\mathcal{L}_{\mu}=\{0\}$ (i.e.~$\mathcal{U}_{\check{\mu}}=V^*$). This condition  appeared in \cite[Proposition 4.1, Theorem B]{furstenberg-kifer} (see also \cite{hennion}) as   a sufficient condition  to ensure the continuity of the function $\mu \mapsto \lambda(\mu)$.  Moreover, it corresponds to a unique cocycle average (see Remark \ref{works_FK}). Recall that by Section \ref{guiding} this condition is satisfied for random walks in irreducible groups and in the affine group in the expansive case.  However we  insist
 on the fact that one of the novelty of the present paper is to give limit theorems,  when $\lambda_1>\lambda_2$, in  the case $\mathcal{L}_{\mu}\neq \{0\}$ (as for instance random walks on the affine group in the contracting case, see Section \ref{guiding}).  We refer also to \cite{bqtcl} where limit theorems   for cocycles are given depending on their cocycle average(s).

 \item Note that if $\lambda_1>\lambda_2$, then it follows from Lemmas \ref{degenere} and \ref{duality} that the following statements are equivalent: 
 
 \begin{enumerate}
 \item $\mathcal{L}_{\mu}=\{0\}$. 
   \item For every $G_{\mu}$-stable
 proper subspace $W$ of $V$,  $\lambda(W)=\lambda_1$
 \item Every $G_{\mu}$-stable proper subspace  of $V$ contains $\mathcal{U}_{\mu}$.  \item
 $\mathcal{U}_{\check{\mu}}=V^*$.
  \end{enumerate}

\item If $\pi: G_{\mu}\longrightarrow \GL(V/\mathcal{L}_{\mu})$ is the morphism of the projection onto 
  $V/\mathcal{L}_{\mu}$, then $\mathcal{L}_{\pi(\mu)}=\{0\}$ and $\mathcal{U}_{\pi(\mu)}=\pi(\mathcal{U}_{\mu})$. Observe also that if $\lambda_1>\lambda_2$, then by  Corollary \ref{corollaire-triangulaire} the top Lyapunov exponent of $V/\mathcal{L}_{\mu}$ is equal to $\lambda_1$ and is also simple. 
 \end{enumerate}
 
  \label{V/L}\end{remarque}

\section{Stationary probability measures on the projective space}

 In this section, we prove Theorem \ref{existence_uncite}.  This will be done  through different steps. 
 In Section \ref{stationnaire1} below, we show that if a stationary measure $\nu$ on $\PV$ such that $\nu([\mathcal{L}_{\mu}])=0$ exists,  then this determines the projective subspace generated by its support. In Section \ref{stationnaire2}, we show the existence of such a measure via Oseledets theorem. In  Section \ref{stationnaire3} we prove that it is unique in a constructive way. More precisely, we show in Proposition \ref{propunicite}  that $\nu$ is the law of a random variable $[Z(\omega)]\in \PV$ characterized in the following way: every limit point of the right    random walk  $(R_n)_{n\in \N^*}$ suitably normalized   is almost surely of rank one with image that projects to $[Z(\omega)]$ in $\PV$.
\\
 
 We recall that $\mathrm{k}$ is a local field, $V$ is a vector space over $\mathrm{k}$ of dimension $d\geq 2$ and $\PV$ denotes the projective subspace of $V$.
 We endow $V$ with the norm $||\cdot||$ described in Section \ref{fubinisection}.  If $\mu$ is a probability measure on $\GL(V)$, then $T_{\mu}$ (resp.~$G_{\mu}$) denotes the sub-semigroup (resp.~subgroup) of $\GL(V)$ generated by the support of $\mu$. We denote by  $\mathcal{W}$  
  the set of all $G_{\mu}$-stable subspaces of $V$ and   for every $W\in \mathcal{W}$, $\lambda(W)$ denotes the Lyapunov exponent relative to $W$. \\
   For every $g\in \GL(V)$, we denote by $g^t\in \GL(V^*)$ its transpose map. We denote by $W^0\subseteq V^*$ the annihilator of a subspace $W$ of $V$.

\subsection{On the support   of stationary probability measures}\label{stationnaire1}

 \begin{prop}
Let $\mu$ be a probability measure on  $\GL(V)$ such that
$\lambda_1>\lambda_2$ and $\nu$ a stationary probability measure
of the projective space $\PV$ such that  $\nu([\mathcal{L}_{\mu}])=0$. Let 
$$\mathcal{U}_{\mu}:=\bigcap_{\underset{\lambda(W)=\lambda_1}{W\in \mathcal{W}}}{W}.$$
Then,

\begin{enumerate}
\item 
The projective subspace generated by the support of $\nu$ is $[\mathcal{U}_{\mu}]$.

\item The
 probability measure $\nu$ is non degenerate in  $[\mathcal{U}_{\mu}]$
 i.e.~its gives zero mass to every proper projective subspace of
 $[\mathcal{U}_{\mu}]$.

\end{enumerate} \label{propstat}\end{prop}

The proof of this proposition will be done through different
intermediate steps. First, we give below  a criterion insuring
that a stationary measure on the projective space is non degenerate. When
$G_{\mu}$ is strongly irreducible, Furstenberg   has shown that
every $\mu$-stationary probability measure   on the projective
space is non degenerate. The proof of Furstenberg yields in fact the
following general result. It will be used in  Lemma \ref{reunion}
in order to identify non degenerate stationary measures outside the
strongly irreducible case.

 \begin{lemme}
Let $E$ be a finite dimension vector space, $\mu$   a probability
measure on $\GL(E)$ and $\nu$ a $\mu$-stationary probability
measure on the projective space $\textrm{P}(E)$ of $E$.  Then there  
exists a projective subspace of $\textrm{P}(E)$ whose $\nu$-measure is non zero,
of minimal dimension and whose $G_{\mu}$-orbit is finite.
Equivalently, there exists a finite index subgroup $G_0$ of
$G_{\mu}$ such that at least one of the projective subspaces of $\textrm{P}(E)$ charged
by $\nu$ is stable under $G_0$.
 \label{support2} \end{lemme}

\begin{proof}   Let   $\Lambda$ be the set of projective subspaces of $\textrm{P}(E)$ charged by $\nu$ and of minimal dimension, say $l$.
Let  $r=\sup\{\nu([W]); [W]\in \Lambda\}$.      By
minimality of $l$, two distinct subspaces  $[W_1]$ and $[W_2]$ of
$\Lambda$ satisfy $\nu([W_1\cap W_2])=0$. Since $\nu$ is of total mass $1$, we deduce that there are only   finitely many subspaces $[W]\in \Lambda$ such that $\nu([W])\geq \frac{r}{2}$. In particular, $r=\max \{\nu([W]); [W]\in \Lambda\}$. Consider then the following non-empty finite set: 
  $\Gamma:=\{[W]\in \Lambda; \nu([W])=r\}$.    We claim that $\Gamma$ is  stable under
$G_{\mu}$, which is sufficient to show the desired lemma. 
Indeed, since $\nu$ is a $\mu$-stationary probability measure, then for every
 $[W]\in \Gamma$  and $n\in
\N$:
\begin{equation} r=\nu([W])= \iint {\mathds{1}_{[W]} (g\cdot [x]) d\mu(g)\;d\nu([x])} = \int {\nu (g^{-1}\cdot [W])
\;d\mu^n(g)}.\label{classique}
\end{equation}
Let $b$ be any probability measure on $\N$ with full support.  By replacing if
necessary $\mu$ by  $\sum_{i=1}^{+\infty} {b(i)\mu^i}$  in the
equality above, we can assume  without loss of generality  that the
support of $\mu$ is  the semigroup  $T_{\mu}:=\cup_{n\in
\N}{\textrm{Supp}(\mu^n)}$.  By combining
this remark, together with equality  \eqref{classique}  and the
maximality of $r$, we obtain that
$$\forall g\in T_{\mu},\, \nu(g^{-1}\cdot
[W])=r\,\,\,\,\,\,\textrm{i.e.}\,\,\,\,\,\, g^{-1}\cdot[W]\in
\Gamma.$$
Hence for every $g\in T_{\mu}$, $g^{-1} \Gamma \subset \Gamma$.  Since $\Gamma$ is finite, we deduce that for every $g\in T_{\mu}$, $g \Gamma =  \Gamma$.  It follows that    $\Gamma$ is
$T_{\mu}$-stable (or equivalently $G_{\mu}$-stable).\end{proof}

 \vspace{0.5cm}

We know that when  $\lambda_1>\lambda_2$,   $G_{\mu}$ is
irreducible if and only if $G_{\mu}$ is strongly irreducible (see  \cite[Theorem 6.1]{bougerol}).
Here's below a generalization.
 
 \begin{lemme} 
 Let $E$ be a finite dimension vector space and $\mu$ a probability measure on $\GL(E)$ such that
 $\lambda_1>\lambda_2$.
 \begin{enumerate}
 \item  If $\mathcal{L}_{\mu}=\{0\}$, then $G_{\mu}$ cannot fix any finite union of
 non zero   subspaces of $E$ unless   they all contain   $\mathcal{U}_{\mu}$. 
 \item Dually, if $\mathcal{U}_{\mu}=E$, then $G_{\mu}$ cannot fix a finite union of proper subspaces of  $E$ unless they are all contained in $\mathcal{L}_{\mu}$. In particular, if $\mathcal{U}_{\mu}=E$,  then a  $\mu$-stationary probability measure $\nu$ on
 $\textrm{P}(E)$ is non degenerate if and only if $\nu([\mathcal{L}_{\mu}])=0$.
\end{enumerate}
  \label{reunion}\end{lemme}

 \begin{proof}
 It is enough to show  statement 1. Indeed, the first part of   statement 2 is actually equivalent to the first one  by passing to the dual $E^*$ of $E$, thanks to Lemma \ref{duality}, the fact that $\mu$ and $\check{\mu}$ have the same Lyapunov exponents and finally to the   fact that $G_{\mu}$ stabilizes a finite union $\{V_1, \cdots, V_r\}$ of subspaces of $E$ if and only if $G_{\check{\mu}}$ stabilizes  $\{V_1^0, \cdots, V_r^0\}$ in $E^*$. The last part of the second statement is a consequence of the first part of the same statement and   of    Lemma  \ref{support2}. Now we prove the first statement. Arguing by contradiction, we let $r$ to be the integer  in $\{1,\cdots, d-1\}$ defined as the minimal dimension of a non zero subspace $V$ of $E$ such that $\mathcal{U}_{\mu}\not\subset V$ and such that $V$ belongs to some finite $G_{\mu}$-invariant set of subspaces of $E$. Let $V$ be   such a subspace of $E$ with dimension $r$ and $L:=\{g V; g\in G_{\mu}\}$ be the orbit of $V$ under $G_{\mu}$. This is a finite $G_{\mu}$-invariant set of subspaces of $E$ all having the same dimension $r$ and all not containing   $\mathcal{U}_{\mu}$ (as the latter   is a $G_{\mu}$-invariant subspace of $E$). Moreover, the     cardinality $s$ of $L$ is greater or equal to $2$ because the assumption $\mathcal{L}_{\mu}=\{0\}$ implies that any   $G_{\mu}$-invariant subspace of $V$ contains $\mathcal{U}_{\mu}$ (see item 2.~ of Remark \ref{V/L}, dual of Lemma \ref{degenere}).   Let then $L:=\{V_1, \cdots, V_s\}$ with the $V_i$'s
  pairwise distinct and consider the following 
 non empty set below:
$$\Gamma:=\{V_i\cap V_j; 1\leq i <j \leq s\}.$$
It is immediate that $\Gamma$ is a finite $G_{\mu}$-invariant set of subspaces of
 $E$, all of them not containing $\mathcal{U}_{\mu}$ and of dimension $<r$. By
minimality of  $r$, we deduce that
   \begin{equation}\forall i\neq j, \,V_i \cap V_j =\{0\}\label{1}\end{equation}
   In particular, the projective subspaces $[V_i]$'s of $\textrm{P}(E)$ are disjoint, 
   so that we can define the following positive real number: 
$$\alpha:=\inf\{\delta([{V_i}], [{V_j}]);
1\leq i<j\leq s\}>0.$$
 Let  $x\in V_1\setminus \{0\}$ and $y\in V_2\setminus \{0\}$. For every $g\in G_{\mu}$, there exist $ i=i(g) ,  j= j(g) \in \{1,\cdots, s\}$ such that  $g x\in
V_i$ and $g y\in V_j$. We claim that       $i\neq j$ for every $g\in G_{\mu}$. Indeed, if
$i=j$, then by denoting by  $k$ the unique integer such that
$g^{-1}V_i=V_k$, we would have  $x\in V_1\cap V_k$ and $y\in V_2\cap
V_k$. This contradicts \eqref{1}.  We deduce that 
\begin{equation}\forall g\in G_{\mu},\, \delta \left(g [x], g[y]\right)\geq \alpha.\label{min1}\end{equation}
But since $\mathcal{L}_{\mu}=\{0\}$  and since $\lambda_1>\lambda_2$, we have by  \cite[Theorem 3.9]{furstenberg-kifer} (see Theorem \ref{furstenberg-kifer})  
 that: 
 $$\delta\left(L_n [x], L_n [y]\right)\leq \frac{||\bigwedge ^2  L_n||\, ||x \wedge y||}{||L_n x ||\, ||L_n y||} \underset{n\rightarrow + \infty}{\overset{\textrm{a.s.}}{\longrightarrow} }0, $$
which contradicts  \eqref{min1}.
   \end{proof}
 
 \vspace{0.5cm}

\begin{proof}[Proof of  Proposition \ref{propstat}]

First, we prove that $[\mathcal{U}_{\mu}]$ contains the projective subspace $S$ of $\PV$ generated by the support of $\nu$. Let  $E$ be a $G_{\mu}$-stable subspace of   $V$ such that $\lambda(E)=\lambda_1$. We want to show that $\nu([E])=1$. First, we check that   for every $[x]\in \PV\setminus [\LL_{\mu}]$, 
the following almost sure convergence holds:

\begin{equation}\delta(L_n[x], [E])\underset{n\rightarrow +\infty}{\overset{\textrm{a.s.}}{\longrightarrow}} 0.\label{eq-1}\end{equation}
Indeed, consider the quotient norm on $V/E$. 
By Lemma \ref{dist_droite_sev},    the following holds for every $[x]\in \PV$:  
\begin{equation}\delta([x], [E])= \frac{||\overline{x}||}{||x||}.\label{eq0}\end{equation}
But since $\lambda(E)=\lambda_1>\lambda_2$, Lemma \ref{corollaire-triangulaire} implies that $\lambda(V/E)<\lambda_1$. 
Hence, 
  \begin{equation}\forall x \in V, \,\, \textrm{a.s.}, \,\,\limsup\frac{1}{n} \log || \overline{L_n x}|| < \lambda_1.\label{eq1}\end{equation}
Combining  \eqref{eq0},  \eqref{eq1} and     Theorem \ref{furstenberg-kifer} gives, for any $[x]\in \PV\setminus [\LL_{\mu}]$, the almost sure convergence 
  \eqref{eq-1}.\\
  Let now $\epsilon>0$.  Since $\nu$ is $\mu$-stationary, we have for every $n\in \N^*$, 
    \begin{equation}\nu\{[x]\in \PV; \delta([x],[E])>\epsilon\} = \int_{\PV}{\p\left(\delta(L_n[x], [E])>\epsilon \right)\,d\nu([x])}.\label{eq-2}\end{equation}
Since $\nu([\LL_{\mu}])=0$,  \eqref{eq-1} holds for $\nu$-almost every $[x]\in \PV$. In particular, for $\nu$-almost every $[x]\in \PV$, the following holds $\p(\delta(L_n[x], [E])>\epsilon)\underset{n\rightarrow +\infty}{\longrightarrow} 0$. By Fubini's theorem and  \eqref{eq-2}, we deduce that $\nu([x]\in \PV; \delta([x], [E])>\epsilon )=0$. This being true for every $\epsilon>0$, we deduce that $\nu([E])=1$. 
  This being true for every such stable subspace $E$, and since the intersection defining $\mathcal{U}_{\mu}$ can be made a finite one (the dimension of $V$ is finite), we deduce that $\nu([\mathcal{U}_{\mu}])=1$. Since $[\mathcal{U}_{\mu}]$ is closed in $\PV$, we deduce that   $S\subset [\mathcal{U}_{\mu}]$.\\
  In order to prove the other inclusion,  write $S=[E]$ for some subspace $E$ of $V$.  Recall that $\textrm{Supp}(\nu)$ is $T_{\mu}$-invariant, i.e.~  
\begin{equation}\forall g\in T_{\mu}, g \cdot \textrm{Supp}(\nu) \subset \textrm{Supp}(\nu).\label{walla}\end{equation}
It follows from  \eqref{walla} that $E$ is a $G_{\mu}$-invariant subspace of $V$. Moreover,  since $\nu([\LL_{\mu}])=0$, Theorem \ref{furstenberg-kifer} implies that the Lyapunov exponent relative to $E$ is $\lambda_1$.  By definition of $\mathcal{U}_{\mu}$, we deduce that $\mathcal{U}_{\mu}\subset E$ and then that $[\mathcal{U}_{\mu}]\subset S$.    Item (1) of the proposition is then proved. \\

\noindent In order to prove point (2) of the proposition, we set for simplicity of notation   $E=\mathcal{U}_{\mu}$ and
denote by $\rho$ the restricted representation $G_{\mu}
\longrightarrow \GL(E)$. It follows from above  that   $\nu$ is a
$\rho(\mu)$-stationary probability measure on $\textrm{P}(E)$. By
definition of $E$, we have the following equalities:
 $$\lambda(E)=\lambda_1\,\,\,,\,\,\,\mathcal{L}_{\rho(\mu)}=\mathcal{L}_{\mu} \cap E\,\,\,,\,\,\, \mathcal{U}_{\rho(\mu)} = E.$$
By Lemma  \ref{degenere}, the first and the third equalities above show
that the probability measure  $\rho(\mu)$ on $\GL(E)$ satisfies
the assumptions of Lemma \ref{reunion}. Since $\nu([\mathcal{L}_{\mu}])=0$,
the second equality above gives $\nu([\mathcal{L}_{\rho(\mu)}])=0$. By Lemma
\ref{reunion} again, $\nu$ is non degenerate on $\textrm{P}(E)$.
\end{proof}

\subsection{Oseledets theorem and stationary measures}\label{stationnaire2}

In this section, we prove that given a probability measure $\mu$
 on $\GL(V)$ such that
$\lambda_1>\lambda_2$, there exists  a $\mu$-stationary
probability measure $\nu$ on the projective space $\PV$ that
satisfies the equality $\nu([\mathcal{L}_{\mu}])=0$ and the conclusions of
Proposition \ref{propstat}. Our proof is constructive: we use
Oseledets theorem to derive  a random variable $[Z]\in \PV$ of law $\nu$
from the random walk associated to
$\mu$.  Since $\lambda_1>\lambda_2$, such a stationary measure will
immediately be a $\mu$-boundary. \\

\noindent We note that the existence of such a probability measure holds even if $\lambda_1=\lambda_2$. This can be proved using the  methods developed in  \cite{furstenberg-kifer}. Since the framework of the latter article is very general, the method is not constructive.\\

\begin{prop}
Let  $\mu$ be a probability measure on $\GL(V)$ such that
$\lambda_1>\lambda_2$. Then, there exists a  $\mu$-stationary
probability measure $\nu$ on $\PV$ such that $\nu([\mathcal{L}_{\mu}])=0$. By
Proposition  \ref{propstat}, $\nu$ is non degenerate on  $[\mathcal{U}_{\mu}]$.
Moreover, $\left( \PV\setminus [\mathcal{L}_{\mu}], \nu \right)$ is a
$\mu$-boundary. \label{propexistence}\end{prop}

\noindent 
Such a measure will be obtained thanks to Oseledets theorem, and
more precisely the equivariance equality we recall below.

    \begin{theo}\cite{oseledets}
Let $(\Omega,\theta,\p)$ be an ergodic dynamical system. Let  $A:
\Omega \longrightarrow \GL(V)$ be a measurable application such
that   $\log||A||$ and $\log||A^{-1}||$ are integrable. Then there
exist $l\in \N^*$, $m_1, \cdots, m_l\in \N^*$ and real numbers
$\lambda_1=\cdots = \lambda_{m_1} > \cdots > \lambda_{m_{l-1}+1}=
\cdots =\lambda_{m_l}$
 such that for $\p$-almost  every $\omega\in \Omega$, there exist
 subspaces
  $E=E^1_{\omega} \supset \cdots \supset E^l_{\omega} \supset E^{l+1}_{\omega}=\{0\}$
  such that:

\begin{enumerate}

\item Equivariance equality: for  every  $1\leq i \leq l$,
$A(\omega)\cdot E^i_{\omega} =E^i_{\theta(\omega)}$ \item  for
every  $i=1,\cdots, l$ and every non zero vector $v$ of $E$, $v\in
E^i_{\omega}\setminus E^{i+1}_{\omega}$ if and only if  $\lim
\frac{1}{n}\log||A(\theta^{n-1} (\omega))\cdots A(\theta
(\omega))A(\omega) v|| = \lambda_{m_i}$.
    \item  $m_i=dim(E^i_{\omega}) - dim(E^{i+1}_{\omega})$, for every  $i=1, \cdots, l$. \end{enumerate}
    If, moreover,  $\theta$ is invertible  then there exists a splitting $V={F}_{\omega}^1  \oplus \cdots \oplus  {F}_{\omega}^l$ such that 

    \begin{enumerate}
 \item[4.] Equivariance equality: for  every  $1\leq i \leq l$,
 \begin{equation}A(\omega)\cdot F^i_{\omega} =F^i_{\theta(\omega)}.\label{equi2}
 \end{equation}
  \item[5.]  for
 every  $i=1,\cdots, l$ and every non zero vector  $v\in
 F^i_{\omega}$, 
\begin{equation}\lim_{n \rightarrow +\infty}
 \frac{1}{n}\log||A(\theta^{n-1} (\omega))\cdots A(\theta
 (\omega))A(\omega) v|| = \lambda_{m_i}\label{future}
 \end{equation}
 and 
 \begin{equation}  \lim_{n\rightarrow +\infty}
 \frac{1}{n}\log||A^{-1}(\theta^{-n} (\omega))\cdots A^{-1}(\theta^{-1}
 (\omega))  v|| = -\lambda_{m_i}.\label{past}  \end{equation} 
  \item[6.] $E_{\omega}^i=\oplus_{j=i}^l{F_{\omega}^j}$, for every  $i=1, \cdots, l$. 
    \end{enumerate}
 Moreover, the subspaces $E_{\omega}^i$ and $F_{\omega}^i$ are unique $\p$-almost everywhere, and
 they depend measurably on $\omega$. 
    
   \label{theoreme_oseledets}\end{theo}

\begin{proof}[Proof of Proposition \ref{propexistence}]
Let $d=dim(V)$,  $\Omega=\textrm{GL}(V)^{\N^*}$, $\p=\mu^{\otimes \N^*}$, $\theta$ the shift operator and $A: \Omega\longrightarrow G, \omega=(g_i)_{i\in \N^*}\longmapsto g_1$.  The distinct Lyapunov exponents relative to the measure $\mu$ will be denoted by $\lambda_1=\cdots = \lambda_{m_1}>\cdots \lambda_{m_{l-1}+1}=\cdots = \lambda_{m_l}=\lambda_d$. The ones relative to the reflected measure $\check{\mu}$, law of $g_1^{-1}$, are $-\lambda_{m_l}>\cdots >-\lambda_{m_1}$. We will construct $\nu$ as the law of the least expanding vector $R_n^{-1}$ given by Oseledets theorem. More precisely, applying Oseledets theorem for the dynamical system $(\Omega, \p, \theta)$ and the transformation $A^{-1}$ (and not $A$),  we obtain for the same  integers $l, 
 m_1, \cdots, m_l$ above and for the same exponents $\lambda_{m_i}$'s,  a  
    random filtration
  $E^{0}_{\omega}=\{0\}\subset E^1_{\omega} \subset \cdots \subset E^l_{\omega}=V$
  such that for $\p$-almost every $\omega=(g_i)_{i\in \N^*}\in \Omega$:

  \begin{enumerate}
  \item  \begin{equation}E^i_{\omega}= g_1 \cdot E^i _{\theta(\omega)},\label{equi}\end{equation}
  for every  $1\leq i \leq l$.
  \item  For every non zero vector  $v$ of $V$ and every $i=1, \cdots, l$:
   \begin{equation}v\in E^i_{\omega}\setminus E^{i-1}_{\omega} \Longleftrightarrow \lim_{n\rightarrow +\infty} \frac{1}{n}\log||R_n^{-1} v || = -\lambda_{m_i},\label{liapou}
   \end{equation}
      where  $R_n(\omega)=g_1\cdots g_n$ is the right random walk.
      \item For every $i=1, \cdots, l$,
      \begin{equation}\label{dim} m_i=dim(E^i_{\omega}) - dim(E^{i-1}_{\omega}).\end{equation}
  \end{enumerate}
 Under the assumption $\lambda_1>\lambda_2$, we have  $m_1=1$ so that by  \eqref{dim} 
 $\mathrm{k}Z(\omega):=E^1_{\omega}$
is a line for   $\p$-almost every $\omega\in \Omega$. Let  $\nu$
be the law of the random variable  $Z: \Omega \longrightarrow
\PV, \omega\longmapsto [Z(\omega)]$ on the projective space. The
probability $\nu$ is $\mu$-stationary. Indeed, for every real
valued  measurable function  $f$ on  $\PV$,

 \begin{eqnarray}
 \int_{\PV}{f([x]) d\nu([x])}&=& \int_{\Omega} {f\left(E^1_{\omega}\right) d\p(\omega)}\nonumber\\
 &=& \int_{\Omega} {f \left(\,g_1\cdot E^1_{\theta(\omega)}\,\right) d\p(\omega)}\label{a}\\
 &=&\int_{G}{\left[ \int_{\Omega}{f\left(\,\gamma \cdot  E^1_{\theta(\omega)}\,\right) d\p(\omega)}\right]\,d\mu(\gamma)}\label{b}\\
 &=& \int_{G\times \PV}{f(\gamma \cdot x)\, d\mu(\gamma)\, d\nu([x])} .\label{c}\end{eqnarray}
Equality  \eqref{a}  is straightforward consequence of the equivariance equality \eqref{equi};  \eqref{b} is due to the independence of $g_1$ and $\theta(\omega)=(g_2, g_2, \cdots, )$ while  \eqref{c} is true because   $\theta$ preserves the measure $\p$.\\
Finally, we show that $\nu([\mathcal{L}_{\mu}])=0$. Let $E$ be a proper $G_{\mu}$-stable 
subspace such that  $\lambda(E)<\lambda_1$. Fix  $\omega\in
\Omega$.  Then, since the least Lyapunov exponent of $\check{mu}$ restricted to $E$ is equal to $-\lambda(E)$, 
$$\forall v \in E, \,\,\lim\frac{1}{n} {\log||R_n^{-1}(\omega) v
||}  >-
\lambda_1.$$ Taking if necessary  $\omega$ in a measurable subset
of $\Omega$ of $\p$-probability $1$, assertion  \eqref{liapou} 
gives
$$\forall v\in E,  v\not\in \mathrm{k}Z(\omega)\,\,\,,\textrm{i.e.}\,\,\,[Z(\omega)]\neq [v].$$
Hence $\nu([E])=0$.\\
The fact that $\nu$ is a $\mu$-boundary is also a consequence of
the equivariance equality  (see for example \cite{kaimanovich},
\cite{ledrappier}, \cite{brofferio-schapira}).
\end{proof}
 
\subsection{Uniqueness of the stationary measure}\label{stationnaire3}
In this section, we prove that the stationary measure given by
Proposition  \ref{propexistence} is the unique $\mu$-stationary probability measure on $\PV\setminus [\LL_{\mu}]$. We fix   an orthonormal basis $(e_1, \cdots, e_d)$ of
$V$ (see Section \ref{nonarchisection} for the non-Archimedean case).  The dual vector space $V^*$ of $V$ will be equipped with the dual norm and with the dual  basis $(e_1^*. \cdots, e_d^*)$. We keep the same notation as Lemma \ref{duality} concerning other duality notation.

Recall that if $K$ denotes the isometry group of $(V,||\cdot||)$ and $A$ the subgroup of
 $\GL(V)$ consisting of diagonal matrices in the chosen basis, then the following decomposition holds $G=KAK$. 
   For $g\in \GL_d(\mathrm{k})$, we write $g=k(g) a(g) u(g)$ a KAK decomposition of 
$g$. We note $a(g):=\left(a_1(g), \cdots, a_d(g)\right)$.  Note   that $g^t=u(g)^t a(g)^t k(g)^t$ is a $KAK$ decomposition of $g^t$ in $\GL(V^*)$. 
When $\kk=\R$ or $\kk=\C$,  one can impose that $a_1(g)\geq \cdots\geq  a_d(g)>0$. When $\kk$
 is non-Archimedean, one can choose   
$a_1(g), \cdots, a_d(g)\in  \varpi ^{\Z}$ (with $\varpi$ a fixed uniformizer of $\kk$) and sort them     in   ascending order of their valuation. 
With this choice,  $a(g)$ is unique   and we can define
   the map  $N: \GL(V) \longrightarrow \kk\setminus \{0\}, g \mapsto a_1(g)$. Note that 
   a similar map $N: V\setminus \{0\}\longrightarrow \kk\setminus \{0\}$ was defined 
in Section 
\ref{geobehind}.  It will be clear from the context whether $N$ is applied to a non zero element of $V$ or  
 to an automorphism of $V$. Recall that in the Archimedean case, one has simply $N(x)=||x||$, $N(g)=||g||$ for $x\in V\setminus \{0\}$ and $g\in \GL(V)$.

\begin{prop}
Let $\mu$ be a probability measure on  $\PV$ such that
$\lambda_1>\lambda_2$ and $\nu$ a $\mu$-stationary probability on
$\PV$ such that $\nu([\mathcal{L}_{\mu}])=0$. Then there exists a random variable $\omega \mapsto [Z(\omega)]\in \PV$ of law $\nu$ such that:  

\begin{enumerate}
\item  almost surely, 
 every limit point of  $\frac{R_n}{N(R_n)}$  in $\textrm{End}(V)$ is a  
matrix of rank one $1$ whose image in $\PV$ is equal to $[Z]$.  \item $k(R_n)[e_1]$ converges almost surely to $[Z]$.  \end{enumerate}

In particular, $\nu$ is the unique such probability measure. 
 \label{propunicite}\end{prop}

\begin{proof} In item i.~below we prove  the proposition in the particular case $\mathcal{U}_{\mu}=V$. In item ii.~we check that this is enough to deduce the  uniqueness of the stationary measure on $\PV\setminus [\mathcal{L}_{\mu}]$. Finally, in item iii.~we prove the limit theorems claimed in the proposition in the general case.  \begin{enumerate}
\item[i.]  Assume first that  $\mathcal{U}_{\mu}=V$. \\
By Proposition \ref{reunion}, $\nu$ is non degenerate on $\PV$.  
Let
$\omega\in \Omega$ and $A(\omega)$ a limit point of
$\frac{R_n(\omega)}{N\left(R_n(\omega)\right)}$. We write
 $\frac{R_{n_k}(\omega)}{ N \left(R_{n_k}(\omega)\right)} \underset{k\rightarrow +\infty}{\longrightarrow} A(\omega)$. 
 Since $\nu$ is  non degenerate, the pushforward measure
$A(\omega)\nu$ on $\PV$ is well defined and we have the
following vague convergence:
\begin{equation}\nonumber R_{n_k}(\omega)   \nu \underset{k \rightarrow \infty} {\overset{\textrm{vague}}{\longrightarrow}} A(\omega)  \nu.\end{equation}
Since $\lambda_1>\lambda_2$, the KAK decomposition of
$R_n(\omega)$ shows that, taking if necessary $\omega$ in a
 measurable subset of $\Omega$ of $\p$-probability $1$, the
matrix $A(\omega)$ has rank $1$. Hence, if we denote by $\mathrm{k}Z(\omega)$ its image, then $A(\omega)\nu=\delta_{[Z(\omega)]}$, so
that

\begin{equation}\label{conv1}R_{n_k}(\omega)   \nu \underset{n \rightarrow \infty} {\overset{\textrm{vague}}{\longrightarrow}}\delta_{[Z(\omega)]}. \end{equation}
 \noindent But using Doob's theorem on convergence of bounded martingales,     Furstenberg  showed in \cite{furstenberg}  that there exists for $\p$-almost
every   $\omega$, a probability measure $\nu(\omega)$ on  $\PV$
such that
 \begin{equation}R_n(\omega)   \nu  \underset{n \rightarrow \infty}{\longrightarrow} \nu_{\omega}\label{conv2}\end{equation}  and
  \begin{equation}\label{conv3}\forall f\in \mathcal{C}\left(\PV\right), \, \E\left(\int {f\,d\nu_{\omega}}\right) = \int{f\,d\nu}.\end{equation}
By  \eqref{conv1} and \eqref{conv2}, we obtain the following
relation:
\begin{equation}   \nu_{\omega}=\delta_{[Z(\omega)]}.\label{dirac}\end{equation}
 In particular $[Z(\omega)]$  does not depend on the subsequence $(n_k)_{k\in \N^*}$. 
 By \eqref{conv3}, $\nu$ is the law of the random variable $\omega \longrightarrow [Z(\omega)]$ on $\PV$.   This proves the uniqueness of $\nu$, together with  item 1 in the case $\mathcal{U}_{\mu}=V$.  Item  2 is an immediate consequence of the KAK decomposition.

 \item[ii.]  Now if $\mathcal{U}_{\mu}\neq V$, we apply the previous part   for the restriction $\rho: T_{\mu} \longrightarrow \GL(\mathcal{U}_{\mu})$ on $\mathcal{U}_{\mu}$. Since $\mathcal{U}_{\rho(\mu)}=\mathcal{U}_{\mu}$, $\mathcal{L}_{\rho(\mu)}=\mathcal{U}_{\mu} \cap \mathcal{L}_{\mu}$ (see Remark \ref{restU}) and since the top Lyapunov exponent of $\rho(\mu)$ is simple, we obtain using item i.~a unique $\mu$-stationary   probability measure on $[\mathcal{U}_{\mu}] \setminus [\mathcal{U}_{\mu}\cap  \mathcal{L}_{\mu}]$. But by Proposition \ref{propstat}, any $\mu$-stationary probability measure on $\PV\setminus [\mathcal{L}_{\mu}]$ gives total mass to $[\mathcal{U}_{\mu}]$, then such a probability measure is unique.

 \item[iii.] It is left to prove the limit theorems  in the first and second claims of Proposition \ref{propunicite} even if $\mathcal{U}_{\mu}\neq V$.    
 For every $n\in \N$, let $k_n$ (resp. $u_n$) be the left (resp. right) $K$ part of $R_n$ in the $KAK$ decomposition.  
The following holds almost surely:
 
 $$\forall x\in V, \frac{R_n x}{N\left(\rho(R_n)\right)} =e_1^* (u_n x)  \frac{N(R_n)}{N \left(\rho(R_n)\right)} k_n e_1 + O\left(\frac{a_2(n)}{||\rho(R_n)||}\right).$$
 But  the Lyapunov exponent of $\rho$ is $\lambda_1$ and   $\frac{\log{a_2(n)}}{n}$ converges almost surely to  
    $\lambda_2<\lambda_1$. Hence $ \frac{a_2(n)}{||\rho(R_n)||}$ converges (exponentially fast)  to zero,  so that almost surely, 
    
\begin{equation}\forall x\in V, \frac{R_n x}{N\left(\rho(R_n)\right)} =e_1^* (u_n x)  \frac{N(R_n)}{N \left(\rho(R_n)\right)} k_n e_1 + o(1).\label{firstpart}\end{equation}

 Let now $\omega\in \Omega$ and  $k_{\infty}$ be a limit point of $(k_n)_{n\in \N^*}$. 
 We write $k_{\infty} = \underset{l\rightarrow +\infty}{\lim} {k_{n_l}} $. Passing to a subsequence if necessary, we may assume that 
 $\frac{\rho(R_{n_l})}{N \left(\rho(R_{n_l})\right)}$ converges to the non zero endomorphism  $A(\omega)$ of $\mathcal{U}_{\mu}$. 
    Choose any $x\in \mathcal{U}_{\mu}\setminus \textrm{Ker}(A(\omega))$. 
  In particular,  $\frac{R_{n_l} x}{N \left(\rho\left(R_{n_l}\right)\right)}  =  \frac{\rho(R_{n_l} x)}{N \left(\rho\left(R_{n_l}\right)\right)}  \underset{l \rightarrow +\infty}{\longrightarrow} A(\omega) x \in V\setminus \{ 0\}$.  
Since $K$ acts by isometry on $V$,     \eqref{firstpart} gives  then that   
 $$  \Big|  e_1^*(u_{n_l} x)  \frac{N(R_{n_l})}{N \left(\rho(R_{n_l})\right)} \Big|    \underset{l\rightarrow +\infty}{\longrightarrow}
||A(\omega) x||>0.$$
In particular, passing if necessary to a subsequence, we may assume that the sequence 
$\left( e_1^* (u_{n_l} x)  \frac{N(R_{n_l})}{N \left(\rho(R_{n_l})\right)}\right)_{l\in \N^*}$ converges in $\kk$ to some $\alpha(\omega)\in \kk\setminus \{0\}$. 
 By \eqref{firstpart} again, we deduce that 
 
   $$\frac{R_{n_l} x}{N\left(\rho(R_{n_l})\right)}     \underset{l\rightarrow +\infty}{\longrightarrow} \alpha(\omega) 
 k_{\infty} e_1 \in V\setminus \{0\} .$$
  In particular, $R_{n_l}[x] \longrightarrow k_{\infty} [e_1]$ in $\PV$. 
 But since $x\in \mathcal{U}_{\mu}\setminus  \textrm{Ker}(A(\omega))$, item i.~shows that $R_{n_l}[x]  \longrightarrow [Z(\omega)]$. 
 Hence $k_{\infty} [e_1]=[Z(\omega)]$.  This being true for all limit points of $(k_n[e_1])_{n\in \N^*}$, we deduce that 
  $k_n [e_1]$ converges almost surely to $[Z]$. This proves part 2 of the proposition in the general case. Since $\lambda_1>\lambda_2$, part 1 is an easy consequence of the KAK decomposition.

 \end{enumerate}
\end{proof}

\vspace{0.2cm}

\begin{corollaire}
Let $\mu$ be a probability measure on $\GL(V)$ such that $\lambda_1>\lambda_2$. Then,
\begin{enumerate}
   \item  For every   sequence $([x_n])_n$ in $\PV$ that converges to some $[x]\in \PV\setminus[\mathcal{L}_{\mu}]$, we have almost 
   surely,
$$  \inf_{n\in \N^*}  \frac{||L_n x_n||}{||L_n|| \,||x_n||} >0.$$
\item      $\frac{1}{n}\E(\log\frac{||L_n x||}{||x||})$ converges to $\lambda_1$ uniformly on compact subsets of $\PV\setminus[\mathcal{L}_{\mu}]$.  
\item There exists a random variable $[Z]\in \PV$ of law $\nu$ such that 
for every $[x]\in \PV\setminus [\mathcal{L}_{\mu}]$, the sequence of random variables $(R_n [x])_{n\in \N^*}$ converges in probability to $[Z]$.    
\end{enumerate}
\label{corollaire-unicite}\end{corollaire}
 
\begin{proof}
 \begin{enumerate}
 \item 
 Let  $([x_n])_{n\in \N^*}$ be 
a sequence in $\PV$ that converges to $[x]\in \PV \setminus [\mathcal{L}_{\mu}]$. Write $L_n=K_nA_nU_n$   the KAK decomposition of $L_n$.  
Since $\lambda_1>\lambda_2$, 

$$\frac{||L_n x_n||}{||L_n|| \,||x_n||}=  \frac{||A_n U_n x_n||}{|a_1(n)| \,||x_n||}=  \Big|  U_n^t e_1^*
\left( \frac{x_n}{N(x_n)}\right)\Big| +o(1).$$
 
\noindent Since $U_n^t=k(L_n^{t})=k(X_1^t \cdots X_n^t)$, and since
 the Lyapunov exponents of $\check{\mu}$ coincide with those of $\mu$, 
item 2 of Proposition \ref{propunicite} applied to $\check{\mu}$ shows then that
  $$\frac{||L_n x_n||}{||L_n|| \,||x_n||} \underset{n\rightarrow +\infty}{\longrightarrow}  {| \check{Z} (x)|}, $$
 with $||\check{Z}||=1$  and $[\check{Z}]$ being a random variable on $\PS$ with law 
   the unique $\check{\mu}$-stationary probability
measure   $\check{\nu}$ 
on  $\textrm{P}\left( V^*\right)\setminus [\mathcal{L}_{\check{\mu}}]$. 
  Let  $H= (\kk x)^0\subset \PS$ be the hyperplane
orthogonal to   $x$. Since $x\not\in \mathcal{L}_{\mu}$, ${\mathcal{L}_{\mu}}^{0}
\not\subset H$, i.e.~by Lemma \ref{duality} $\mathcal{U}_{\check{\mu}} \not\subset H$.
 Since, by proposition \ref{propstat} $\check{\nu}([\mathcal{U}_{\check{\mu}}])=1$ and 
$\check{\nu}$  is non degenerate on $[\mathcal{U}_{\check{\mu}}]$,  we deduce that 

$$\check{\nu}([H])=\check{\nu}([H \cap
 \mathcal{U}_{\check{\mu}}])=0.$$
Hence, almost surely, $| \check{Z} (x)
| \neq 0$. Item 1.  is then proved. 
\item To prove item 2, take a compact subset $K$ of $\PV\setminus [\mathcal{L}_{\mu}]$.  
By compactness of $K$, it is enough to show that for any sequence $([x_n])_n$ in $K$ that converges to some $[x]\in K$, one has that 
$\frac{1}{n}\E(\log\frac{||L_n x_n||}{||x_n||}) \longrightarrow \lambda_1$. By the previous item 1.,
 we deduce that $\frac{1}{n} \log\frac{||L_n x_n||}{||x_n||}$ converges to $\lambda_1$. But by the law of large numbers, it is easy to see that the sequence $\{\frac{1}{n}  \log \frac{||L_n x_n||}{||x_n||},
n\geq 1\}$ is uniformly integrable. This is enough to conclude.

\item Now we prove item 3. We claim that for every compact subset $K$ of $\PV\setminus [\mathcal{L}_{\mu}]$, 
\begin{equation}\sup_{[x],[y]\in K} \, \E\left(\delta(L_n[x],L_n[y])\right) \underset{n\rightarrow +\infty}{\longrightarrow} 0.\label{while}\end{equation}
Admit for a while  \eqref{while} and let us   indicate how to conclude.   By the proof of item 1 of Proposition \ref{propunicite}, there exists a random variable $[Z]\in \PV$ of law $\nu$ such that, almost surely, $R_n \nu \overset{\textrm{vague}}{\underset{n\rightarrow +\infty}{\longrightarrow}} \delta_{[Z]}$. Let $\epsilon>0$. Since $\nu$ is a probability 
measure on the Polish space $\PV\setminus [\LL_{\mu}]$, one can find a compact   subset $K_0=K_0(\epsilon)$   of  $\PV\setminus   [\mathcal{L}_{\mu}]$ such that $\nu(K_0)>1-\epsilon$. 
Now we write for every $n\in \N^*$: 
\begin{eqnarray}
\E \left( \delta(R_n[x], [Z])\right) & = & \int_{\PV}{\E \left( \delta(R_n[x], [Z])\right)\,d\nu([y])}\nonumber\\
& \leq & \int_{\PV}{\E \left( \delta(R_n[x], R_n[y])\right)\,d\nu([y])} +    \int_{\PV}{\E \left( \delta(R_n[y],[Z])\right)\,d\nu([y])}    \nonumber\\
& \leq & \epsilon  + \int_{K_0}{\E \left( \delta(R_n[x], R_n[y])\right)\,d\nu([y])} + 
\int_{\PV}{\E \left( \delta(R_n[y],[Z])\right)\,d\nu([y])} 
\nonumber\\
& \leq &\epsilon + \sup_{[x],[y]\in K_0\cup \{x\}}\E \left(\delta(R_n[x], R_n[y])\right)  +  \E\left( \int_{\PV}{\delta(R_n[y], [Z])\,d\nu([y])}\right) \nonumber
\end{eqnarray}
 In the third line we used $\delta\leq 1$ and,    in the last line, we used Fubini's theorem. 
  The second term of the right hand side converges to zero as $n$ tends to infinity by  \eqref{while}  and the fact that $R_n$ and $L_n$ have the same law for every $n$. The last term converges to zero  by the dominated convergence theorem and the fact that, almost surely,   $R_n \nu \overset{\textrm{vague}}{\longrightarrow} \delta_{[Z]}$. Since $\epsilon>0$ was arbitrary, we deduce that  $\E \left( \delta(R_n[x], [Z])\right)  \underset{n\rightarrow +\infty}{\longrightarrow} 0$ and a fortiori that  $R_n[x]$ converges in probability to $[Z]$ as desired. \\
Finally, we prove  \eqref{while}.  Fix   a compact
subset $K$ of $\PV\setminus [\mathcal{L}_{\mu}]$. For every sequences  $([x_n])_{n\in \N^*},([y_n])_{n\in \N^*}$ of elements in $K$ that converge in $K$, item 1.~  shows that  almost surely:

\begin{equation}
\delta(L_n[x_n], L_n[y_n])\leq \frac{||\bigwedge^2 L_n||}{||L_n^2||} \, \frac{||L_n||\, ||x_n||}{||L_n x_n||} \frac{||L_n|| \,||y_n||}{||L_n y_n||}\underset{n\rightarrow +\infty}{\longrightarrow} 0.\label{while1}\end{equation}
By the dominated convergence theorem and the compactness of $K$, we deduce that for every sequence of elements $[x_n],[y_n]$ in $K$, $\E\left(\delta(L_n[x_n], L_n[y_n])\right) \underset{n\rightarrow +\infty}{\longrightarrow}  0$. This implies  \eqref{while}.

\end{enumerate}
\end{proof}

\begin{remarque} We deduce from the previous corollary that: if $\lambda_1>\lambda_2$, then 
\begin{enumerate}
\item $\underset{[x]\in \PV}{\sup} \frac{1}{n}\E(\log\frac{||L_n x||}{||x||}) \underset{n\rightarrow +\infty}{\longrightarrow}  \lambda_1$. 
 \item if $\mathcal{L}_{\mu}=\{0\}$, then $ \underset{[x]\in \PV}{\inf} \frac{1}{n}\E(\log\frac{||L_n x||}{||x||})\underset{n\rightarrow +\infty}{\longrightarrow}\lambda_1$. This is coherent with the result of Furstenberg-Kifer   saying that $\mathcal{L}_{\mu}=\{0\}$ if and only if there exists a unique cocycle average (see Remark \ref{works_FK}). \end{enumerate}
 \label{liapou_remarque}\end{remarque}

 We end this section by noting that Corollary \ref{spectrum-FK} is obtained    by applying  Theorem \ref{existence_uncite} on each subspace $\mathcal{L}_i$ given by Theorem \ref{furstenberg-kifer}. Indeed,  $\mathcal{L}_i = \mathcal{L}_{\rho_i(\mu)}$   where $\rho_i$ is the restriction to $\mathcal{L}_i$.   
 \section{The limit set and the support of the stationary measure}\label{limit_proof}
 
   In this section, we understand further the support of the unique 
 $\mu$-stationary measure given by Theorem \ref{existence_uncite},   
 by relating it to the limit set of $T=T_{\mu}$. 
 We will adapt the proof of \cite{gold-guiv} to our setting. 
 Finally we give two concrete examples by simulating the limit set of two non irreducible subgroups of 
 $\GL_3(\R)$.   
 \subsection{Proof of Theorem \ref{limitset}}
 We keep the same notation as in Section   \ref{section_limit_set} concerning the set $\mathcal{Q}$ of quasi-projective maps of $\PV$, the limit set $\Lambda(T) \subset[U]$ of $T$ relative to the subspaces $L$ and $U$, and the subsets $p^+(T_0)$ (resp.~$p^+(T^a_0)$ ) of $\PV$ of attractive points of proximal elements of $T$ in $[U]$ (resp.~in $[U\setminus L]$).     First we check that following property that we claimed to hold: 
 
 \begin{lemme}  $\Lambda(T)$ is a closed $T$-invariant subset of $[U]$. \label{quasi}\end{lemme}

\begin{proof} Only the closed part needs a proof. Let $y_i\in \Lambda(T)$ be a sequence in $\Lambda(T)$ that converges in $\PV$ to some $y$. Clearly $y\in [U]$. For each $y_i=p(\mathfrak{q}_i)$, find a projective subspace $[W_i]$ of $\PV$, a sequence of projective maps $\{[g_{i,n}]\}_{n\in \N}$ such that  $[g_{i,n}]$ converges pointwise, when $n$ tends to infinity, to $\mathfrak{q}_i$ with $\mathfrak{q}_i$ that maps $\PV \setminus [W_i]$ to $y_i$. Since by \cite[Lemma 2.10, 1.]{goldsheid-margulis},  $\mathcal{Q}$ is sequentially compact for the topology of pointwise convergence, there exists a  subsequence of the $\mathfrak{q}_i$'s that converges to some quasi-projective map $\mathfrak{q}$. To simplify notations, we will  write $\mathfrak{q}=\underset{i \rightarrow +\infty}{\lim}\mathfrak{q}_i$. 
Let $$W:=\liminf_{i\rightarrow +\infty} {W_i}=\{x\in V; \exists i(x); \forall i\geq i(x), x\in W_i\}=\bigcup_n\bigcap_{k\geq n}{W_k}.$$
It is clear that $W$ is a  subspace of $V$ and hence  that the   union above   is a finite one.  Taking the latter fact into account and the fact that   $U \not\subset W_i$ for every $i$, we deduce that  $U\not\subset W$. Let now $[x]\not\in [W]$. By definition of $W$, one can find  a subsequence $(W_{i_j})_{j\in \N^*}$ such that   $x\not\in W_{i_j}$ for every $j$. Hence
   $$\mathfrak{q} [x]=\underset{i \rightarrow +\infty}{\lim} {\mathfrak{q}_i [x]}= \underset{j \rightarrow +\infty}{\lim} {\mathfrak{q}_{i_j}  [x]}=\underset{j \rightarrow +\infty}{\lim}{y_j}=y.$$ 
   It is left to show that $\mathfrak{q}$ is a pointwise limit of projective transformations that belong to $\textrm{PT}\subset \textrm{PGL}(V)$. Since each $\mathfrak{q}_i$ is such a map and since $\mathfrak{q}$ is the limit of the $\mathfrak{q}_i$'s, this follows from \cite[Lemma 2.10, 2.]{goldsheid-margulis}.  \end{proof}

We are now  able to prove Theorem \ref{limitset}. In item 1. of the following proof,
 we use the same notation as the invertible version of Oseledets theorem  (Theorem \ref{theoreme_oseledets}).  Also, any linear transformation of $V=\kk^d$ will be identified with its matrix in the canonical basis $B_0=(e_1, \cdots, e_d)$. The set of linear maps between two vector spaces $V_1$ and $V_2$
   will be denoted by   $\mathcal{L}(V_1, V_2)$.  
 \begin{proof}[Proof of Theorem \ref{limitset}]
 \begin{enumerate}
 \item  Consider the  dynamical system $\left( \Omega =\textrm{GL}(V)^{\Z}, \p=\mu^{\otimes \Z}, \theta \right)$ with $\theta$ the shift $\theta\left((g_{i})_{i\in \Z}\right):=(g_{i+1})_{i\in \Z}$.   Applying  Oseledets theorem in the invertible case, and using equivariance property \eqref{equi2}, we get that  the cocycle $\Z\times \Omega \longrightarrow \textrm{GL}(V), L_n(\omega)=A\left(\theta^{n-1}(\omega)\right) \cdots A(\omega)=g_{n-1} \cdots g_0$ is    cohomologous to a block diagonal   one. More precisely, denote by      $\omega \mapsto \phi_{\omega} \in \textrm{End}(V)$  the random transition matrix from $B_0$   to a   measurable adapted basis of the splitting $V=\oplus_{i=1}^l{F_{\omega}^i}$.   Then there exists a random block diagonal matrix $\Delta_n$ such that the following identity holds for $\p$-almost every $\omega\in \Omega$:   
 
\begin{equation}
L_n =(\phi \circ \theta^n) \,  \Delta_n  \,\phi^{-1} 
   \label{mark}\end{equation}
 For every $\omega\in \Omega$, let   $v_n(\omega):=\phi(\theta^n \omega) \phi^{-1}(\omega)$. \\
    As in    \cite[Lemma 2]{Guivarch3}, using Poincaré recurrence theorem,     
  we can find  almost surely  a  random subsequence  $n_k(\omega)=n_k$ such that $\underset{{k\rightarrow +\infty}}{\lim} {v_{n_k}(\omega)}=I$. \\
But  since $\lambda_2<\lambda_1$,   \eqref{future} gives that  $\Delta_n e_1= {\lambda_n^1} e_1$ with $\lambda_n^1$ being a random non zero scalar such that  
 $\lim \frac{1}{n}\log{|\lambda_n^1|}=\lambda_1$ almost surely.
 Also, since   $\frac{1}{n} \log ||\Delta_n x|| \leq \lambda_2<\lambda_1$ for every $x\in V\setminus \kk e_1$, we deduce that
  $(\phi  \Delta_n \phi^{-1})/{\lambda_n^1}$  converges almost surely to  the random projection endomorphism $\Pi$ on the line $[Z]:=[\phi (e_1)]$ parallel to $\textrm{Ker}(\Pi)=\oplus_{i=2}^l{F_{\omega}^i}$.

Combining the previous fact with  \eqref{mark}, we get     that almost surely 
   \begin{equation}\label{mark4} \lambda_{n_k}^{-1}L_{n_k} \underset{k\rightarrow +\infty}{\longrightarrow} 
  \Pi. \end{equation}
 
Let $\omega\in \Omega$ where the previous convergence holds. Since  
$\Pi(\omega)$ is a  rank one projection, it is proximal. 
 By  
a perturbation argument,     $L_{n_k}(\omega)$ is also proximal for all large $k$ with a dominant eigenvector $p^+(L_{n_k})$ close to $[Z(\omega)]$. 
By \eqref{future}, $[Z(\omega)] \not\in [L]$. Hence for all large $k$, $p^+(L_{n_k})\not\in [L]$.  
Let us check  that $p^+(L_{n_k})\in [U]$. Indeed, the largest eigenvalue of $L_{n_k}$ is either an eigenvalue of its   restriction to $U$ with  its corresponding eigenvector being that of the restriction operator,  or is an eigenvalue of its projection  on $V/U$. But the latter eigenvalue grows at most as 
  $\exp(n_k \lambda_2)$, while it follows from  \eqref{mark4} that the spectral radius of $L_{n_k}$ growth as the norm 
 of $||L_{n_k}||$, i.e.~as $\exp(n_k \lambda_1)$. Since $\lambda_2<\lambda_1$, we deduce that $p^+(L_{n_k})\in [U]$, for all large $k$. 
 Hence $L_{n_k}\in T_0^a$ so that  $[Z (\omega)]\in \overline{p^+(T^a_0)}$. In particular, ${p^+(T^a_0)} \neq \emptyset$. \\
 
 Now the law of the random variable $[Z]$ on $\textrm{P}(V)$ is the stationary measure $\nu$.
 Indeed, by \eqref{past} and the uniqueness part of Oseledets theorem, we deduce that 
  the filtration $\{0\}\subset {F_\omega^1}\subset {F_\omega^1} \oplus {F_\omega^2} \subset \cdots  \oplus_{i=1}^{l-1}{F_{\omega}^i}\subset V$ depends only on the past of $\omega=(g_i)_{i\in \Z}$ i.e.~ on $(\cdots, g_{-2}, g_{-1})$.  Hence $[Z]$ is an independent copy of    the least expanding vector 
   of $R_n^{-1}(\omega)=g_n^{-1}\cdots g_1^{-1}$ given by Oseledets theorem.   The proof of Proposition \ref{propexistence} shows then that $[Z]$ has law $\nu$. 
Hence    $$\nu\left(\overline{p^+(T^a_0)}\right)=\p\left( [Z] \in  \overline{p^+(T^a_0)}\right)=1,$$
  so that \begin{equation}\textrm{Supp}(\nu) \subset \overline{p^+(T^a_0)}\subset  \overline{p^+(T_0)}.\label{mark5}\end{equation}

  Conversely, let $h\in T_0$. Then $h^n/||h^n||$ converges to the projection $\eta$ on the line generated by $p^+(h)\in [U]$ and parallel to some $h$-invariant subspace of $V$. In particular, $U\not\subset Ker(\eta)$. Since by Theorem \ref{existence_uncite} $\nu$ is non
   degenerate in $[U]$, we have 
  that $\nu\left([U] \cap Ker(\eta)\right)=0$ so that  ${{h^n  \nu}} \underset{n\rightarrow +\infty}{\overset{\textrm{weakly}}{\longrightarrow}} \delta_{p^+(h)}$. Since 
  $\textrm{Supp}(\nu)$ is $T$-invariant,  we get $p^{+}(h)\in \textrm{Supp}(\nu)$.  Hence 
  $p^{+}(T_0)\subset \textrm{Supp}(\nu)$ and 
  
\begin{equation}\textrm{Supp}(\nu)\supset \overline{p^{+}(T_0)}\supset \overline{p^{+}(T^a_0)}.\label{mark6}\end{equation}
  Inclusions  \eqref{mark5}  and  \eqref{mark6} show item 1. 
 
 \item Let $\omega\in \Omega$ and   $W:= \textrm{Ker}\left(\Pi(\omega)\right)$.  
  By the previous item,  we deduce that 
       the sequence  of projective maps   $\left([L_{n_k}]\right)_{k\in \N}$ converges pointwise on
        $\PV\setminus [W]$ to the constant map $[x]\mapsto [Z(\omega)]$. 
        Since $\Pi(\omega)$ is  a rank  one  proximal endomorphism of $V$,  
          $\textrm{Im}(\Pi(\omega))\not\subset  \textrm{Ker}\left(\Pi(\omega)\right)$.
            But $\textrm{Im}(\Pi(\omega))=[Z(\omega)]\in [U]$, so    that $[U]\not\subset [W]$.  
     Considering the     sequence $r_{n_k}:=\sup \{N(L_{n_k} x);x\in W, ||x||=1\}$ in $\kk$, 
     we see that    
     the sequence  $(r_{n_k} {L_{n_k}}_{|_{W}})$ of linear maps  from $W$ to $V$
     admits a subsequence that converges in $\mathcal{L}(W,V)$ to   a linear map $\Pi'(\omega)$
      such that $W':=\textrm{Ker}\left(\Pi'(\omega)\right)\subsetneqq W$. In particular, $[L_{n_k}]$
       admits a subsequence  that converges pointwise on 
      $[W]\setminus [W']$ and hence on $\PV\setminus [W']\supsetneqq \PV\setminus [W]$.
       Repeating this procedure at most $\textrm{dim}(W)$ times,  
       we obtain    a subsequence  of    $[L_{n_k}]$   that  converges pointwise on $\PV$ to 
      a quasi projective map $\mathfrak{q}$  such that $\mathfrak{q}$ maps $\PV \setminus [W]$ to $p(\mathfrak{q}):=[Z(\omega)]\in [U]$.    
         Hence $\mathfrak{q} \in \widehat{T}$ and $[Z(\omega)]\in \Lambda(T)$. Since by Lemma \ref{quasi} $\Lambda(T)$ is closed in $\PV$ and since $[Z]$ has law $\nu$, we deduce that  
         
    $$\textrm{Supp}(\nu)\subset \Lambda(T).$$
    
  Conversely,  let   $y=p(\mathfrak{q})\in \Lambda(T)$ and $([g_n])_{n\in \N}$ a sequence of projective maps converging 
  pointwise to $\mathfrak{q}$, together with a projective subspace $[W]$ of $\PV$ that does not contain $[U]$ and such that with 
  $\mathfrak{q}$ maps $\PV\setminus [W]$ to the point $y$ of $\PV$. Since $\nu$ is non degenerate on $[U]$ and since $[W]$ does not contain $[U]$, 
  we deduce that $\nu([U] \cap [W])=0$  so that $\mathfrak{q} \nu$ is the Dirac measure on $y$. We conclude  that  $g_n \nu \underset{n\rightarrow +\infty}{\overset{\textrm{weakly}}{\longrightarrow}}\delta_{y}$. Since $\textrm{Supp}(\nu)$ is $T$-invariant, 
   we deduce that $y\in \textrm{Supp}(\nu)$. Consequently, $$ \textrm{Supp}(\nu)\supset \Lambda(T).$$
  Item 2 is then proved.

  \item Let $[x]\in \PV\setminus [L]$. By Theorem \ref{furstenberg-kifer},  there exists $\Omega_x\subset  \Omega$ such that for every $\omega\in \Omega_x$, $\lim_{n\rightarrow +\infty}{\frac{1}{n} \log {||L_n(\omega) x ||}}=\lambda_1$. Reducing if necessary $\Omega_x$ to a subset of $\p$-probability one, we deduce from \eqref{future} that  
  $x\not\in \oplus_{i=2}^l{F_{\omega}^i}=\textrm{Ker}(\Pi(\omega))$ for every $\omega\in \Omega_x$. By \eqref{mark4}, we deduce that for every $\omega\in \Omega_x$ there exists a random subsequence $(n_k)_k$ such that $L_{n_k}[x] \underset{k\rightarrow +\infty}{\longrightarrow} [Z(\omega)]$; so that $[Z(\omega)]\in \overline{T \cdot[x]}$. Since $\p(\Omega_x)=1$ and $[Z]$ has law $\nu$, we deduce that $\textrm{Supp}(\nu)\subset \overline{T\cdot[x]}$. 
  \end{enumerate}
  \end{proof}
  
  \begin{remarque} One can  also use  Corollary \ref{corollaire-unicite} to prove item 3.~ 
  above with the right random walk. Indeed it follows from item 3.~of Corollary \ref{corollaire-unicite} that there exists  a non random subsequence $(n_k)$ such that    $R_{n_k}[x]$ converges almost surely to a random variable of law $\nu$.   \end{remarque}

 \vspace{1cm}
We deduce easily the proof of Corollary \ref{cor-limit} stated in Section \ref{section_limit_set}. 
\begin{proof}[Proof of Corollary \ref{cor-limit}:]  
The implication $1\Longrightarrow  2$ follows immediately from the last part of item 3.~of Theorem \ref{limitset}. The implication $2 \Longrightarrow  1$ 
is an easy consequence of the fact that $[L]$ is $T$-invariant. 
  The equivalence between $1$ and $3$ follows directly from the first part of  item 3.~of Theorem \ref{limitset}. \\
Finally, we check the equivalence $1 \Longleftrightarrow 4$. 
 By item 1.~of Theorem \ref{limitset}, $\textrm{Supp}(\nu)$ 
 is compact if and only $p^+(T_0^a)$ is precompact in $O$. As in Section \ref{geobehind}, identify   $O$ with the 
 quotient of  $X=L\times S(V/L)$ by the action of the unit sphere $\mathcal{S}_{\kk}$ of $\kk$.
  A straightforward computation for any eigenvector of an upper triangular bloc matrix shows that an element of  $p^+(T_0^a)$ is identified with 
 $\mathcal{S}_{\kk}(t_0, \xi_0)$, where 
    $\begin{pmatrix} 
A & B \\
0 & C 
\end{pmatrix}  \in T$, $C$ is proximal, $\rho_{\textrm{spec}}(C)>\rho_{\textrm{spec}}(A)$,  $\xi_0$    a chosen normalized eigenvector of $C$, and 
$t_0$ a fixed point of the affine map $t\mapsto \frac{At +B\xi_0}{\lambda_{\textrm{top}}(C)}$ of $L$.
 Since $\rho_{\textrm{spec}}(A)<\rho_{\textrm{spec}}(C)$, this affine map has a unique fixed point in $L$, 
 namely $-\left(A-\lambda_{\textrm{top}}(C) I\right)^{-1} (B \xi_0)$. 
 The desired result follows then from the compactness  of the unit sphere of $V/L$.
 \end{proof}

  The   last part in the proof of item 1.~ of Theorem \ref{limitset}  
   gives actually a stronger result and a sufficient non-compactness criterion for the support of $\nu$, seen in 
   $O=\PV\setminus [L]$.  
   Elements of $T$ are represented by matrices in a suitable basis of $V$ where the first diagonal bloc is the restriction to $L$.

\begin{lemme}
  Assume $U=V$. If there exists $g=\left( \begin{matrix} 
 A & B\\
0 & C
\end{matrix}\right)\in T$ such that $\rho_{\textrm{spec}}(A)>\rho_{\textrm{spec}}(C)$,  then $\textrm{Supp}(\nu)\subseteq O$ is not compact. 
  Here $\rho_{\textrm{spec}}(\cdot)$ denotes the spectral radius evaluated in some finite extension of the local field $\mathrm{k}$. \label{additional}\end{lemme}
 \begin{proof}
Indeed, let $g\in T$ be such an automorphism. 
It is enough to check  that there exists a projective subspace $[E]$ of $\PV$ such that for every $[x]\in \PV\setminus [E]$, every limit point of $g^n[x]$ belongs to $[L]$. Indeed,  since   $U=V$, the probability measure $\nu$ on $\PV$ is non-degenerate (Theorem \ref{existence_uncite}) so that $\nu([E])=0$. In particular, every limit point of the sequence of probability measures $(g^n \nu)_{n\in \N^*}$ on $\PV$ gives total mass to $[L]$. Since
  $\textrm{Supp}(\nu)$ is $T$-invariant, we deduce that $\textrm{Supp}(\nu)\cap [L]\neq \emptyset$ and then that $\textrm{Supp}(\nu)\cap O$ is not compact. \\
Now we check our claim.  Assume first that the characteristic polynomial of $g$ splits over $\kk$.  Since  $\rho_{\textrm{spec}}(A)>\rho_{\textrm{spec}}(C)$, the  generalized highest eigenspace $W$ for $g$  corresponding to the top eigenvalue  coincides with the one for  $A$ corresponding to the same eigenvalue. 
In particular, $W\subseteq L$. Writing  now  $g$ in its Jordan canonical form in $\GL(V)$, we deduce from the inequality $\rho_{\textrm{spec}}(A)>\rho_{\textrm{spec}}(C)$,   the existence of  a $g$-invariant supplementary $E$ of $W$ in $V$, such for every $[x]\in \textrm{P}(V)\setminus [E]$, every limit point of $g^n[x]$  in $\textrm{P}(V)$ lies in $[W]\subseteq [L]$. This is what we wanted to prove.   It is left to check that the same holds when $\kk$ does not contain all the eigenvalues of $\kk$. This can be done by applying the previous reasoning to a finite extension $\kk'$ of $\kk$, and then use the natural embedding $ \textrm{P}(V)\xhookrightarrow{} \textrm{P}(V\otimes_{\kk} \kk ') $, where $\textrm{P}(V\otimes \kk')$ is the $\kk'$-projective space of the $\kk'$-vector space $V\otimes_{\kk} \kk ' $. 
    \end{proof}

\subsection{Compactness criterion, examples and    simulations}
   In this section, we illustrate our results for semigroups of linear transformations of $\R^3$ for which $\LL_{\mu}$ is a line, i.e.~ those  relative to  Example 3.~ of Section \ref{example_guiding}. Note that when $\lambda_1>\lambda_2$, 
   this is essentially the  only case to illustrate    since if $\LL_{\mu}$ is a plane, we are essentially in the contracting case of the affine situation and, if $\LL_{\mu}=\{0\}$, we are either in the i-p case (if the action is irreducible) or in the expansive case of the affine case   or in the degenerate case explained at the end of Example 3.~ of Section \ref{example_guiding}. \\
   
 In section \ref{compactness_criterion_section} below, 
   we give a sufficient compactness criterion for the support of the stationary measure in the open dense subset 
    $O=\textrm{P}^2(\R)\setminus [\LL_{\mu}]$ of the projective plane using the notion of joint spectral radius and Section \ref{geobehind}. In Section \ref{example-simulation}, we give three examples,   simulate the limit sets of two sub-semigroups of $\GL_3(\R)$ and justify our observations using the techniques we have developed.

        \subsubsection{Compactness criterion}\label{compactness_criterion_section}
     We recall the classical notion of joint spectral radius of a bounded set of square matrices introduced by Rota and Strang 
     in  \cite{rota-strang}.  
     Let $d\in \N^*$, $\mathcal{M}_d(\R)$ the set $d\times d$ matrices  and   denote by $\rho_{\textrm{spec}}(A)$ the spectral radius of $A\in \mathcal{M}_d(\R)$. 
     A subset of $\mathcal{M}_d(\R)$  is said to be bounded if it is bounded when $\mathcal{M}_d(\R)$  is endowed with one, or equivalently any, norm on $\mathcal{M}_d(\R)$.

         \begin{propdef}
    Let $d\geq 2$ and $\Sigma\subset \mathcal{M}_d(\R)$ be a bounded set.  Let $||\cdot||$ be any norm on $\R^d$ 
    and denote by the same symbol the operator norm it induces on $\mathcal{M}_d(\R)$. The joint spectral radius $\mathfrak{R}(\Sigma)$ of $\Sigma$ 
     is the following non negative real number, independent of the chosen norm:
     
\begin{equation}\mathfrak{R}(\Sigma):=\inf_{n\in \N^*} {\max \{ ||A_1\cdots A_n ||^{\frac{1}{n}}; A_i\in \Sigma\} }= \sup_{n\in \N^*} {\max\{\rho_{\textrm{spec}}(A_1\cdots A_n)^{\frac{1}{n}}; A_i\in \Sigma \}}.\label{joint1}\end{equation}
    
  \end{propdef}
\noindent The last equality was   proved by 
Berger-Wang \cite{berger-wang} after a question of Daubechies-Lagarias \cite{lagarias}.\\
Now we state our compactness criterion when $\LL_{\mu}$ is a line.

\begin{prop}
Let $\mu$ be a probability measure on $\GL_3(\R)$ with compact support. Assume that   $\lambda_1>\lambda_2$ and that $\LL_{\mu}$ is a line $L$. In a suitable basis of $\R^3=L\oplus \tilde{L}$, $\mu$ is seen as a probability measure on the  group 
 $$G:=\left\{
g=\left( \begin{matrix} 
 a_g & \underline{b}_g \\
0 & C _g
\end{matrix}\right); a_g\in \R^*, \underline{b}_g\in \R^2, C_g\in \GL_2(\R) \right\}\subset \GL_3(\R).$$
Let 
   $$r:=\mathfrak{R}(\{|a_g| C_g^{-1}; g\in \textrm{Supp}(\mu)\}).$$
  If $r<1$, then the support of the unique $\mu$-stationary probability measure $\nu$ on $O=\textrm{P}^2(\R)\setminus [\LL_{\mu}]$ is compact. 
 \label{compactness-criterion}\end{prop}
\vspace{0.5cm}

 In the following proof, $\Sigma$ denotes the support of $\mu$,  $T=T_{\mu}$   the semigroup generated by the support of $\mu$ and, for every $n\in \N^*$, $\Sigma^n$ denotes the subset of $T$ that consists of all the  elements that can be written $g_1\cdots g_n$ with $g_i\in \Sigma^n$ (so $\Sigma^n=\textrm{Supp}(\mu^n)$). The set $\textrm{Aff}(\R)$ of all affine maps of the real line is identified topologically with the product space $\R^*\times \R$. \\
We adopt also the notation and setting of Section \ref{geobehind} namely: 
 $X:=L\times S(\R^3/L)\simeq \R \times S^1$ (with $\R^3/L \simeq \tilde{L}$ and $\R^3/L$ endowed with its Euclidean structure), $T$ is considered acting on $X$ by the formula \eqref{easybutcrucial}  and the open   subset $O$ of $\textrm{P}^2(\R)$ is identified with $X/\{\pm 1\}$. 
 
\begin{proof} 
 Suppose $r<1$. Let $||\cdot||$ be the canonical norm on $\R^2$. By the definition of the joint spectral radius, there exists $n_0\in \N$ such that for every
 $g\in \Sigma^{n_0}$, $|a_g| ||C_g^{-1}|| < 1$.  
 This implies that for every $\xi\in S^1$ and   every $g\in \Sigma^{n_0}$,  
 the affine map   $$\sigma(g,\xi): t \mapsto \frac{a_g t+b_g \xi}{||C_g \xi||}$$
 of the real line   has linear part strictly less than $1$ in absolute value.  We will check hereafter that this implies that there exists a compact subset $ K_1\subseteq \R$ stabilized by the family $\{\sigma(g,\xi) ;g\in \Sigma^{n_0}; \xi \in S^1\}$. Let us indicate first how to conclude.  Let $K_2:=K_1 \times S^1$. This is a compact subset of $X=\R\times S^1$ stabilized by all elements of $\Sigma^{n_0}$. Hence, by setting 
   $$K_3:=K_1\,\cup\,  \bigcup_{g\in \underset{{n\leq  n_0}}{\cup}{ \Sigma^n}}{g\cdot K_1},$$
   we obtain another 
  compact subset of $X$ which is stabilized by $T$.  
 In particular, $K:=K_3/\{\pm 1\}$ is a compact subset of $X/\{\pm 1\}\simeq O$ stabilized by $T$. This implies that there exists a  $\mu$-stationary probability measure on $K$. By the uniqueness part of Theorem \ref{existence_uncite}, the aforementioned probability measure on $K$   coincides with $\nu$. Hence  
 $\textrm{Supp}(\nu)\cap O$ is compact.  \\
   Now we check the missing part of our proof.  We are 
  in the following general situation: $Y$ is a compact topological space  (here $Y= 
 {\Sigma^{n_0}} \times S^1 $), $\sigma$ is a continuous map
    $\sigma : y\in Y \mapsto 
 \sigma_y( t)=a_y t+ b_y$ from $Y$ to $\textrm{Aff}(\R)$ such that $|a_y|<1$ for every $y\in Y$. It is standard that   $\sigma(Y)$ stabilizes a compact subset of $\R$.  Let us check it for the completeness of the proof.  For every point $p\in \R$ and every $R>0$, let    $B(p,R):=[p-R, p+R]$ be the closed interval of center $p$ and radius $R$.    
Fix $y_0\in Y$. For every $y\in Y$,  let $f_y\in \R$ be the unique fixed point of the affine map $\sigma_y$ and   $R_y:=\frac{1+|a_y|}{1-|a_y|} |f_y-f_{y_0}|$.     For 
every $y\in Y$ and every $R>R_y$, 
             \begin{eqnarray}
\sigma_y
\left(B(f_{y_0},R)\right)&\subseteq&  \sigma_y \left(   B(f_y, R+|f_y-f_{y_0}|) \right)\nonumber\\
    & \subseteq& B\left(f_y, |a_y|(R+|f_y-f_{y_0})\right) \label{r3} \\
          &\subseteq&     B\left(f_y, R - |f_y-f_{y_0}|\right)     \nonumber\\
    & \subseteq& B(f_{y_0},R).
    \nonumber \end{eqnarray}
 Estimate  \eqref{r3}  follows from our choice of $R_y$. 
     Hence  for every $R>R_y$, $B(f_{y_0},R)$ is stable under $f_y$. Since $y\mapsto R_y$ is a continuous map on the compact space $Y$, then $R:=\sup\{R_y; y\in Y\}$ is finite so that 
   $\sigma(Y)$ stabilizes the compact subset $B(f_{y_0},R)$ of $\R$.    This is what we wanted to prove.    
\label{compactness-criterion}\end{proof}

  \begin{remarque}
The condition $r<1$ is equivalent to  the existence of some norm $||\cdot||_1$ on $\R^2$ such that $|a_g| ||C_g^{-1}||_1<1$ for all $g$ in the support of $\mu$. This is a consequence of a theorem of Rota-Strang \cite{rota-strang} which asserts that, for a bounded subset $\Sigma \subset \mathcal{M}_d(\R)$, 
 $$\mathcal{R}(\Sigma)=\inf_{||\cdot||\in \mathcal{N}}{\max \{||A||; A\in \Sigma\}},$$
where $\mathcal{N}$ is the set of norms on $\R^d$.  Hence, although the condition $r<1$ involves the eigenvalues/norms of the elements in the semigroup generated by the support of $\mu$, it can be read also on the support of $\mu$ solely. 
  \end{remarque}

    \begin{remarque}(Generalizations)
    \begin{itemize}
    \item 
Let $K\subseteq S^1$ be any lift of the limit set in $\textrm{P}^1(\R)$ 
of the strongly irreducible and proximal semigroup $T_{\pi(\mu)}$, where  $\pi: G\longrightarrow \GL(\R^3/L)\simeq \GL_2(\R)$.  
The proof of Proposition \ref{compactness-criterion} yields the following stronger statement.  \\
Let $||\cdot||$ be the canonical norm on $\R^2$. If  $|a_g| < ||C_g t||$ for every every $t\in K$ and every $g\in \Sigma^{n_0}$ (for some $n_0$ large enough) then the support of $\nu$ is compact.   \\
In view of Rota-Strang's theorem, we can formulate the criterion above in the following way: 
if one can find a norm $||\cdot||_1$ on $\R^2$ such that  
$|a_g| <||C_g t||_1$ for every $g\in \Sigma$  and every $t\in K$, then the support of $\nu$ is compact in $O$. 
\item 
When $\textrm{dim}(\LL_{\mu})>1$, the  proof of Proposition \ref{compactness-criterion} generalizes easily 
in view of Rota-Strang's theorem and reads as follows: 
if one can find norms $||\cdot||_1$ and $||\cdot||_2$   such that  for every $g\in \Sigma$,  $||A_g||_1\, ||C_g^{-1}||_2< 1$, 
 then $\textrm{Supp}(\nu)\cap O$ is   compact. However, formulating such a criterion in a more intrinsic way (i.e.~using eigenvalues of elements of $T$), requires taking into account  the data given by all the eigenvalues of elements in the semigroup living in the product group $\GL(\LL_{\mu})\times \GL(\R^d/\LL_{\mu})$, and not considering  the joint spectral radii of the diagonal parts $A$ and $C$ separately  (otherwise a criterion is valid but is  too restrictive). 
This can be done using the very recent notion of 
     \emph{joint spectrum of a bounded set of matrices}   introduced recently by Breuillard   and Sert \cite{Cagri1},  \cite{Cagri-Breuillard}. 
       We refrain from doing it here as it goes beyond the scope of the paper. 
       \end{itemize}
\end{remarque}

    \subsubsection{Examples and Simulations}\label{example-simulation}
     All our vectors will be written using the  cartesian coordinates in $\R^3$. Also, automorphisms of $\R^3$ will be identified with their matrices in the canonical basis of $\R^3$. We endow $\R^3$ with the canonical norm. 
   
 \paragraph{Example 1: non compact support in $O$}
 Consider the following matrices of $\GL_3(\R)$: 
    
 $$g_1=\left(\begin{matrix} 
 1 & 2 & 3 \\
 0 & 1 &  1\\
 0 & 0 & 1
 \end{matrix}\right)\,\,,\,\, 
 g_2=\left(\begin{matrix} 
 -1 & 1 & 2 \\
 0 & 0& -1\\
 0 &   1 & 0\end{matrix}\right).$$
 Let $L=\R(1,0,0)$ and 
   $\mu$ be the uniform probability measure on the set $\{g_1, g_2\}$. The projection of $T_{\mu}$ on the quotient $\R^3/L$ is a non-compact  strongly irreducible 
   sub-semigroup of $\SL_2(\R)$. Hence by Furstenberg theorem \cite{furstenberg},  
       $\lambda_1(\R^3/L)>0$ (this can be also deduced from Guivarc'h-Raugi theorem \cite{guivarch-raugi}  as the action on the quotient is i-p). 
   In particular, $\int_{G}{\log |a_g|\,d\mu(g)}<\lambda_1(\R^3/L)$  and the top Lyapunov exponent on the quotient is simple.   By Lemma \ref{corollaire-triangulaire},   $\lambda_1(\mu)>\lambda_2(\mu)$.  By the irreducibility of the action on the quotient,     $\LL_{\mu}=L$. Let us check that  $\mathcal{U}_{\mu}=\R^3$. All we need to prove is that there does not exist a $G_{\mu}$-invariant plane $W$ in $\R^3$ such that $\R^3=L\oplus W$.  Arguing by contradiction, suppose that such a plane $W$ exists. 
     Writing $g_2$ in its (real) Jordan canonical form, we deduce the existence of a $g_2$-invariant plane $W'$ such that $\R^3=L\oplus W'$. The non zero $g_2$-subspace $W\cap W'$ cannot be a line because $[L]$ is the only real eigenspace of $g_2$.  Hence $W=W'$ and in particular $g_1$ stabilizes $W'$.  
    However, a  direct computation shows that $W'=\textrm{Span}\left( (3,0,2), (-1,2,0)  \right)$ and that $g_1  (3,0,2)^t = (9,2,2)^t \not\in W'$. Contradiction.  \\

Let $X:=\R \times S^1$ embedded in $\R^3$ using the cartesian coordinates $(t,\xi)$ with $t\in \R$ and $\xi\in S^1$. A  bounded portion of $X$ is represented in the gray  
 cylinder of  Figure 1. Its axis is $L=\LL_{\mu}=\R(1,0,0)$ and the symmetry with respect to the origin of $\R^3$ (the centroid of the cylinder in Figure 1) acts naturally on $X$. 
 The space $X$ is a two-fold cover of the open dense subset $O=\textrm{P}^2(\R)\setminus [L]$ of the projective plane. The latter is     then thought   as a one-point compactification of $X/\{\pm 1\}$, the direction of the line $L$ being  the point at infinity.   Identifying $\R/L$ with the $yz$-plane plane $\tilde{L}$, we   let   $G$ act on $X$  by the formula \eqref{easybutcrucial} of Section \ref{geobehind}. By theorem \ref{existence_uncite} and Section \ref{geobehind}, there exists a  unique $\mu$-stationary stationary measure $\tilde{\nu}$ on  $X$ which is $\{\pm 1\}$ invariant.  
 By Theorem \ref{limitset}, $\textrm{Supp}(\tilde{\nu})$ is a two-fold cover of the limit set $\Lambda(T)$ of $T_{\mu}$ (Theorem \ref{limitset}), and the unique such one which is $\pm 1$ invariant. \\

The blue points in Figure 1 below live on the cylinder $X$ and represent     the points $\pm Z(\omega)$ where $Z(\omega)\in X$ and  $[Z(\omega)]$ is given by Theorem \ref{existence_uncite}.  Hence, the picture is a simulation of $\tilde{\nu}$. The points that are in the transparent face of the cylinder are exactly identical to the visible ones by symmetry. In Figure 2, we plot the projection of the   points of Figure 1 on $S^1$.
 By the discussion of Section \ref{geobehind}, Picture 2 is then a simulation of   a two-fold cover of the limit set of the i-p semigroup of $\textrm{SL}_2(\R)$ generated by $\pi(g_1)$ and $\pi({g_2})$, projection of $g_1$ and $g_2$ on $\GL(\R^3/L)$.  We  observe in Figure 1    the fibered structured described in Section \ref{geobehind}:  each fiber  is contained in an affine line with (horizontal) direction $L$ and projects to a point of Figure 2.

 \begin{center}
 
 \includegraphics[scale=0.65]{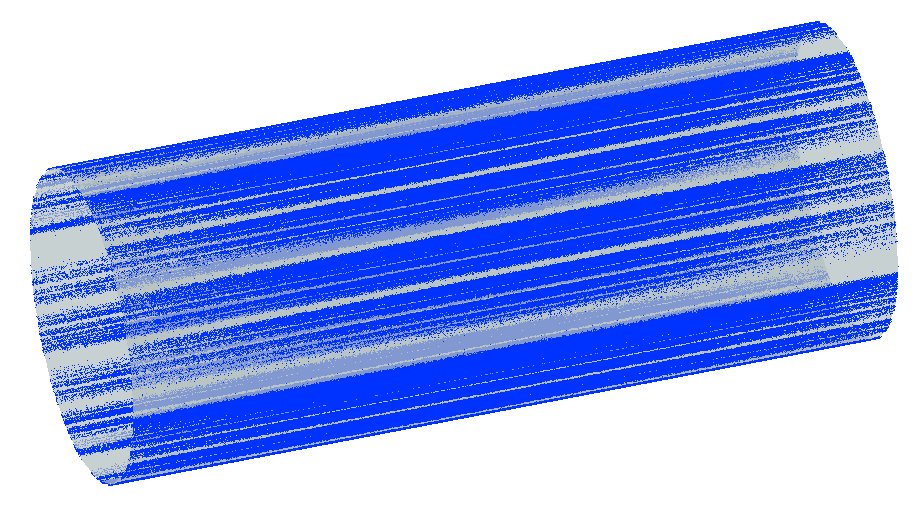}
\captionof{figure}{Simulation of $\tilde{\nu}$}

\end{center}

  \begin{center}
 \includegraphics[scale=0.27]{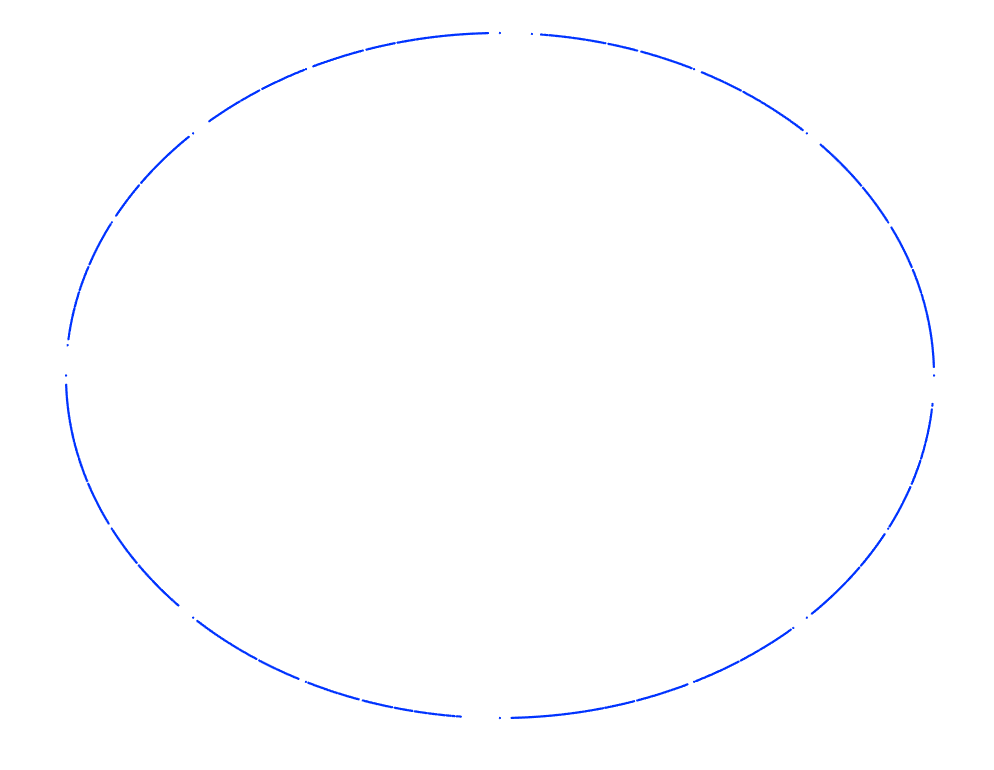}
\captionof{figure}{Projection on $S^1$}
\end{center}
   \noindent The support of $\nu$ is not compact in $O$ as suggested by Figure 1.
   Indeed, the unipotent structure of $g_1$ gives that 
 $g_1^n =  \left(\begin{matrix} 
 1 & \psi_1(n) & \psi_2(n) \\
 0 & 1 &  \psi_3(n)\\
 0 & 0 & 1
 \end{matrix}\right)$, 
 with $\psi_1(n) = O(n)$, $\psi_2(n)=O(n^2)$ and $\psi_3(n)=O(n)$. In particular, for every $[x]\in O=\textrm{P}^2(\R)\setminus [L]$, every limit point of $\left(g_1^n[x]\right)_{n\in \N}$ belongs to the point  at infinity $[L]$. Hence, the closure of the orbit of any point of $O$ under $T_{\mu}$ meets $[L]$. 
We conclude by item 3.~ of Corollary \ref{cor-limit}.  
  \paragraph{Example 2:   compact support in $O$}

We replace $g_1$  and $g_2$ of Example 1 with $$g_1=\left(\begin{matrix} 
 0.5 & 2 & 3 \\
 0 & 1 &  1\\
 0 & 0 & 1
 \end{matrix}\right)\,\,,\,\, 
 g_2=\left(\begin{matrix} 
0.5 & 1 & 2 \\
 0 & 0& -1\\
 0 &   1 & 0\end{matrix}\right)$$ 
instead of $g_1$ and $g_2$. \\

\noindent Again $\lambda_1>\lambda_2$, $\mathcal{L}_{\mu}=L$ and $\mathcal{U}_{\mu}=\R^3$ (same argument as above with 
$W'=\textrm{Span}\left((0,0,1), (2,-1,0)\right)$ instead and $g_1(0,0,1)^t = (3,1,1)\not\in W'$).  
  We obtain the following simulation of the unique $\pm 1$-stationary measure $\tilde{\nu}$ on the cylinder $X$.  
 \begin{center}
 \includegraphics[scale=0.7]{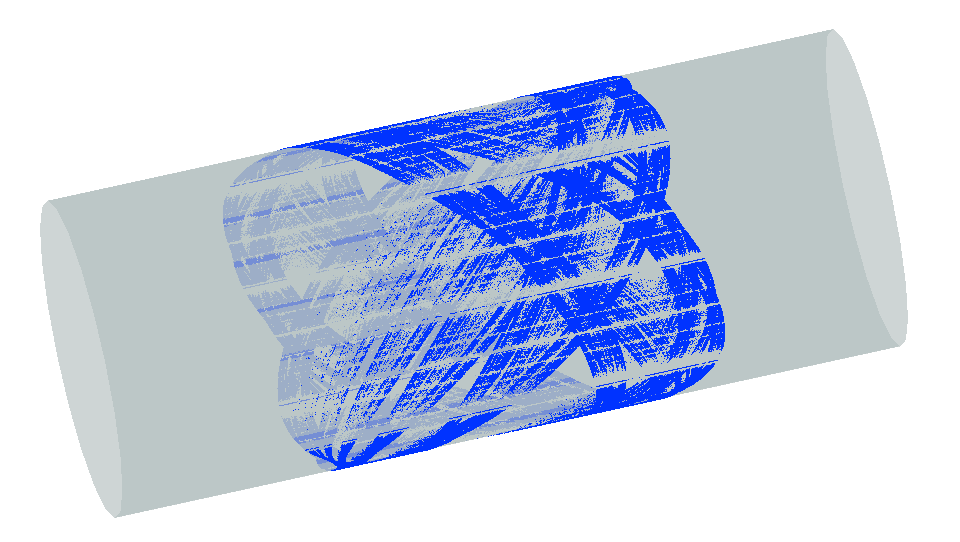}
 \captionof{figure}{Simulation of $\tilde{\nu}$}
\end{center}
Note that the projection of $\tilde{\nu}$ on $S^1$ is identical to the one given by Figure 2 above as the two semigroups of Example 1 and 2 have the same projection on $\GL(\R^3/L)$. 
 The support of $\nu$ is compact in $O$ as suggested by the picture.    Indeed, by Proposition \ref{compactness-criterion}, it is enough to check that 
for $C_1:= \left(\begin{matrix} 
 
 1 &  1\\
 0 & 1
 \end{matrix}\right)$ and $C_2:= \left(\begin{matrix} 
 
  0 &  -1\\
  1 & 0
 \end{matrix}\right)$, one has that 
\begin{equation}\mathfrak{R}(\{C_1^{-1}, C_2^{-1}\})<2.\label{takhodna}\end{equation}
   Consider now the the operator norm on $\mathcal{M}_2(\R)$ induced by the $L^2$ norm on $\R^2$. We have: 
   
   $$||C_1^{-1}||=||C_1||=\sqrt{\rho_{\textrm{spec}}(C_1 C_1^t)}  = \sqrt{\frac{3+\sqrt{5}}{2}}<2.$$
   Clearly, $||C_2||=1$. Hence $\max \{||C_1^{-1}||, ||C_2^{-1}||\}<2$. 
    But by definition of the joint spectral radius, we have $\mathfrak{R}(\Sigma)\leq \max \{||g||; g\in \Sigma\}$ for every bounded subset $\Sigma$ of $\mathcal{M}_d(\R)$ and for every matrix norm. Hence (\ref{takhodna}) is fulfilled and the support of $\nu$ is compact in $O$. 
    Observe that the freedom in the choice of the norm was crucial in the previous proof (as for instance neither the $L^1$ nor the $L^{\infty}$ norm would help fulfilling the criterion).

\paragraph{Example 3}
 Here we present another example for a behavior similar to the one of Example 1 in order to justify  item 4.~ of Remark \ref{remark_support}. 
Let $\mu$ be the uniform probability measure on the set $\{g_1, g_2, g_3\}$ with 
 $g_1= \begin{pmatrix} 
0.5 & 2 & 3 \\
0 &    4 & 1 \\
0 & 0 & 0.25
\end{pmatrix}$,   $g_2= \begin{pmatrix} 
0.5 & 1 &2 \\
0 &    0 & -1 \\
0 & 1 & 0
\end{pmatrix}$ and  $g_3= \begin{pmatrix} 
0.5 & 1 & 1 \\
0 &    0.25  & -1 \\
0 & 0 & 4
\end{pmatrix}$. We have also $\lambda_1>\lambda_2$, $\LL_{\mu}=L$ and $\mathcal{U}_{\mu}=\R^3$.    We will show that the support of the unique stationary measure on $O$  is not compact, although $\rho_{\textrm{spec}}(A)\leq \frac{1}{2} < 1 \leq \rho_{\textrm{spec}}(C)$ for every    $\begin{pmatrix} 
A & B \\
0 & C 
\end{pmatrix} \in T$. 
Denote by $C_i\in GL_2(\R)$ the projection of $g_i$ on $\R^3/L$ and observe that   $C_3=C_1^{-1}$.  Since $g_3$ is  a proximal element of $\GL_3(\R)$ that belongs to   $T$, we have  by Theorem \ref{limitset} that $p^+(g_3)\in \textrm{Supp}(\nu)$. We will check that  the orbit of $p^+(g_3)$ under the cyclic group generated by $g_1$ is not compact in $O$. This is enough to conclude as $\textrm{Supp}(\nu)$ is $T$-invariant.   
To do so, we use the identification 
$O \simeq X/\{\pm 1\}$ of Section \ref{geobehind} and write 
 $p^+(g_3)=\pm (t_0, \xi_0)$, with $t_0 = \frac{22} { 7\sqrt{241}}\in L\simeq \R$ and $\xi_0=\frac{1}{\sqrt{241}} (-4,15)\in S^1$ a 
 normalized eigenvector for the top eigenvalue $4$ of  $C_3=C_1^{-1}$. 
But $\xi_0$ is also a fixed point of $C_1$ for the natural action on $S^1$ as $\xi_0$ is an eigenvector of $C_1$    for its least eigenvalue $0.25>0$. 
Hence, by the 
 cocycle property shared by $\sigma$ (notation of Section \ref{geobehind}):

  $$\forall n\in \N, g_1^n(t_0, \xi_0) = \left(\sigma(g_1, \xi_0)^n(t_0)\,,\, \xi_0\right).$$
 
\noindent Now $\sigma(g_1, \xi_0)$ is the affine map of the real line $t \mapsto 8t+4 \langle (2,3), \xi_0\rangle$, 
with $\langle \cdot, \cdot \rangle$ the Euclidean inner product. Since its linear part  has absolute value $>1$ and since $t_0$ is not equal to its unique fixed point (by a direct computation), then       $\sigma(g_1, \xi_0)^n(t_0) \underset{n\rightarrow +\infty}{\longrightarrow} +\infty$ in $\R$. In particular, $g_1^n p^+(g_3) \underset{n\rightarrow +\infty}{\longrightarrow} [L]$ in $\textrm{P}^2(\R)$. This ends the proof. 
 
\section{Regularity of the stationary measure} 
\subsection{Introduction}
Let $V$ be a vector space over the local field ${\kk}$ of dimension $d\geq 2$. We use the same notation as Section \ref{fubinisection} concerning the choice of the norm on $V$  and the Fubini-Study metric on $\PV$. 
In this section, we prove that under an exponential moment of
$\mu$,   the stationary measure given by Theorem \ref{existence_uncite}  has
H\"older regularity, and more precisely the following

  \begin{theo}
Let $\mu$ be a probability measure on $\GL(V)$ with an
exponential moment such that $\lambda_1>\lambda_2$. Let $\nu$ be the unique $\mu$-stationary
probability measure on $\PV$ such that  $\nu([\mathcal{L}_{\mu}])=0$. Then there exists $\alpha>0$ such that, 
\begin{equation} 
\sup_{[f] \in \textrm{P}(V^*)\setminus [\LL_{\check{\mu}}]} {\delta\left([f], \PS \setminus [\mathcal{L}_{\check{\mu}}]\right)} \int_{\PV}{\delta^{-\alpha}([x],[\textrm{Ker}(f)])\,d\nu([x])}<+\infty \label{aprouver}\end{equation}
 \label{regularite} \end{theo} 
\vspace{0.5cm}

A crucial step is to show that  the random walk converges exponentially fast towards its stationary measure uniformly on compact subsets of $\PV\setminus [\LL_{\mu}]$, namely: 

\vspace{0.5cm}

 \begin{theo}
Let $\mu$ be a probability measure on $\GL(V)$ with an
exponential moment such that $\lambda_1>\lambda_2$. Let     $\nu$
be the unique $\mu$-stationary measure on $\PV$ such that
$\nu([\mathcal{L}_{\mu}])=0$.   Then, there exists a random variable $[Z]$ on  $\PV$ with law $\nu$, there exist $\beta>0$
and  $n_0\in \N^*$ such that for every  $n\geq n_0$ and every $[x]\in \PV\setminus [\LL_{\mu}]$,     
\begin{equation}  \E \left( \delta(R_n[x], [Z] )\right)\leq   \frac{\exp(-n \beta)}{\delta([x], [\mathcal{L}_{\mu}])}. \label{est1}\end{equation}
\label{direction1}\end{theo}

\begin{remarque}
The above statement is stronger than just saying that $\nu$ is a
$\mu$-boundary. Indeed, it implies that, for any probability
measure $\eta$ on $\PV$ that gives zero mass to $[\mathcal{L}_{\mu}]$,

$$\mu^n \star \eta \underset{n\rightarrow +\infty}{\overset{\textrm{weakly}}{\longrightarrow}} \nu.$$

\label{weakly}\end{remarque}

 \vspace{0.25cm}

 When $T_{\mu}$ is irreducible (or equivalently i-p), Theorem \ref{regularite} was shown  in \cite{Guivarch3} using the spectral gap property   \cite{lepage}.
Other alternative proofs were then proposed     \cite{aoun-comptage},  \cite{BQbook}. When $T_{\mu}$ is a non degenerate sub-semigroup   of the affine group of ${\kk}^d$, our result on Hausdorff dimension is new. Here are 
the main ingredients of the proof.\\

  A first step is Theorem \ref{direction1} above. It consists of showing that   $R_n[x]$ converges exponentially fast towards the stationary measure, with exponential speed and uniformly on compact subsets of $\PV\setminus[\mathcal{L}_{\mu}]$. In the i-p case, this is known (see \cite{bougerol} for the convergence and \cite{aoun-free} for the speed). For affine groups in the contracting setting, this is straightforward by direct computation. When $\lambda_1>\lambda_2$ and  $G_{\mu}$ is any group of upper triangular matrix blocs, such as a subgroup of the automorphism group of the Heisenberg group (see Section \ref{guiding}),   this result is new. \\

The second step is the deterministic Lemma \ref{distance_hyperplan}. This lemma will imply  that estimating the distance from $R_n[x]$ to a fixed hyperplane $H$ consists, with probability exponentially close to one, of establishing large deviation estimates of the ratio  of norms $\frac{||R_n^{t} f||}{||f||\,||R_n||}$ uniformly on $f\in V^*$. \\

 In both steps, we need large deviation inequalities for norms ratios. This is done using a classical cocycle lemma (see Lemma \ref{cocycle} below). Since we do not need the more delicate large deviation estimates for the norms themselves, we do not aim to give the optimal  formulation  (see Corollary \ref{rapport}). We refer to \cite{bqtcl} for related estimates for cocycles.\\

 In terms of techniques,  we note that even though our result applies to the interesting case $\mathcal{L}_{\mu}\neq \{0\}$ (as the contracting case in the context of affine groups),  our proof  uses heavily different passages through the easier case $\mathcal{L}_{\mu}=\{0\}$  (as the expansive case for affine groups or the irreducible groups)  via  group representations.  We refer to  Remark \ref{V/L} for more on this condition. 
     \subsection{Cocycles}\label{sec-cocyle}
We begin by recalling a cocycle lemma:  Lemma \ref{cocycle} below. The case a)  allows us to
obtain large deviations estimates of cocycles whose average is
negative. It is due to Le Page \cite{lepage} and was crucial in order to establish fine limit theorems for the norm of
matrices.  Case b) treats the case
where the average of the cocycle is zero and appears in
\cite{Guivarch3}, \cite{aoun-free}.
\begin{lemme}(Cocycle lemma)\cite{lepage,aoun-free}\label{cocycle}
Let $G$ be a semigroup acting on a space $X$, $s$ an additive
cocycle on  $G \times X$, $\mu$ a probability measure on $G$ such
that  there exists  $\tau>0$
satisfying:
\begin{equation}\label{condition}\E\left(\exp\left(\tau \,\sup_{x\in X}{|s(X_1,x)|}\right) \right)< \infty.\end{equation} Set $l=\underset
{n\rightarrow \infty}{\lim} \;\frac{1}{n}\sup_{x\in
X}\;{\E(s(R_n,x))}$.

\begin{enumerate}
\item[a)] If $l<0$, then there exist  $\lambda>0$, $r_0>0$,
$n_0\in \N^*$ such that for every $0<r<  r_0$ and
$n> n_0$:\; $\sup_{x\in X}\;\E\left(\exp
(r\, s(R_n,x) ) \right)\leq \exp(-n r \lambda)$.

   \item[b)]  If $l=0$, then for every $\gamma>0$, there exist $r(\gamma)>0$, $n
(\gamma)\in \N^*$ such that for every
$0<r<r(\gamma)$ and  $n>n(\gamma)$: $\sup_{x\in X}\;\E\left(\exp
(r\, s(R_n,x) ) \right)  \leq \exp(n r \gamma)
$.\end{enumerate}
\end{lemme}

\begin{corollaire}(Controlling ratio norms)
Let $\mu$ be a probability measure on $\GL(V)$ such that $\mu$ has an exponential moment and 
$\lambda_1>\lambda_2$.   Then, for every $\epsilon>0$,
there exist $\beta=\beta(\epsilon)>0$, $n_0=n(\epsilon)\in \N^*$ such
that for every $n\geq n_0$ and   every $[x]\in  \PV\setminus [\mathcal{L}_{\mu}]$,

\begin{equation}
 {\p \left( \frac{||L_n x||}{||L_n||\,||x||} \leq
\exp(-n\epsilon) \right)}\leq   
\frac{\exp(-n \beta)}{\delta([x], [\mathcal{L}_{\mu}])}.\label{ratio-sous}\end{equation}
\label{rapport}\end{corollaire}

\begin{proof}
Endow $V/\LL_{\mu}$ with the quotient norm. Let $(e_1, \cdots, e_d)$ be an orthonormal    basis of $V$. For every $x\in V$, denote by $\overline{x}$ its projection on
the quotient vector space $V/\mathcal{L}_{\mu}$. Let $\pi$ be the morphism
action of $G_{\mu}$ on $V/\mathcal{L}_{\mu}$. Recall that  
$\delta([x],[\mathcal{L}_{\mu}])=\frac{||\overline{x}||}{||x||}$  (Lemma \ref{dist_droite_sev})   and $||gx||\geq ||\pi(g)\overline{x}||$ for every $g\in \GL(V)$ and $x\in V\setminus \{0\}$. 
Fix now $[x]\in \PV\setminus [\LL_{\mu}]$ and    $\epsilon>0$. Then for every $r>0$ and every $n\in \N^*$, 
\begin{eqnarray}
 \p \left( \frac{||L_n x||}{||L_n||\,||x||} \leq \exp(-n\epsilon) \right)& = & \p \left[ \left(\frac{||L_n  ||\,||x||}{||L_n  x ||}\right)^{r} \geq 
 \exp(n\epsilon r) \right]\nonumber\\
 & \leq & \exp(-n\epsilon r) \,\E \left[ \left(\frac{||L_n||\,||{x}||}{||L_n {x}||}\right)^{r} \right]\nonumber\\
   & \leq &  \frac{ \exp(-n\epsilon r)}{ \delta^r([x],[\mathcal{L}_{\mu}]) }\E \left[ \left(\frac{||L_n||\,||\overline{x}||}{||\pi(L_n) \overline{x}||}\right)^{r} \right]\nonumber.\end{eqnarray}
   Since $||g||\leq d \max\{||g e_i||; i=1, \cdots, d\}$ for every $g\in \GL(V)$ and since the expectation of the maximum of $d$ random real variables is less than $d$ times the maximum of the expectations, we get that:

\begin{eqnarray}
 \p \left( \frac{||L_n x||}{||L_n||\,||x||} \leq \exp(-n\epsilon) \right)&  \leq &   \frac{d^{r+1} \exp(-n \epsilon r)}{\delta^r([x],[\mathcal{L}_{\mu}])}\max_{i=1, \cdots , d}{ \E \left[ \left(\frac{||L_n e_i||\,||\overline{x}||}{||\pi(L_n) \overline{x}||}\right)^{r} \right]}\nonumber\\
      & \leq &   \frac{d^{r+1} \exp(-n \epsilon r)}{\delta^r([x],[\mathcal{L}_{\mu}])} \sup_{z \in \textrm{P}(V/\mathcal{L}_{\mu}) \times \PV} {\E \left[\exp \left(r s (R_n, z)\right)\right]}, 
      \label{ineq}\end{eqnarray}

\noindent where $s$ is the function defined on $G_{\mu}\times \left(\textrm{P}(V/\mathcal{L}_{\mu})\times \PV\right)$  by $$s\left( g,([\overline{x}],[y])\right):=\log
 \frac{||gy|| |\, ||\overline{x}||}{||\pi(g)\overline{x}||\,||y||}.$$  
 Now let $G_{\mu}$ act naturally on the product space
$Z:=\textrm{P}(V/\mathcal{L}_{\mu})\times \PV$. It is  immediate to see that
 $s: G_{\mu}\times Z \longrightarrow \R$ is a cocycle. Since $\mu$ has an exponential moment, then
 condition
 \eqref{condition} of Lemma \ref{cocycle} is satisfied.
With the notations of the aforementioned lemma, let us show that
$l=0$.
 Since $ \mathcal{L}_{\pi(\mu)}=\{0\}$ (see Remark \ref{V/L}) and $\lambda_1\left(\pi(\mu)\right)=\lambda_1(\mu)=\lambda_1$, then Corollary
 \ref{corollaire-unicite} (and Remark \ref{liapou_remarque} part 2.) shows that:
  $$\inf_{[\overline{x}]\in \textrm{P}(V/\mathcal{L}_{\mu})}
 \frac{1}{n} \E\left(\log\frac{||\pi(L_n)\overline{x}||}{||\overline{x}||}\right)
\underset{n\rightarrow +\infty}{\longrightarrow} \lambda_1.$$ Moreover, by Remark \ref{liapou_remarque} part 1.,     $\underset{[y]\in
\PV}{\sup}\frac{1}{n}\E\left(\log\frac{||L_ny||}{||y||}\right) \underset{n\rightarrow +\infty}{\longrightarrow}  \lambda_1$.
Hence,
$$l:=\underset {n\rightarrow \infty}{\lim} \;\frac{1}{n}\sup_{z\in
Z}\;{\E(s(R_n,z))} = 0.$$
  Applying the cocycle lemma for $\gamma:=\epsilon/2$ gives some $r=r(\epsilon)>0$, $n=n(\epsilon)\in \N^*$ such that for every   $0<r<r(\epsilon)$ and every 
$n>n(\epsilon)$,

\begin{equation} \sup_{z \in \textrm{P}(V/\mathcal{L}_{\mu}) \times \PV} {\E \left[\exp \left(r s (R_n, z)\right)\right]} \leq \exp(n r \epsilon/2).\label{interm}\end{equation}
 
\noindent Without loss of generality, one can assume $0<r<1$. Since the Fubini-Study metric is bounded by one, we obtain the desired estimate by combining   \eqref{interm} and \eqref{ineq}.

\end{proof}
 
 \subsection{Exponential convergence in direction}
 In this section, we prove Theorem \ref{direction1} stated above. 
  \begin{proof}
\textbf{Step 1}: First, we check  that it is enough to show the
following statement: there exists $\beta>0$, $n_0\in \N^*$ such that
for every $[x]\in \PV \setminus [\mathcal{L}_{\mu}]$, and every
$n\geq n_0$,
\begin{equation} {\E \left( \delta(R_n[x], R_{n+1}[x] )\right)\leq \frac{\exp(-n \beta)}{\delta([x], [\mathcal{L}_{\mu}])}}. \label{est2} \end{equation}
Indeed  \eqref{est2} would imply that for every $x\not\in \mathcal{L}_{\mu}$,  $(R_n[x])_{n\in \N^*}$ is almost surely a Cauchy sequence in the complete space $\PV$. Hence, it converges to a random variable $[Z_x]\in \PV$. By item 3.~ of Corollary \ref{corollaire-unicite}, $[Z_x]=[Z]$ is almost surely independent of $x$ and has law $\nu$. Now  \eqref{est1} would follow immediately from  \eqref{est2} by applying Fatou's lemma and the triangular inequality. \\

  \noindent\textbf{Step 2}: 
\noindent  
 Next, we give an upper bound of the left
 side of estimate  \eqref{est2}. We denote by $\pi:
G_{\mu}\longrightarrow \GL(V/\mathcal{L}_{\mu})$ the morphism of the projection of
$G_{\mu}$ onto $V/\mathcal{L}_{\mu}$.  Let $[x]\in \PV \setminus [\mathcal{L}_{\mu}]$. 
For every
$n\in \N^*$, the following almost sure estimates hold: 
  \begin{eqnarray}
   \delta(R_n[x], R_{n+1}[x] )&=& \delta (R_n[x], R_n X_{n+1}[x])\nonumber\\
   &= &\frac{||\bigwedge^2 R_n (x \wedge X_{n+1} x)||}{||R_n x||\, ||R_n X_{n+1} x||}\nonumber\\
   &\leq &  \frac{||\bigwedge^2 R_n (x \wedge X_{n+1} x)||}{||\pi(R_n) \overline{x}|| \, ||\pi(R_n)  \pi( X_{n+1})  \overline{x}||}\label{est3}\end{eqnarray}
We let $G_{\mu}$ act naturally on  
$Z:=\textrm{P}(\bigwedge^2 V) \times \textrm{P}(V/\mathcal{L}_{\mu})^2$ and set, for every
$z=([a \wedge b], [\overline{c}], [\overline{d}])\in Z$ and every $g\in
G_{\mu}$,

 $$s(g,z):= \log \frac{||\bigwedge^2 g (a \wedge b)||\,
||\overline{c}||\,||\overline{d}||}{||a\wedge b||\,
||\pi(g)\overline{c}||\, ||\pi(g)\overline{d}||}.$$ 
Hence if $Y_n$ denotes the following random variable in $Z$, $Y_n:= \left([x \wedge X_{n+1} x], [\overline{x}], [\pi(X_{n+1})\overline{x}] \right)$, \eqref{est3} becomes, 

\begin{equation}\delta(R_n[x], R_{n+1}[x] ) \leq \exp\left(s(R_n, Y_n)\right)\,\times\, \frac{||x \wedge X_{n+1} x||}{||\overline{x}||\,||\pi(X_{n+1}) \overline{x}||}.\label{est3b}\end{equation}
\noindent By combining
 \eqref{est3b}, the equality $\delta([x], \mathcal{L}_{\mu}) = \frac{
||\overline{x}||}{||x||}$ and
the inequalities $||x \wedge y||\leq ||x||\,||y||$, $||gx||\geq
\frac{||x||}{||g^{-1}||}$, $||\pi(g)||\leq ||g||$ true for every
$x,y\in V\setminus \{0\}$ and $g\in G_{\mu}$, we obtain the following almost sure inequality ($[x]$ is always fixed): 
\begin{equation}
 {  \delta(R_n[x], R_{n+1}[x] ) }\leq \frac{1}{\delta([x], [\mathcal{L}_{\mu}])^2}
||X_{n+1}||\,||X_{n+1}^{-1}||\,    {\exp\left( s (R_n,
Y_n)\right)}.  \label{est4} \end{equation} 

\noindent Using Cauchy-Schwarz inequality and the   fact that $R_n=X_1 \cdots X_n$ and $Y_n$ are independent random variables, we deduce that for every $\alpha>0$ (to be chosen in Step 3 below), 

\begin{equation}
\E\left(  {  \delta^{ {\alpha}/{2}}(R_n[x], R_{n+1}[x] ) }\right)\leq \frac{1}{\delta([x],[\mathcal{L}_{\mu}])^{\alpha}}\,
\sqrt{\E\left(||X_{n+1}||^{ \alpha}\,||X_{n+1}^{-1}||^{ \alpha}\right)}\,   \sqrt{\underset{z\in Z}{\sup}\, \E\left({\exp\left(  \alpha s (R_n,
z)\right)}\right)}.  \label{est5} \end{equation} 

\noindent \textbf{Step 3}: Finally,
we check that we are in the case b) of  the cocycle lemma (Lemma
\ref{cocycle}). The   map $s: G_{\mu}\times Z \longrightarrow
\R$ is clearly a cocycle on $G_{\mu}\times Z$. Since $\mu$ has an
exponential moment, condition \eqref{condition} is fulfilled.
Moreover the representation $\pi$ satisfies $ \mathcal{L}_{\pi(\mu)}=0$. Hence, by
Corollary \ref{corollaire-unicite},  $\inf_{[\overline{c}]\in
\textrm{P}(V/\mathcal{L}_{\mu})}\, {\frac{||\pi(R_n)
\overline{c}||}{||\overline{c}||}} \underset{n\rightarrow +
\infty}{\longrightarrow} \lambda_1$. Consequently, $l \leq
\lambda_2-\lambda_1<0$. The cocycle lemma gives then $\alpha_1>0$ and $\beta>0$ such
that for every $\alpha\in [0,\alpha_1)$ and every large $n$,
\begin{equation}
\sup_{z\in Z} {\E \left(  \exp\left( \alpha s (R_n, z)\right)
\right)} \leq \exp(-\beta n).\label{est6}
\end{equation}
Since $\mu$ has an exponential moment, there exists $\alpha_2>0$
such that for every $\alpha\in [0,\alpha_2)$, $\E \left(
||X_{n+1}||^{\alpha}\,||X_{n+1}^{-1}||^{\alpha}\right)<+\infty$.
Apply now  \eqref{est5} for  $\alpha=\min\{\alpha_1,\alpha_2,1\}$. Since  the  Fubini-Study metric $\delta$ is bounded  by one,  we obtain   the desired estimate  \eqref{est2}. Theorem \ref{direction1} is then proved.

  \end{proof}

\subsection{Proof of the regularity of the stationary measure}
We begin with the following deterministic lemma. 
\begin{lemme}
Let ${\kk}$ be a local field, $V$ a vector space over ${\kk}$ of dimension $d\geq 2$ endowed with the norm described in 
Section \ref{fubinisection}, $L$ a subspace of $V$   and $F$   an orthonormal basis of an orthogonal supplement to $L$ in $V$  (see Section \ref{preliminaries} when $\kk$ is non-Archimedean). 
 Let $C({\kk},d)=\frac{1}{\sqrt{d+1}}$ if ${\kk}$ is Archimedean and $C({\kk},d)=1$ otherwise. Then for any $g\in \GL(V)$ such that $g(L)=L$ and for any $f\in V^*\setminus\{0\}$,   there exists $x\in F$ such that:

$$\delta\left(g[x],[\mathrm{Ker}(f)]\right) \geq C  \frac{||g^t
f||}{||g^t||\,||f||} \,\mathds{1}_{C  \frac{||g^t
f||}{||g^t||\,||f||} > \frac{||g_{| L}||}{||g||}}.$$
  \label{distance_hyperplan}\end{lemme}

\begin{proof}
Let $d'=\dim(L)$,  $L^{\perp}$ an orthogonal of $L$ in $V$,   $B=\{e_1, \cdots, e_d\}$ an orthonormal basis of $V$ such that $B':=(e_1, \cdots, e_{d'})$ is a basis of
$L$ and $F=B\setminus B'$ a basis for  $L^{\perp}$. 
 Assume first that ${\kk}$ is non-Archimedean. Then
  the following relation is true for every $g\in \GL(V)$,

\begin{eqnarray}
\frac{ ||g^t f ||}{||g^t||\,||f||} &=& \max_{1\leq i\leq d}{\frac{|f(ge_i)|}{||g^t|| \,||f||}}\nonumber\\
&= & \max\Big\{ \max_{1\leq i \leq d'}{\frac{|f(ge_i)
 |}{||g|| \,||f||}}\,,\, \max_{d'+1\leq i \leq
 d}{\frac{|f(ge_i)|}{||g|| \,||f||} }\Big\}\label{note1}\\
&\leq & \max \Big\{ \frac{||g_{| L}||}{||g||}\,,\,
\max_{d'+1\leq i \leq
 d}{\delta\left(g[e_i],[\textrm{Ker}(f)]\right)}\Big\}\label{note2}
\end{eqnarray}

\noindent Equality \eqref{note1}  holds because  $||g^t||=||g||$
and inequality  \eqref{note2} is true because  for any $x\in V$,
$\frac{|f(x)|}{||f||}\leq ||x||$ and
$$\delta\left(g [x],[\textrm{Ker}(f)]\right)=
\frac{|f(gx)|}{||gx||\,||f||}\geq \frac{|f(gx)|}{||g||\,||f||\,||x||}.$$
Estimate  \eqref{note2}  shows that the lemma is true for $C=1$.
When  ${\kk}$ is  Archimedean, estimate  \eqref{note2} is replaced by:

$$\left(\frac{ ||g^t f ||}{||g^t||\,||f||} \right)^2 \leq
\left(\frac{||g_{| L}||}{||g||} \right)^2+
\sum_{i=d'+1}^d{\delta\left(g[e_i],[\textrm{Ker}(f)]\right)^2}.$$
Hence, for any $C<1$,   $$C \frac{||g^t f||}{||g^t||\,||f||}\geq   \frac{||g_{|L}||}{||g||} \,\Longrightarrow\, \underset{d'\leq i\leq d}{\max}{\delta\left(g[e_i], [\mathrm{Ker}(f)]\right)}\geq \sqrt{\frac{1-C ^{2}}{ d}} \frac{||g^t f||}{||g^t||\,||f||}.$$
 The constant $C({\kk},d):=\frac{1}{\sqrt{d+1
}}$ solves the equation  $C=\sqrt{\frac{1-C^{2}}{d}}$. \end{proof}

\begin{proof}[Proof of Theorem \ref{regularite}:]

      Let $[Z]\in \PV$ be the random variable given by Theorem \ref{direction1}. 
 Let $ f \in V^* \setminus \mathcal{L}_{\check{\mu}}$ and $H=\mathrm{Ker}(f)$.   
Since the Lyapunov exponent of the restriction to $\mathcal{L}_{\mu}$ is less than $\lambda_1$, one can show using the same techniques as the proof of   Corollary \ref{rapport}  that   there exists $\beta_1>0$ such that  $\frac{||{R_n}_{| L}||}{||R_n||} \leq \exp(-\beta_1 n)$, with probability tending to one exponentially fast.  
  Corollary \ref{rapport} applied to the
measure $\check{\mu}$, together with $\delta\leq 1$, show   then that  for any $C>0$ there exists $\beta_2>0$ and $n_0\in \N$ (both independent of $f$) such that for all $n\geq n_0$, 
\begin{equation}\p\left(  C \frac{||R_n^t f||}{||R_n^t||\,||f||} \geq
\frac{||{R_n}_{| L}||}{||R_n||}\right)\geq 1-\frac{\exp(-\beta_2 n)}{\delta([f], [\mathcal{L}_{\check{\mu}}])}.\label{oulbaa}\end{equation}
Take now $C$ to be the constant $C({\kk},d)$ given by Lemma  \ref{distance_hyperplan}. The aforementioned lemma 
  together with estimate  \eqref{oulbaa} imply that   for   every $n\geq n_0$: 

$$\p\left(\exists x\in F;
\delta\left(R_n[x],[H]\right) \geq C \frac{||R_n^t
f||}{||R_n^t||\,||f||}\right)\geq 1 - \frac{\exp(-\beta_2 n)}{\delta([f], [\mathcal{L}_{\check{\mu}}])}.$$

\noindent Hence, for every $\epsilon>0$,

$$
\p\left(\delta\left([Z], [H]\right)\leq \exp(-\epsilon n)\right)  \leq \frac{\exp(-\beta_2 n)}{\delta([f], [\mathcal{L}_{\check{\mu}}])} +
\sum_{x\in F} {\p\left( C \frac{||R_n^t f||}{||R_n^t||\,||f||} \leq
\exp(-\epsilon n)+ \delta(R_n[x],[Z])\right)}.$$

\noindent 
But by Theorem \ref{direction1} and   Markov's inequality,  one deduces that  there exist  $\beta_3=\beta_3(\epsilon),\beta_4=\beta_4(\epsilon)>0$ such
that for all $n$ large enough, 

$$
\p\left(\delta\left([Z],[H]\right)\leq \exp(-\epsilon n) \right)  \leq \frac {\exp(-\beta_3 n) }{\delta([f], [\mathcal{L}_{\check{\mu}}])}+
\sum_{x\in F} {\p\left( \frac{||R_n^t f||}{||R_n^t||\,||f||} \leq
\exp(-\beta_4 n)\right)}.$$

\noindent Using Corollary \ref{rapport} and the fact that $[Z]$ has law $\nu$, we deduce finally that for every $\epsilon>0$ there exists $\beta=\beta(\epsilon)>0$ and $n_0=n_0(\epsilon)\in \N$  (both independent of $f\in V^*\setminus [\LL_{\check{\mu}}]$) such that for every $n\geq n_0$ , 
 
 \begin{equation}\nu\left\{[x]\in \PV; \delta([x],[H]) \leq \exp(-\epsilon n) \right\} \leq  \frac{\exp(-\beta n)}{ {\delta([f], [\mathcal{L}_{\check{\mu}}])}}.\label{energy}\end{equation}
  \noindent Let now $A_n:=\{[x]\in \PV; \delta([x],[H])\in ( e^{-(n+1)}, e^{-n }] \}$, $n\in \N$. On the one hand, $(A_n)_{n\in \N}$ cover $\PV\setminus [H]$. 
On the other hand, $\nu( [H])=\nu(  [H] \cap [\mathcal{U}_{\mu}])=0$ because $\nu$ is not degenerate on $[\mathcal{U}_{\mu}]$ (Theorem \ref{existence_uncite}) and  $H\not\subset \mathcal{U}_{\mu}$ (as   $f\not\in \LL_{\check{\mu}}$, see Lemma \ref{duality}). Estimate  \eqref{energy} applied for $\epsilon =1$ gives then 
some $\beta>0$ and some $n_0\in \N$ (both independent on $f$)
 such that for any $\alpha>0$, 
\begin{eqnarray}
\int_{\PV}{\delta^{-\alpha}([x],[H])\,d\nu([x]) }&=&  \sum_{n=0}^{n_0-1}{\int_{A_{n}}\delta^{-\alpha}([x],[H])\,d\nu([x]) }+\sum_{n=n_0}^{+\infty}{\int_{A_{n}}{\delta^{-\alpha}([x],[H])\,d\nu([x]) }}\nonumber\\
&\leq & n_0 \exp(\alpha n_0) +\frac{1}{\delta([f], [\mathcal{L}_{\check{\mu}}])} \sum_{n=n_0}^{+\infty}{\exp( \alpha (n+1)  ) \exp(-  \beta n )}\nonumber\\
&\leq & \frac{1}{\delta([f], [\mathcal{L}_{\check{\mu}}])}
\left( n_0 \exp(\alpha n_0 ) +\exp(\alpha ) \sum_{n=n_0}^{+\infty} {\exp\left(- (\beta  -\alpha) n  \right)}\right)\nonumber
\end{eqnarray}
Hence ${{\delta([f], [\mathcal{L}_{\check{\mu}}])}}\int_{\PV}{\delta^{-\alpha}([x],[H])\,d\nu([x]) }$ is finite (and independent of $[f]\in \PS\setminus [\mathcal{L}_{\check{\mu}}]$) as soon 
  as $0<\alpha< \beta$. 
\end{proof}

\vspace{1cm}

\noindent Finally, we show how    to conclude easily from Theorem \ref{regularite}  the proof of some results stated in Section \ref{section_regularity}. \\
Proposition \ref{hitting}  concerning the exponential decay of the probability of hitting a hyperplane follows immediately from Theorem \ref{regularite} and Theorem \ref{direction1} proved above.  \\
Also, Corollary \ref{dimaffine} concerning the positivity of the Hausdorff dimension of the unique stationary measure in the affine space in the context of affine groups in the contracting case follows also from Theorem \ref{regularite} and  Example 2 of Section \ref{example_guiding}.

\end{document}